\documentclass[11pt, final]{article}
\usepackage{a4}
\usepackage{amsmath}%
\usepackage{amstext}%
\usepackage{amssymb}%
\usepackage{amsthm}%
\usepackage{amsxtra}
\usepackage{graphicx}
\usepackage{showkeys}%
\usepackage{epsfig}%
\usepackage{cite}
\usepackage{tikz}
\usepackage{caption}
\usepackage{float}
\usepackage{longtable}
\usepackage{threeparttable,booktabs,multirow,lscape}
\usepackage{lipsum}
\usepackage{rotating}
\usepackage{multirow}
\usepackage{booktabs}
\usepackage{amsmath}
\usepackage{amssymb}
\usepackage{graphicx}
\usepackage{graphicx}
\usepackage{subfigure}
\usepackage{listings}
\usepackage[ruled,vlined]{algorithm2e}

\usepackage{soul}
\usepackage{xcolor}
\usepackage[pdftex,dvipsnames]{xcolor}  %
\definecolor{mydarkblue}{rgb}{0, 0, 0.5}
\usepackage{hyperref}
\hypersetup{
    colorlinks=true,
    linkcolor=mydarkblue,  
    citecolor=mydarkblue,  
    filecolor=magenta,      
    urlcolor=cyan,
}
\usepackage{cleveref}

\newtheorem{theorem}{Theorem}[section]

\newtheorem{assumption}[theorem]{Assumption}

\newtheorem{fact}[theorem]{Fact}

\newtheorem{corollary}[theorem]{Corollary}

\newtheorem{lemma}[theorem]{Lemma}

\newtheorem{proposition}[theorem]{Proposition}
\newtheorem{rem}[theorem]{Remark}
\newenvironment{remark}{\rem\rm}{\endrem}

\newcounter{unnumber}

\newcommand{\zv}{\mathbf{z}}%
\newcommand{\xv}{\mathbf{x}}%
\newcommand{\yv}{\mathbf{y}}%
\newcommand{\uv}{\mathbf{u}}%
\newcommand{\bv}{\mathbf{b}}%
\newcommand{\cv}{\mathbf{c}}%
\newcommand{\gv}{\mathbf{g}}%
\newcommand{\R}{\mathbb{R}}%
\newcommand{\N}{\mathbb{N}}%

\DeclareMathOperator*\fix{Fix}

\newcommand{\FDSM}{\hyperref[alg:FDSM]{FDSM}}

\newcommand{\FDASM}{\hyperref[alg:FDASM]{FDASM}}

\newcommand{\DFDSM}{\hyperref[alg:DFDSM]{DFDSM}}

\title{Fixed-Point Delayed Subgradient Methods for Nonsmooth Convex Optimization Problems}

\sethlcolor{yellow} 

\usepackage[dvipsnames,svgnames]{xcolor}
\definecolor{darkblue}{rgb}{0,0,0.5}

\author{Ontima Pankoon\thanks{Department of Mathematics, Faculty of Science, Khon Kaen University, Khon Kaen 40002, Thailand, email: ontimapa@kkumail.com},
	\and	
	Nimit Nimana\thanks{Department of Mathematics, Faculty of Science, Khon Kaen University, Khon Kaen 40002, Thailand,
		email: nimitni@kku.ac.th},
        \and
        Yeol Je Cho\thanks{Department of Mathematics Education, Gyeongsang National University, Jinju 52828, Korea, and
              Center for General Education, China Medical University, Taichung, 40402, Taiwan,
              email: yjchomath@gmail.com}
}

\begin{document}
	\maketitle
	\begin{abstract}
 In this paper,   we consider the nonsmooth convex optimization problems over the fixed point constraint sets of firmly nonexpansive operators. To find an optimal solution of the problem, we present an iterative method based on the hybrid steepest descent method and the idea of a delayed subgradient scheme in which allows the use of staled subgradients from the earlier iteration when updating the next iteration.
We start the convergence part by deriving an upper bound for the difference of the best-achieved function values and the optimal value. After that, to ensure the convergence in iterations, we prove that there exists a subsequence of the generated sequence by the proposed method which converges to an optimal solution. Moreover, we subsequently show that the whole generated sequence converges to an optimal solution when the strict convexity of the objective function is imposed. Further, we extend the presented results to the inexact version and the centralized network system consisting of a finite number of workers and a central server.
Finally, we apply the proposed method to image inpainting problems.
The numerical results describe the effect of delay in many cases of objective functions.

\vspace{0.1cm}
	
\textbf{Key words:} 
	Convex Optimization, Delay, Firmly Nonexpansive, Fixed Point, Nonsmooth Optimization, Subgradient Method.
    
\textbf{MSC Classification:} 47H10;  47J25; 47N10; 90C25.
	\end{abstract}


\section{Introduction}

\quad\, In this paper, we focus on solving of the {\it minimization problem} of a convex function over a fixed-point constraint  of the following form:
\begin{align}\label{main-pb}
	\begin{array}{ll}
		\textrm{minimize }\indent f(\xv)\\
		\textrm{subject to}\indent \xv\in \displaystyle\fix(T),
	\end{array}%
\end{align}	
where the objective function $f:\mathbb{R}^d \to \R$ is convex and the constraint operator $T:\mathbb{R}^d \to \mathbb{R}^d$ is firmly nonexpansive with    $\displaystyle\fix(T):=\{\xv\in\mathbb{R}^d:T\xv=\xv\} \neq \emptyset$. For simplicity of notation, we denote the optimal value and the solution set of the problem (\ref{main-pb}) by $f^*$ and $\mathcal{S}$, respectively.

The {\it convex minimization problem} with a fixed-point constraint (\ref{main-pb}) not only is a key tool for solving minimization problems whose the constraint set is the intersection of finitely many closed convex sets $C_j\subset \R^d$ for each $j=1,\ldots,m$:
\begin{align*}
	\begin{array}{ll}
		\textrm{minimize }\indent f(\xv)\\
		\textrm{subject to}\indent \xv\in \displaystyle C:=\bigcap_{j=1}^mC_j,
	\end{array}%
\end{align*}
by chosing an appropriate nonlinear operator $T$ such that $\fix(T) = C$ \cite{C-12,I-15,I-16,Y-01,YOS-02}, but it also covers many practical bilevel optimization problems in which a convex function is minimized over a solution of another nonlinear problem, such as the  solution set of equilibrium problem \cite{BO-94,CH-05}, the set of zeros of  a sum of monotone operators, the solution set of a variational inequality problem,  and the solution set of a convex minimization problem (see \cite[Subsection 26.1]{BC-11} and \cite[Section 5]{N-20} for further discussion). Moreover, the problem (\ref{main-pb}) also includes some practical applications in support vector machine learning \cite{PPPN-23}, and image inpainting as shown in Section \ref{Image Inpainting Problems}. 

In the case when the objective function $f$ is $\eta$-strongly convex and $L$-smooth, the simplest method that can solve problem (\ref{main-pb}) is the celebrated {\it projected gradient method} \cite{E-85} which is defined by
\begin{equation}\label{PGM}
    \begin{cases}
    \xv_0 \in C \text{ is arbitrarily given, } \\
    \xv_{n+1}= P_{C}(\xv_n-\mu \nabla f (\xv_n)),
    \end{cases}   
\end{equation}
where $\mu \in \left( 0,  \frac{2\eta}{L^2} \right)$ and $\nabla f$ is the gradient of $f$.

It is well known that the sequence $\{\xv_n\}_{n=0}^\infty$ generated by \eqref{PGM} converges to the unique solution of the minimization problem. However, a limitation of this method is that it requires a closed-form expression of the metric projection onto the constrained set $C$, which can be very difficult or even impossible to deal with.

In 2001, Yamada \cite{Y-01} introduced the {\it hybrid steepest descent method} (HSDM, for short)  which uses the nonexpansive operator $T$ with $\fix(T)=C$ instead of the projection $P_C$ directly. This method is defined by
\begin{equation}\label{Yamada2001}
    \begin{cases}
    \xv_0 \in \mathbb{R}^d \text{ is arbitrarily given, } \\
    \xv_{n+1}= T\xv_n -\mu \alpha_n \nabla f(T\xv_n),
    \end{cases}   
\end{equation}
where $\mu \in \left( 0,  \frac{2\eta}{L^2} \right)$, and $ \{\alpha_n \}_{n=0}^\infty \subset (0,1]$.
The convergence of the HSDM to the unique solution of the problem is guaranteed when the step size $\{\alpha_n \}_{n=0}^\infty$ is diminishing and satisfies 
$\lim_{n\to \infty}(\alpha_n -\alpha_{n+1})/\alpha_n^2=0.
$

After the starting point of the HSDM, many authors have investigated generalizations and applications of the HSDM. For example, to name but a few, Yamada and Ogura \cite{YO-05} extended the HSDM to a wider class of operators, the so-called {\it quasi-shrinking operator}.  Cegielski \cite{C-15} proposed a {\it generalized hybrid steepest descent method} by using a sequence of quasi-nonexpansive operators. For an  annotated bibliography of the HSDM and its  developments in the case where the objective function is smooth and strongly convex and the considered constraint sets are fixed-point sets of nonlinear operators, we refer the reader to the recent work of Prangprakhon and Nimana \cite{PN-23}.


In the nonsmooth setting, in 2016, Iiduka \cite{I-16} considered the {\it nonsmooth convex optimization problem} with a fixed-point constraint, motivated by a networked system with finitely many workers in which each worker tries to minimize its own objective function over its own constraint set.  The problem can be formulated formally as minimizing the sum of nonsmooth convex objectives over the intersection of fixed-point sets:
\begin{align}\label{main pb dist}
	\begin{array}{ll}
		\textrm{minimize }\indent f(\xv) :=\displaystyle\sum_{j=1}^m f_j(\xv)\\
		\textrm{subject to}\indent \xv\in \displaystyle \bigcap_{j=1}^m \fix(T_j),
	\end{array}%
\end{align}
where $f_j:\mathbb{R}^d \to \mathbb{R}$ is a nonsmooth convex function and $T_j:\mathbb{R}^d \to \mathbb{R}^d$ is a certain quasi-nonexpansive operator for each $j=1, \ldots, m$. 

In order to solve the problem (\ref{main pb dist}), Iiduka proposed the {\it parallel subgradient method} which is given by
\begin{equation}\label{Iduka2016}
    \begin{cases}
    \xv_0 \in \mathbb{R}^d \text{ is arbitrarily chosen, } \\
    \xv_{n,j} = T_j \xv_n -\alpha_n \widetilde{\nabla}f_j(T_j \xv_n), \\
    \xv_{n+1} = \displaystyle \frac{1}{m}\sum_{j=1}^m \xv_{n,j}.
    \end{cases}   
\end{equation}

Under the assumptions that the step size 
$\{\alpha_{n}\}_{n=0}^\infty$ is diminishing and the sequences  $\{\xv_{n,j}\}_{n=0}^\infty$ are bounded for each $j=1, \ldots, m$, 
there exists a subsequence $\{\xv_{n_k}\}_{k=0}^\infty$ of the sequence $\{\xv_n\}_{n=0}^\infty$ generated by (\ref{Iduka2016}) that converges to an optimal solution of the problem (\ref{main pb dist}) is guaranteed. Moreover, if some $f_j$ is a strictly convex function, it has been proved that the whole sequence $\{\xv_n\}_{n=0}^\infty$ converges to the unique solution of the problem (\ref{main pb dist}). Some further improvements of iterative schemes for solving  the problem (\ref{main pb dist}) include the incremental scheme \cite{I-16-incremental}, proximal point schemes \cite{EAN-21,I-16-P}, and the stochastic fixed point subgradient method \cite{I-19} which can be seen as a stochastic version of the method (\ref{Iduka2016}).

On the other hand, it is well known that the subgradient type methods often struggle in the computationally expensive or impractical in some particular situations of the objective function. A simple technique for tackling with this issue is to reuse the previously obtained information,known as {\it delayed information}, instead of spending time computing new information. This technique is particularly beneficial and has been incorporated into iterative methods for solving not only large-scale convex optimization problems (for instance, incremental methods \cite{A-19,FAN-25,GOP-17,TY-14,VGO-18} and distributed methods \cite{ANN-23,NPN-23}), but also nonconvex optimization problems \cite{PN-25}.  
\vskip 2mm

In this paper, we propose a method, called the {\it fixed-point delayed subgradient method} (FDSM, for short), to solve the problem (\ref{main-pb}). This method is based on the ideas of the HSDM, the parallel subgradient method (\ref{Iduka2016}), and the delayed subgradient scheme. The key idea of the delayed subgradient scheme is that it allows the use of staled subgradient information from earlier iterations when updating the next iterate. The presence of delayed subgradients is practical since it helps in reducing the computational cost when evaluating subgradients at each iteration is expansive. Subsequently, we present convergence results of the proposed method. We then propose an inexact version of the FDSM in which the approximate subgradients are utilized. Furthermore, we also propose a distributed version of the FDSM for solving the centralized network system problem (\ref{main-pb}). Then we apply the proposed method to image inpainting problems. 
\vskip 1mm

The paper is organized as follows: Section \ref{Preliminary} contains the mathematical tools for analyzing the convergence of the proposed method.
Section \ref{Method and Convergent Results} presents the FDSM and convergence results, including discussion of sufficient conditions for convergences. This section is divided into two subsections. In the first subsection, we establish the rate of convergence of the sequence of function values. We show that the optimal value can be approximated by the best-achieved function values and, for a specific choice of  step sizes, we derive an upper bound on the difference between the best-achieved function values and the optimal value. In the second subsection, we  show the convergences of a subsequence of the generated sequence to an optimal solution of the problem. Moreover, we show that, by assuming strict convexity of the objective function, the whole generated sequence converges to an optimal solution of the considered problem. In Section \ref{inexact}, we consider an inexact version of the FDSM that uses approximate subgradients.
In Section \ref{distributed}, we consider a distributed version of the FDSM for solving a particular situation of the problem (\ref{main-pb}). We also establish convergence results for this case.
In Section \ref{Image Inpainting Problems}, we present numerical experiments applying the proposed method to image inpainting. Finally, we  give concluding remarks in Section \ref{conclusion}.

\section{Preliminaries}\label{Preliminary}

\quad\, In this section, we recall some notations, definitions, and useful facts that will be utilized in this work.
\vskip 2mm

Let $\mathbb{R}^d$ be a Euclidean space with the inner product $\langle \cdot,\cdot\rangle$ and its induced norm $\Vert \cdot \Vert$ and   $\mathbb{N}_0$ be the set of all nonnegative integers.  Let $T:\mathbb{R}^d \to \mathbb{R}^d$ be an operator. 
We say that $T$ is a {\it firmly nonexpansive} operator \cite[Definition 2.2.1]{C-12} if
$$
\langle T\xv-T\yv,\xv-\yv \rangle \geq \Vert T\xv-T\yv \Vert^2,\,\,\,\forall \xv, \yv \in \mathbb{R}^d.
$$

Firmly nonexpansive operators have some  properties that will be helpful for proving the convergence results in the subsequent sections. Readers may consult Cegielski's book \cite[Theorem 2.2.4, Theorem 2.2.5, Corollary 4.28, Lemma 3.2.5]{C-12} for proving details and further discussion.

\begin{fact}\label{fne-fact}
Let $T:\mathbb{R}^d \to \mathbb{R}^d$ be a firmly nonexpansive operator. Then the following statements hold:
\begin{enumerate}
    \item[\rm(1)] $T$ is a nonexpansive operator, i.e., $\Vert T\xv-T\yv \Vert \leq \Vert \xv-\yv \Vert$ for all $\xv, \yv \in \mathbb{R}^d;$ 
    \item[\rm(2)] If  $\fix(T) \neq \emptyset$, then $T$ is a cutter, i.e., $\langle \xv-T\xv, \zv-T\xv \rangle \leq 0$ for all $\xv\in \mathbb{R}^d$ and $\zv\in \fix(T);$ 
    \item[\rm(3)] $\fix(T)$ is a closed convex set; 
    \item[\rm(4)] If  $\fix(T) \neq \emptyset$, then $T$ satisfies the demi-closedness principle, i.e., for any sequence $\{\xv_n\}_{n=0}^\infty \subset \mathbb{R}^d$ such that $\displaystyle \lim_{n\to \infty}\xv_n=\yv \in \mathbb{R}^d$ and $\displaystyle \lim_{n\to \infty}\Vert T\xv_n-\xv_n \Vert =0,$ we have $\yv\in \fix(T).$ 
\end{enumerate}
\end{fact}

Let $X\subset \mathbb{R}^d$ be a nonempty closed convex set and $\xv \in \mathbb{R}^d$. 
The {\it metric projection} of $\xv$ onto $X$, denoted by $P_X(\xv)$, is the point $P_X(\xv)$ in which  
$$
\Vert P_X(\xv) -\xv \Vert = \inf_{\yv \in X}\Vert \yv-\xv \Vert.
$$

It should be noted that the metric projection $P_X(\xv)$ always exists and is unique for all $\xv\in \mathbb{R}^d$ \cite[Theorem 1.2.3]{C-12}. 
 The operator $P_X:\R^d\to X$ is called the {\it metric projection} onto $X$ \cite[Definition 1.2.1]{C-12}. 
 Moreover, the metric projection $P_X$ is a firmly nonexpansive operator with $\fix P_X = X$ \cite[Theorem 2.2.21 (iii)]{C-12}. Let  $f:\mathbb{R}^d \to \mathbb{R}$ be a real-valued function. The function $f$ is said to be {\it convex} \cite[Definition 1.1.48]{C-12} if
  $$
f((1-\lambda)\xv + \lambda \yv) \leq (1-\lambda)f(\xv)+\lambda f(\yv),\,\,\,\forall \xv, \yv \in \mathbb{R}^d,\, \lambda \in [0,1].
$$

  The function $f:\mathbb{R}^d \to \mathbb{R}$ is said to be {\it strictly convex} \cite[Definition 1.1.48]{C-12} if
 $$
 f((1-\lambda)\xv + \lambda \yv) < (1-\lambda)f(\xv)+\lambda f(\yv),\,\,\,\forall \xv, \yv \in \mathbb{R}^d,\, \lambda \in (0,1).
 $$
 
It is obvious that every strictly convex function is convex. 
For any $\xv\in\R^d$, the {\it subdifferential} of $f$ at $\xv$ \cite[Definition 1.1.55]{C-12} is the set
$$
\partial f(\xv) := \{ \gv\in \mathbb{R}^d : \langle \gv,\yv-\xv\rangle \leq f(\yv)-f(\xv),\,\forall\yv\in \mathbb{R}^d \}.
$$

The point $\gv\in \partial f(\xv)$ is called a {\it subgradient} of $f$ at $\xv$, which is denoted by $\widetilde{\nabla}f(\xv)$.
\vskip 2mm

The following facts demonstrate the properties of the subdifferentials of convex functions. Readers may consult Beck's book \cite[Theorem 3.9, Theorem 3.14]{B-17} and \cite[Proposition 16.20]{BC-11} for further details. 

\begin{fact}\label{sg-bbd}
Let $f: \mathbb{R}^d \to \mathbb{R}$ be a convex function. Then the following holds:
          \begin{enumerate}
              \item[\rm(1)] For every  $\xv\in \mathbb{R}^d$, the subdifferential $\partial f(\xv)$ is a nonempty closed and convex set; 
              \item[\rm(2)] If  $X \subset \mathbb{R}^d$ is nonempty and bounded, then $\bigcup_{\xv\in X}\partial f(\xv)$ is a nonempty and bounded set.
          \end{enumerate}
\end{fact}

 The following fact concerns the local Lipschitz property of convex functions, see Bauschke and Combettes \cite[Proposition 16.17]{BC-11} for more details.

\begin{fact}\label{f-Lip}
Let $f:\mathbb{R}^d \to \mathbb{R}$ be a convex function. If $X\subset \mathbb{R}^d$ is a nonempty bounded set, then $f$ is Lipschitz continuous relative to $X$, i.e., there exists $L>0$ such that
$$
\vert f(\xv)-f(\yv)\vert \leq L\Vert \xv-\yv\Vert,\,\,\,\forall \xv,\yv \in X.
$$  
\end{fact}

\section{The FDSM and Convergence Results} \label{Method and Convergent Results}

In this section, we will start by proposing the main iterative method for solving problem (\ref{main-pb}) and analyze its convergence properties.

\begin{assumption}\label{assumption_tau}
The sequence of time-varying delays $\{\tau_n\}_{n=0}^\infty \subset \mathbb{N}_0$ is bounded, that is, there exists a nonnegative integer  $\tau$ such that $0 \leq \tau_n \leq \tau$ for all $n\in \mathbb{N}_0$.
 \end{assumption}
  
\vskip 2mm
\begin{algorithm}[H]
	\SetAlgoLined
	\vskip2mm
	\textbf{Initialization}: Given a step size $\{\alpha_n\}_{n=0}^\infty \subset (0,\infty)$ and  initial points $\xv_0,\xv_{-1},\ldots,\xv_{-\tau}\in \mathbb{R}^d.$
	
	\textbf{Iterative Step}: For the current point $\xv_n\in \mathbb{R}^d$, we compute
    $$
\xv_{n+1}:= T\xv_n-\alpha_n {\widetilde{\nabla}} f(T\xv_{n-\tau_n}),\,\,\,\forall n\in \mathbb{N}_0,
$$
where $\widetilde{\nabla}f(T\xv_{n-\tau_n})$ is a stale subgradient of $f$ at $T\xv_{n-\tau_n}$.
 
	{\bf Update} $n:=n+1$.
	\caption{\bf Fixed-Point Delayed Subgradient Method (FDSM, for short)}
	\label{alg:FDSM}
\end{algorithm}
\vskip2mm

Throughout this paper, we assume that  the solution set $\mathcal{S}\neq\emptyset$.   Moreover, for simplicity of notation in our analysis, we take $\xv_0=\xv_{-1}=\cdots=\xv_{-\tau}$.

\begin{remark}\label{R:A:main}
\begin{enumerate}
    \item[\rm(1)] The convexity of the function $f$ ensures that subgradient of $f$ at at $T\xv_{n-\tau_n}$ exists for all $n\in \mathbb{N}_0$.
   \item[\rm(2)] In the case where $\tau_n=0$ for all $n\in \mathbb{N}_0$, the \FDSM{} relates to method (\ref{Iduka2016})  proposed by Iiduka \cite{I-16} with the number of users $m=1$. 
 \item[\rm(3)] The assumption on the existence of the bound $\tau$ of delays as in Assumption \ref{assumption_tau} is widely used to analyze convergence results of delayed-type methods.
    Besides the fixed delay $\tau_n=\bar{\tau}$ for some $\bar{\tau}>0$, another common choice is the cyclical delay $\tau_n = n \mod{(\tau+1)}$ \cite{ANN-23,NPN-23,GOP-17}. This choice produces a repeating pattern in which a stale subgradient is reused to update the next iterates instead of recalculating at every iteration. With the cyclical delay, it is noted that a new subgradient is calculated once every $\tau+1$ iterations.
\end{enumerate}
\end{remark}

We first derive a key inequality, which is an important tool for proving the convergence results.

\begin{lemma}\label{mainlemma}
Let  $\{\xv_n\}_{n=0}^\infty$ be a sequence generated by the \FDSM{}  and $a\in (0,1)$.  
Suppose that Assumption \ref{assumption_tau} holds.
Then, for all $n\in\mathbb{N}_0$ and  $ \xv \in \fix(T),$ we have
\begin{align*}
2\alpha_n (f(\xv_{n+1})-f(\xv)) 
    &\leq  \Vert \xv_n-\xv\Vert^2 -\Vert \xv_{n+1}-\xv\Vert^2
+2\alpha_n^2\|\widetilde{\nabla}f(\xv_{n+1})\|\|\widetilde{\nabla}f(T\xv_{n-\tau_n})\| \nonumber \\
        &\quad+8\alpha_n^{2-a}\|\widetilde{\nabla}f(T\xv_{n})\|^2 + 8\alpha_n^{2-a}\|\widetilde{\nabla}f(T\xv_{n+1})\|^2 
  +24\alpha_n^{2-a}\|\widetilde{\nabla} f(T\xv_{n-\tau_n})\|^2 \\
    &\quad     +\Big( \frac{3\alpha_n^a}{8}-1\Big)\Vert \xv_{n+1}-\xv_n \Vert^2   +\frac{2\alpha_n^a(\tau+1)}{8}\sum_{i=0}^\tau \Vert \xv_{n-i+1}-\xv_{n-i}\Vert^2.
\end{align*}
\end{lemma}

\begin{proof}
    Let $n\in\mathbb{N}_0$  and $\xv\in \fix(T)$. 
    Since $T$ is firmly nonexpansive, it follows from Fact \ref{fne-fact} (2) that
    \begin{align}\label{L1}
        \langle \xv_n-T\xv_n,\xv-T\xv_n \rangle \leq 0,
    \end{align}
   which  together with  $T\xv_n=\xv_{n+1}+\alpha_n\widetilde{\nabla}f(T\xv_{n-\tau_n})$
   yields
    \begin{align*}
        2\langle \xv_n-\xv_{n+1},\xv-T\xv_n \rangle 
        \leq 2\alpha_n\langle \widetilde{\nabla}f(T\xv_{n-\tau_n}),\xv-T\xv_n \rangle \nonumber.
    \end{align*}
    By using the definition of $\xv_{n+1}$ on the left-hand side of the above relation, we get 
    \begin{align}\label{L2}
        2\langle \xv_n-\xv_{n+1}, \xv-\xv_{n+1} \rangle \leq 2\alpha_n\langle \widetilde{\nabla}f(T\xv_{n-\tau_n}),\xv-T\xv_n\rangle   + 2\alpha_n\langle \widetilde{\nabla}f(T\xv_{n-\tau_n}),\xv_n-\xv_{n+1}\rangle.
    \end{align}

    Next, we will investigate the upper bound of the terms in the right-hand side of \eqref{L2}.
    From the well-known Young's inequality, note that
    \begin{equation}\label{L3}
    \big\langle 8\alpha_{n}^{\frac{2-a}{2}}\widetilde{\nabla}f(T\xv_{n-\tau_n}),\alpha_n^\frac{a}{2}(\xv_n-\xv_{n+1}) \big\rangle 
    \leq \frac{64\alpha_n^{2-a}}{2}\|\widetilde{\nabla}f(T\xv_{n-\tau_n})\|^2+\frac{\alpha_n^a}{2} \Vert \xv_n-\xv_{n+1}\Vert^2,
    \end{equation}
    that is,
    \begin{equation}\label{L4}
        2\alpha_n\langle \widetilde{\nabla}f(T\xv_{n-\tau_n}),\xv_n-\xv_{n+1}\rangle \leq 8\alpha_n^{2-a}\|\widetilde{\nabla}f(T\xv_{n-\tau_n})\|^2 + \frac{\alpha_n^a}{8} \Vert \xv_n-\xv_{n+1}\Vert^2.
    \end{equation}

Next, we consider the first term of the right-hand side of \eqref{L2} as
{\small \begin{align}\label{L5.0}
    2\alpha_n \langle \widetilde{\nabla}f(T\xv_{n-\tau_n}),\xv-T\xv_n \rangle
    =\ 2\alpha_n \langle \widetilde{\nabla}f(T\xv_{n-\tau_n}),\xv-T\xv_{n-\tau_n}\rangle 
     +2\alpha_n\langle \widetilde{\nabla}f(T\xv_{n-\tau_n}),T\xv_{n-\tau_n}-T\xv_n \rangle.
\end{align}}
Applying the definition of subgradients of $f$ at $T\xv_{n-\tau_n},\, \xv_{n+1},\, T\xv_n$ and $T\xv_{n+1}$, respectively, it follows from \eqref{L5.0} that
{\small
\begin{align}\label{L5}
2\alpha_n \langle \widetilde{\nabla}f(T\xv_{n-\tau_n}),\xv-T\xv_n \rangle
   & \leq  2\alpha_n(f(\xv)-f(\xv_{n+1}))  +2\alpha_n \langle \widetilde{\nabla}f(\xv_{n+1}),\xv_{n+1}-T\xv_{n}\rangle \nonumber\\
    &\quad+2\alpha_n \langle \widetilde{\nabla}f(T\xv_{n}),T\xv_{n}-T\xv_{n+1}\rangle  + 2\alpha_n \langle \widetilde{\nabla}f(T\xv_{n+1}),T\xv_{n+1}-T\xv_{n-\tau_n} \rangle \nonumber\\
    &\quad+2\alpha_n\langle \widetilde{\nabla}f(T\xv_{n-\tau_n}),T\xv_{n-\tau_n}-T\xv_n \rangle.
\end{align}
}
The Cauchy-Schwarz inequality together with $\xv_{n+1}-T\xv_n = -\alpha_n\widetilde{\nabla}f(T\xv_{n-\tau_n})$ implies
\begin{align}\label{L8-1}
    2\alpha_n\langle \widetilde{\nabla}f(\xv_{n+1}), \xv_{n+1}-T\xv_n\rangle \leq 2\alpha_n^2\|\widetilde{\nabla}f(\xv_{n+1})\|\|\widetilde{\nabla}f(T\xv_{n-\tau_n})\|.
\end{align}
Again, using Young's inequality as the same technique in \eqref{L3} and invoking the nonexpansiveness of $T$ in Fact \ref{fne-fact} (1), we obtain the following two inequalities:
\begin{align}\label{L6}
    2\alpha_n \langle \widetilde{\nabla}f(T\xv_n),T\xv_n-T\xv_{n+1} \rangle
    \leq 8\alpha_n^{2-a}\|\widetilde{\nabla}f(T\xv_n)\|^2 + \frac{\alpha_n^a}{8} \Vert \xv_n-\xv_{n+1}\Vert^2
\end{align}
and
{\small \begin{align}\label{L7}
    2\alpha_n \langle \widetilde{\nabla}f(T\xv_{n+1}),T\xv_{n+1}-T\xv_{n-\tau_n}\rangle
    \leq  8\alpha_n^{2-a}\|\widetilde{\nabla}f(T\xv_{n+1})\|^2 + \frac{\alpha_n^a}{8} \Vert \xv_{n+1}-\xv_{n-\tau_n}\Vert^2.
\end{align} }
By applying the same technique as above and the fact that $(u+v)^2\leq 2u^2 +2v^2$ for all $u,v\geq0$, we have
{\small \begin{align}\label{L8}
2\alpha_n \langle \widetilde{\nabla} f(T\xv_{n-\tau_n}), T\xv_{n-\tau_n}-T\xv_n\rangle 
    &\leq 16\alpha_n^{2-a}\|\widetilde{\nabla} f(T\xv_{n-\tau_n})\|^2 + \frac{\alpha_n^a}{8} \Vert \xv_n-\xv_{n+1}\Vert^2 \nonumber\\
    &\quad+ \frac{\alpha_n^a}{8} \Vert \xv_{n+1}-\xv_{n-\tau_n}\Vert^2.
\end{align} }
Substituting the inequalities \eqref{L8-1}, \eqref{L6}, \eqref{L7}, and \eqref{L8} into \eqref{L5}, we obtain 
{\small \begin{align}\label{L9}
2\alpha_n \langle \widetilde{\nabla} f(T\xv_{n-\tau_n}),\xv-T\xv_n \rangle 
    &\leq  2\alpha_n(f(\xv)-f(\xv_{n+1})) + 2\alpha_n^2\|\widetilde{\nabla}f(\xv_{n+1})\|\|\widetilde{\nabla}f(T\xv_{n-\tau_n})\|\nonumber\\
        &\quad+ 8\alpha_n^{2-a}\|\widetilde{\nabla}f(T\xv_{n})\|^2 
        + 8\alpha_n^{2-a}\|\widetilde{\nabla}f(T\xv_{n+1})\|^2 
  + 16\alpha_n^{2-a}\|\widetilde{\nabla}f(T\xv_{n-\tau_n})\|^2 \nonumber\\
    &\quad+ \frac{2\alpha_n^a}{8} \Vert \xv_{n+1}-\xv_n\Vert^2   + \frac{2\alpha_n^a}{8} \Vert \xv_{n+1}-\xv_{n-\tau_n} \Vert^2.
\end{align} }
It remains to produce the upper bound of the last term of \eqref{L9} as
\begin{align}\label{L10-1}
\Vert \xv_{n+1}-\xv_{n-\tau_n}\Vert^2 
    &= (\tau_n+1)^2 \Big\Vert \sum_{i=0}^{\tau_n} \frac{\xv_{n-i+1}-\xv_{n-i}}{\tau_n+1} \Big\Vert^2 \nonumber \\
  &\leq (\tau_n+1)\sum_{i=0}^{\tau_n} \Vert \xv_{n-i+1}-\xv_{n-i}\Vert^2 
   \leq (\tau+1)\sum_{i=0}^{\tau} \Vert \xv_{n-i+1}-\xv_{n-i}\Vert^2,
\end{align}
which, combined with the above relation, implies that
{\small
\begin{align}\label{L10}
  2\alpha_n \langle \widetilde{\nabla} f(T\xv_{n-\tau_n}),\xv-T\xv_n \rangle  
& \leq  2\alpha_n(f(\xv)-f(\xv_{n+1})) + 2\alpha_n^2\|\widetilde{\nabla}f(\xv_{n+1})\|\|\widetilde{\nabla}f(T\xv_{n-\tau_n})\|\nonumber\\
        &\quad+ 8\alpha_n^{2-a}\|\widetilde{\nabla}f(T\xv_{n})\|^2 
        + 8\alpha_n^{2-a}\|\widetilde{\nabla}f(T\xv_{n+1})\|^2
        + 16\alpha_n^{2-a}\|\widetilde{\nabla}f(T\xv_{n-\tau_n})\|^2
  \nonumber \\
    &\quad+ \frac{2\alpha_n^a}{8} \Vert \xv_{n+1}-\xv_n\Vert^2 +\frac{2\alpha_n^a(\tau+1)}{8}\sum_{i=0}^{\tau} \Vert \xv_{n-i+1}-\xv_{n-i}\Vert^2. 
\end{align}   
}

Finally, let us note that
\begin{align}\label{L11}
    2\langle \xv_n-\xv_{n+1},\xv-\xv_{n+1}\rangle = \Vert \xv_n-\xv_{n+1}\Vert^2 + \Vert \xv_{n+1}-\xv \Vert^2 - \Vert \xv_n-\xv \Vert^2,
\end{align}
    which, together with \eqref{L2}, \eqref{L4} and  \eqref{L10}, implies that 
\begin{align*}
2\alpha_n (f(\xv_{n+1})-f(\xv)) 
    &\leq  \Vert \xv_n-\xv\Vert^2 -\Vert \xv_{n+1}-\xv\Vert^2
+2\alpha_n^2\|\widetilde{\nabla}f(\xv_{n+1})\|\|\widetilde{\nabla}f(T\xv_{n-\tau_n})\| \nonumber \\
        &\quad+8\alpha_n^{2-a}\|\widetilde{\nabla}f(T\xv_{n})\|^2 + 8\alpha_n^{2-a}\|\widetilde{\nabla}f(T\xv_{n+1})\|^2 
  +24\alpha_n^{2-a}\|\widetilde{\nabla} f(T\xv_{n-\tau_n})\|^2 \\
    &\quad
    +\Big( \frac{3\alpha_n^a}{8}-1\Big)\Vert \xv_{n+1}-\xv_n \Vert^2   +\frac{2\alpha_n^a(\tau+1)}{8}\sum_{i=0}^\tau \Vert \xv_{n-i+1}-\xv_{n-i}\Vert^2,
\end{align*}
as desired. This completes the proof.
\end{proof}

\subsection{Convergence for the Best-Archieved Function Values}

\quad\, In order to establish the rate of convergence of the generated sequence of function values, we need the following assumption:

\begin{assumption}\label{A:sg bbd}
   There exists  $C > 0$ such that
    $\max\{\Vert \widetilde{\nabla}f(\xv_n)\Vert,\,\Vert \widetilde{\nabla}f(T\xv_n)\Vert\} \leq C$
    for all $n\in \mathbb{N}_0$. 
\end{assumption}

Assumption \ref{A:sg bbd} plays a crucial role in proving the convergence results for nonsmooth convex optimization both over simple constrained sets \cite{B-15,NB-01,M-08,RNV-09} and, as in the current setting, over the fixed-point sets of nonlinear operators \cite{HI-19,HI-19-Incremental,I-16}. 
Moreover, it is worth noting that a simple situation in which Assumption \ref{A:sg bbd} is satisfied is when the function $f$ is polyhedral, that is,  
$$
f(\xv):=\max_{1\leq i \leq m}\{\langle \cv_i,\xv\rangle + b_i\},\,\,\,\forall \xv\in\R^d,
$$
 where $\cv_i\in\mathbb{R}^d$ and $b_i\in\mathbb{R}$ for all $i=1,\ldots,m$. In this case, the subdifferential set of $f$ at $\xv$ is the convex hull of the set $\{\cv_1,\ldots,\cv_m\}$ and the constant $C$ in Assumption \ref{A:sg bbd} can be taken as $C:=\max_{1\leq i \leq m}\|\cv_i\|$. The reader may consult, for instance,  Bertsekas' book \cite[Assumption 3.2.1]{B-15}. 
\vskip 2mm

Now, we are in a position to prove one of the main convergence results. The following theorem shows that the best-archieved function values $\left\{\min_{0\leq n \leq N}f(\xv_{n+1})\right\}_{N=0}^\infty$ approximate the optimal value $f^*$.

\begin{theorem}\label{thm:rate}
Let  $\{\xv_n\}_{n=0}^\infty$ be a sequence generated by the \FDSM{} and $a\in (0,1)$.  
Suppose that Assumptions \ref{assumption_tau} and \ref{A:sg bbd} hold and  the sequence $\{\alpha_n\}_{n=0}^\infty\subset (0,\infty)$ is nonincreasing and satisfies  $(3+2(\tau+1)^2)\alpha_0^a < 8$. 
Then, for all $N\in\mathbb{N}_0$ and  $ \xv^* \in \mathcal{S}$, we have
\begin{equation*}
    \min_{0\leq n \leq N}f(\xv_{n+1})-f^* \leq \frac{\Vert \xv_0-\xv^*\Vert^2 + 2C^2\sum_{n=0}^N\alpha_n^2+ 40C^2\sum_{n=0}^N \alpha_n^{2-a}}{2\sum_{n=0}^N \alpha_n}.
\end{equation*}   
\end{theorem}

\begin{proof}
Let $\xv^*\in \mathcal{S}$ be given. By using the inequality obtained in Lemma \ref{mainlemma} together with Assumption \ref{A:sg bbd}, it follows that, for all $n\in\N_0$,
\begin{align*}
2\alpha_n (f(\xv_{n+1})-f^*)  
    &\leq \Vert \xv_n-\xv^*\Vert^2 -\Vert \xv_{n+1}-\xv^*\Vert^2
+2\alpha_n^2\|\widetilde{\nabla}f(\xv_{n+1})\|\|\widetilde{\nabla}f(T\xv_{n-\tau_n})\| \nonumber \\
        &\quad+8\alpha_n^{2-a}\|\widetilde{\nabla}f(T\xv_{n})\|^2 + 8\alpha_n^{2-a}\|\widetilde{\nabla}f(T\xv_{n+1})\|^2 
  +24\alpha_n^{2-a}\|\widetilde{\nabla} f(T\xv_{n-\tau_n})\|^2 \\
    &\quad
    +\Big( \frac{3\alpha_n^a}{8}-1\Big)\Vert \xv_{n+1}-\xv_n \Vert^2   +\frac{2\alpha_n^a(\tau+1)}{8}\sum_{i=0}^\tau \Vert \xv_{n-i+1}-\xv_{n-i}\Vert^2\\
    &\leq \Vert \xv_n-\xv^*\Vert^2 -\Vert \xv_{n+1}-\xv^*\Vert^2
    +2C^2\alpha_n^2+40\alpha_n^{2-a}C^2  \\
    &\quad  +\Big( \frac{3\alpha_n^a}{8}-1\Big)\Vert \xv_{n+1}-\xv_n \Vert^2+\frac{2\alpha_n^a(\tau+1)}{8}\sum_{i=0}^\tau \Vert \xv_{n-i+1}-\xv_{n-i}\Vert^2.
\end{align*}

Now, for any nonnegative integer $N$, summing the above inequality from $n=0$ to $n=N$, telescoping the consecutive terms, and neglecting the nonposive term, we obtain 
{\small \begin{align}\label{TT2}
\sum_{n=0}^N 2\alpha_n (f(\xv_{n+1})-f^*)
    &\leq \Vert \xv_0-\xv^*\Vert^2  + 2C^2\sum_{n=0}^N\alpha_n^2  +40C^2\sum_{n=0}^N \alpha_n^{2-a} + \sum_{n=0}^N \Big( \frac{3\alpha_n^a}{8}-1\Big)\Vert \xv_{n+1}-\xv_n \Vert^2 \nonumber\\
    &\quad  +  \frac{2(\tau+1)}{8}\sum_{n=0}^N \alpha_n^a \sum_{i=0}^\tau \Vert \xv_{n-i+1}-\xv_{n-i}\Vert^2. 
\end{align}
}
Now, by changing the order of the double summation and shifting the summation index in the last term of \eqref{TT2}, we have
\begin{align}\label{TT3}
\sum_{n=0}^N \alpha_n^a \sum_{i=0}^\tau \Vert \xv_{n-i+1}-\xv_{n-i}\Vert^2
   & = \sum_{i=0}^\tau \sum_{n=-i}^{N-i}\alpha_{n+i}^a \Vert \xv_{n+1}-\xv_n \Vert^2 \nonumber \\
    &\leq \sum_{i=0}^\tau \sum_{n=0}^{N}\alpha_{n}^a \Vert \xv_{n+1}-\xv_n \Vert^2 \nonumber \\
   & = (\tau+1)\sum_{n=0}^N \alpha_{n}^a \Vert \xv_{n+1}-\xv_n \Vert^2,
\end{align}
where the inequality holds true by the assumptions that $\xv_0=\xv_{-1}=\cdots=\xv_{-\tau}$ and $\{\alpha_n\}_{n=0}^\infty$ is nonincreasing. Thus, by combining the obtained relations \eqref{TT2} and \eqref{TT3}, we have 
\begin{align*}
2\sum_{n=0}^N \alpha_n(f(\xv_{n+1})-f^*)
    &\leq \Vert \xv_0-\xv^*\Vert^2 +2C^2\sum_{n=0}^N\alpha_n^2 +40C^2\sum_{n=0}^N \alpha_n^{2-a} \nonumber\\
    &\quad +\sum_{n=0}^N \left( \frac{(3 + 2(\tau+1)^2)\alpha_n^a }{8} -1\right)\Vert \xv_{n+1}-\xv_n \Vert^2, 
    \end{align*}
which, together with the assumption that the stepsize $\alpha_n$ is nonincreasing and satisfies $(3 + 2(\tau+1)^2)\alpha_0^a <8 $, yields 
\begin{align*}
    2\sum_{n=0}^N \alpha_n(f(\xv_{n+1})-f^*) \leq \Vert \xv_0-\xv^*\Vert^2 +2C^2\sum_{n=0}^N\alpha_n^2 +40C^2\sum_{n=0}^N \alpha_n^{2-a}.
\end{align*}
Hence we have
\begin{equation*}
    \min_{0\leq n \leq N}f(\xv_{n+1})-f^* \leq \frac{\Vert \xv_0-\xv^*\Vert^2 +2C^2\sum_{n=0}^N\alpha_n^2+ 40C^2\sum_{n=0}^N \alpha_n^{2-a}}{2\sum_{n=0}^N \alpha_n}.
\end{equation*}
This  completes the proof. 
\end{proof}

In the following corollary, by considering a specific choice of the step size $\{\alpha_n\}_{n=0}^\infty$, we derive an explicit  upper bound of the gap between the sequence $\left\{\min_{0\leq n \leq N}f(\xv_{n+1})\right\}_{N=0}^\infty$ of best-archieved function values and the optimal value $f^*$.

\begin{corollary}
Let  $\{\xv_n\}_{n=0}^\infty$ be a sequence generated by the \FDSM{} and $a\in (0,1)$.  
Suppose that Assumptions \ref{assumption_tau} and \ref{A:sg bbd} hold. If the stepsize sequence $\{\alpha_n\}_{n=0}^\infty$ is given by $\alpha_n=\frac{\alpha}{n+1}$, where $\alpha>0$, such that  $(3+2(\tau+1)^2)\alpha^a < 8$, then, for all $N\in\mathbb{N}_0$ and  $\xv^* \in \mathcal{S}$,  it follows that
    \begin{align*}
        \min_{0\leq n \leq N} f(\xv_{n+1})-f^* \leq \frac{\Vert \xv_0-\xv^*\Vert^2 +4C^2\alpha^2+ 40C^2\Big({\frac{\alpha^{2-a}(2-a)}{1-a}}\Big)}{2\alpha}\Big(\frac{1}{\log(N+2)}\Big).
    \end{align*}
\end{corollary}

\begin{proof}
By invoking \cite[Lemma 8.26]{B-17}, we  estimate the upper bound of the sum in the numerator and the lower bound of the sum in the denominator of the right-hand side term in Theorem \ref{thm:rate} as follows: 
Let  $N\in\mathbb{N}_0$ be fixed. Now, we note that 
    \begin{align*}
            \sum_{n=0}^N\alpha_n^{2-a} 
                = \alpha^{2-a} + \alpha^{2-a}\sum_{n=1}^N \frac{1}{(n+1)^{2-a}} 
                \leq \alpha^{2-a}+ \alpha^{2-a}\int_{0}^{N} \frac{1}{(x+1)^{2-a}}\, dx  
                \leq \frac{\alpha^{2-a}(2-a)}{1-a}
    \end{align*}
    and
    \begin{align*}
        \sum_{n=0}^N\alpha_n^{2} 
                = \alpha^{2} + \alpha^{2}\sum_{n=1}^N \frac{1}{(n+1)^{2}} 
                \leq \alpha^{2}+ \alpha^{2}\int_{0}^{N} \frac{1}{(x+1)^{2}}\, dx  
                \leq 2\alpha^2.
    \end{align*}
    In addition, we have
    \begin{align*}
        \sum_{n=0}^N \alpha_n
            = \alpha \sum_{n=0}^N \frac{1}{n+1}
            \geq \alpha\int_{0}^{N+1}\frac{1}{x+1}\, dx
            = \alpha \log(N+2).
    \end{align*}
By combing these three estimates with Theorem \ref{thm:rate}, we obtain the required inequality. This completes the proof. 
\end{proof}

\subsection{Convergence of the Generated Sequence}

\quad\, Note that the convergence results previously established show the convergence of the function values. In this subsection, we show the convergence of the generated sequence to an optimal solution of the problem (\ref{main-pb}). To do so, we need a rather strong boundedness than Assumption \ref{A:sg bbd}.

\begin{assumption}\label{A: xn bbd}
    There exists $M>0$ such that $\Vert \xv_n \Vert \leq M$ for all $n \in \mathbb{N}_0.$
\end{assumption}
 
    Due to the nonmonotone behaviour of the generated sequence $\{\xv_n\}_{n=0}^\infty$, Assumption \ref{A: xn bbd} is an essential property for ensuring the existence of a convergent subsequence of the generated sequence. Indeed, Assumption \ref{A: xn bbd} has been used \cite{HI-18,I-15-Acceleration,I-16,I-16-incremental,I-16-P,SHI-20} when demonstrating the convergence of the generated sequence, or of  a subsequence of the generated sequence to an optimal solution of the convex optimization problem with fixed-point constraints. {  
    
    Note that Assumption \ref{A: xn bbd} implies Assumption \ref{A:sg bbd}.
Indeed, for a fixed $\xv\in\fix(T)$, we have from the facts that $\{\xv_n\}_{n=0}^\infty$ is bounded and $T$ is nonexpansive that
$$
\|T\xv_n\| \leq \|T\xv_n-\xv\| + \|\xv\| \leq \|\xv_n-\xv\| + \|\xv\| \leq M+2\|\xv\|,\,\,\,\forall n\in\mathbb{N}_0,
$$
and so $\{T\xv_n\}_{n=0}^\infty$ is bounded. Thus it follows
from Fact \ref{sg-bbd} that $\bigcup_{n\in\mathbb{N}_0}\partial f(\xv_n)$ and
$\bigcup_{n\in\mathbb{N}_0}\partial f(T\xv_n)$ are bounded and so
Assumption \ref{A:sg bbd} holds.} 

Moreover, a typical numerical strategy for ensuring  Assumption \ref{A: xn bbd} in the literature is to consider a well-chosen compact convex set $X$  such that $\fix(T) \subset X$. One simple choice for $X$ is a closed ball with a large enough radius such that $\fix(T) \subset X$. 
In this situation, the sequence $\xv_{n+1}$ in \FDSM{} is modified as follows:
$$
\xv_{n+1} := P_X(T\xv_n-\alpha_n \widetilde{\nabla}f(T\xv_{n-\tau_n})).
$$
  
   For the case when the initial point $\xv_0 \in X$, we immediately get that the sequence $\{\xv_n\}_{n=0}^\infty$ is included in the compact convex set $X$ and hence $\{\xv_n\}_{n=0}^\infty$ is also bounded.
    If the initial point $\xv_0 \notin X$, then $\{\xv_n\}_{n=1}^\infty \subset X$ is bounded, which is denoted its bound by $M_1$.
    In this case, we have $\{\xv_n\}_{n=0}^\infty$ is bounded by $\mathrm{max}\{\Vert \xv_0 \Vert, M_1\}$. 
    \vskip 2mm

{
For simplicity of notation, we denote, for every $n\in \N_0$ and $\xv\in \fix(T)$, that
\begin{eqnarray}\label{vn-def} 
V_n(\xv) := \|\xv_n-\xv\|^2 + \frac{1}{4(\tau+1)}\sum_{i=1}^{\tau+1} (\tau+2-i)\|\xv_{n-i+1}-\xv_{n-i}\|^2.
\end{eqnarray}

\begin{lemma}\label{LL}
   Let  $\{\xv_n\}_{n=0}^\infty$ be a sequence generated by the \FDSM{}.  
Suppose that Assumptions \ref{assumption_tau} and \ref{A:sg bbd} hold.  
Then, for all ${n}\in\mathbb{N}_0$ and  $ \xv \in \fix(T)$, we have
    \begin{align*}
        V_{n+1}(\xv)
       & \leq
        V_n(\xv) - \frac{\|\xv_{n+1}-\xv_n\|^2}{4} - \frac{\|\xv_n-\xv_{n-\tau_n}\|^2}{8(\tau+1)^2}\\
        &\quad  + (2+8(\tau+1)^2)C^2\alpha_n^2 
        + 2\alpha_n( f(\xv)-f(T\xv_{n-\tau_n})).
    \end{align*}
\end{lemma}

\begin{proof}
   Let $n\in \mathbb{N}_0$ and $\xv \in \fix(T)$.
    As the same analogy of the inequality \eqref{L2} in Lemma \ref{mainlemma}, we have
    \begin{align}\label{LL1}
        2\langle \xv_n-\xv_{n+1}, \xv-\xv_{n+1} \rangle 
        \leq 2\alpha_n\langle \widetilde{\nabla}f(T\xv_{n-\tau_n}),\xv-T\xv_n\rangle + 2\alpha_n\langle \widetilde{\nabla}f(T\xv_{n-\tau_n}),\xv_n-\xv_{n+1}\rangle. 
    \end{align}
Note that 
\begin{align}\label{LL4}
        2\langle \xv_n-\xv_{n+1},\xv-\xv_{n+1}\rangle = \Vert \xv_n-\xv_{n+1}\Vert^2 + \Vert \xv_{n+1}-\xv \Vert^2 - \Vert \xv_n-\xv \Vert^2.
\end{align}

Let us focus on the terms at the right-hand side of \eqref{LL1}. We note that
 \begin{align}\label{LL2}
    2\alpha_n\langle \widetilde{\nabla}f(T\xv_{n-\tau_n}),\xv-T\xv_{n}\rangle
    &=  2\alpha_n\langle \widetilde{\nabla}f(T\xv_{n-\tau_n}),\xv-T\xv_{n-\tau_n}\rangle + 2\alpha_n\langle \widetilde{\nabla}f(T\xv_{n-\tau_n}),T\xv_{n-\tau_n}-T\xv_n\rangle\nonumber\\
    &\leq 2\alpha_n(f(\xv)-f(T\xv_{n-\tau_n})) + 2\alpha_nC\|\xv_{n-\tau_n}-\xv_n\| \nonumber\\
    &\leq 2\alpha_n(f(\xv)-f(T\xv_{n-\tau_n})) + \frac{\|\xv_{n-\tau_n}-\xv_n\|^2}{8(\tau+1)^2} + 8(\tau+1)^2\alpha_n^2 C^2,
\end{align} 
where the first inequality follows from the definition of the subgradient $f$ at $T\xv_{n-\tau_n}$, the Cauchy-Schwarz inequality, the nonexpansiveness of $T$, and Assumption \ref{A:sg bbd}, and the second inequality is due to Young's inequality.

Next, it follows from the Cauchy-Schwarz inequality, Assumption \ref{A:sg bbd}, and Young's inequality that
\begin{align}\label{LL3}
      2\alpha_n\langle \widetilde{\nabla}f({T}\xv_{n-\tau_n}),\xv_n-\xv_{n+1}\rangle 
     \leq 2\alpha_nC\|\xv_n-\xv_{n+1}\|
     \leq 2\alpha_n^2C^2 + \frac{\|\xv_n-\xv_{n+1}\|^2}{2}.
\end{align}
By substituting \eqref{LL4}, \eqref{LL2} and \eqref{LL3}  into \eqref{LL1}, we obtain
    \begin{align}\label{LL5}
        \|\xv_{n+1}-\xv\|^2-\|\xv_n-\xv\|^2
       &\leq
            \frac{\|\xv_{n-\tau_n}-\xv_n\|^2}{8(\tau+1)^2} - \frac{\|\xv_{n+1}-\xv_n\|^2}{2} \nonumber\\
        &\quad+ (8(\tau+1)^2+2)\alpha_n^2 C^2 + 2\alpha_n(f(\xv)-f(T\xv_{n-\tau_n})).
    \end{align}   
    
Let us note that, if $\tau_n=0$, then we get  
$$
\|\xv_n-\xv_{n-\tau_n}\|^2\leq (\tau+1)\sum_{i=1}^{\tau+1}\|\xv_{n-i+1}-\xv_{n-i}\|^2.
$$

Otherwise, we  have
\begin{align*}
    \|\xv_n-\xv_{n-\tau_n}\|^2 
    &=\tau_n^2\Big\|\sum_{i=1}^{\tau_n}\frac{(\xv_{n-i+1}-\xv_{n-i})}{\tau_n}\Big\|^2\\
    &\leq \tau_n \sum_{i=1}^{\tau_n} \|\xv_{n-i+1}-\xv_{n-i}\|^2 \\
    &\leq (\tau+1)\sum_{i=1}^{\tau+1} \|\xv_{n-i+1}-\xv_{n-i}\|^2.
\end{align*}
Thus we obtain 
\begin{align}\label{LL6}
    V_{n+1}(\xv)-V_n(\xv)
    &= \|\xv_{n+1}-\xv\|^2-\|\xv_n-\xv\|^2 + \frac{\|\xv_{n+1}-\xv_n\|^2}{4}\nonumber\\
    &\quad -\frac{1}{4(\tau+1)}\sum_{i=1}^{\tau+1}\|\xv_{n-i+1}-\xv_{n-i}\|^2  \nonumber\\
    &\leq \|\xv_{n+1}-\xv\|^2-\|\xv_n-\xv\|^2 + \frac{\|\xv_{n+1}-\xv_n\|^2}{4}-\frac{\|\xv_n-\xv_{n-\tau_n}\|^2}{4(\tau+1)^2}. 
\end{align}

Combining \eqref{LL5} and \eqref{LL6} leads to
\begin{align*}
       V_{n+1}(\xv)
        &\leq        V_n(\xv) - \frac{\|\xv_{n+1}-\xv_n\|^2}{4} - \frac{\|\xv_n-\xv_{n-\tau_n}\|^2}{8(\tau+1)^2} + (8(\tau+1)^2+2)C^2\alpha_n^2 \\
        &\quad + 2\alpha_n( f(\xv)-f(T\xv_{n-\tau_n})).
\end{align*}
This completes the proof.
\end{proof}
}

To derive the convergence of the generated sequence, we use the following lemma.

\begin{lemma} {\rm \cite[Lemma 3.1]{M-08}\label{prop1}}
    Let $\{a_n\}_{n=0}^\infty$ be a sequence of nonnegative real numbers such that there exists a subsequence $\{a_{n_j}\}_{j=0}^\infty$ of $\{a_n\}_{n=0}^\infty$ with { $a_{n_j} < a_{n_{j}+1}$} for all $j\in \N_0$.  For all $n\geq n_0$, define
   $\mu(n) = \mathrm{max}\{k\in \mathbb{N}_0 : n_0 \leq k \leq n, a_k<a_{k+1}\}$.
    Then the sequence $\{\mu (n) \}_{n=n_0}^\infty$ is nondecreasing such that $\displaystyle\lim_{n\to \infty} \mu(n) = \infty$. Moreover, it holds that $a_{\mu(n)} \leq a_{\mu(n)+1}$ and $a_n \leq a_{\mu(n)+1}$  for all $n\geq n_0$.
\end{lemma}

The next theorem shows the existence of a subsequence of $\{\xv_n\}_{n=0}^\infty$ that converges to an optimal solution of problem  (\ref{main-pb}).

{
\begin{theorem}\label{thm:con}
   Let  $\{\xv_n\}_{n=0}^\infty$ be a sequence generated by the \FDSM{}.  
Suppose that Assumptions \ref{assumption_tau} and \ref{A: xn bbd} hold and  $\mathcal{S}\neq\emptyset$. If  the sequence $\{\alpha_n\}_{n=0}^\infty\subset (0,\infty)$ satisfies that 
$\lim_{n\to \infty}\alpha_n=0$ and 
$\sum_{n=0}^\infty \alpha_n = \infty$, then there exists a subsequence of $\{\xv_n\}_{n=0}^\infty$ that converges to an optimal solution in $\mathcal{S}$.
\end{theorem}

\begin{proof}
    Let $\xv^* \in \mathcal{S}$ be given.
The boundedness of the sequence $\{\xv_n\}_{n=0}^\infty$ and the nonexpansiveness of $T$ yield that  
$$\|T\xv_{n-\tau_n}-\xv^*\|\leq \|\xv_{n-\tau_n} - \xv^*\|\leq M+\|\xv^*\|,\,\,\,\forall n\in\N_0.
$$
Moreover, since $f$ is Lipschitz continuous relative to every bounded subset of $\R^d$, there is a Lipschitz constant $L>0$ such that
$$
{f^*}-f(T\xv_{n-\tau_n}) \leq L\Vert \xv^* - T\xv_{n-\tau_n} \Vert \leq L(\Vert \xv^* \Vert + M),\,\,\,\ \forall n\in\N_0.
$$ 
Thus it follows from Lemma \ref{LL} that, for all $n\in\N_0$,
\begin{align}\label{tt1}
V_{n+1}(\xv^*) &\leq V_n(\xv^*) - \frac{\|\xv_{n+1}-\xv_n\|^2}{4}-\frac{\|\xv_n-\xv_{n-\tau_n}\|^2}{8(\tau+1)^2} \nonumber\\
&\quad+(2+8(\tau+1)^2)C^2\alpha_n^2 +2L\alpha_n(\|\xv^*\|+M).
\end{align}

Next, we divide the existence of a convergent subsequence of $\{\xv_n\}_{n=0}^\infty$ into two cases according to the behavior of the sequence $\{ V_n(\xv^*) \}_{n=0}^\infty$.
\vskip 2mm

{\bf Case 1:} Suppose that there exists 
$n_0\in \mathbb{N}_0$ such that $V_{n+1}(\xv^*)\leq V_n(\xv^*)$
 for all $n \geq n_0$. 
 Since the sequence $\{ V_n(\xv^*) \}_{n=0}^\infty$ is bounded from below, we can ensure that $\displaystyle \lim_{n\to \infty}V_n $ exists.
 It follows from \eqref{tt1} that
 \begin{align*}
     \frac{\|\xv_{n+1}-\xv_n\|^2}{4} +\frac{\|\xv_n-\xv_{n-\tau_n}\|^2}{8(\tau+1)^2} 
     &\leq V_n(\xv^*)-V_{n+1}(\xv^*) +(2+8(\tau+1)^2)C^2\alpha_n^2\\ &\quad+2L\alpha_n(\|\xv^*\|+M).
 \end{align*}
    By taking the limit as $n\to \infty$ in the above inequality and using the fact that $\displaystyle\lim_{n\to \infty}\alpha_n =0$, we obtain 
    \begin{align}\label{tt2.1}
        \lim_{n\to \infty}\Vert \xv_{n+1}-\xv_n \Vert =
    \lim_{n\to \infty}\|\xv_n-\xv_{n-\tau_n}\|=0.
    \end{align}
     Since 
$\Vert T\xv_n-\xv_n \Vert \leq \Vert T\xv_n-\xv_{n+1}\Vert + \Vert \xv_{n+1}-\xv_n \Vert         \leq  \alpha_n C + \Vert \xv_{n+1}-\xv_n \Vert$ for all $n\in\N_0$,
 letting the limit as $n\to \infty$ gives
   \begin{align}\label{tt2}
       \lim_{n\to \infty}\Vert T\xv_n -\xv_n \Vert =0.
   \end{align}

 On the other hand, let us note that,  for all $n\in\N_0$,
 \begin{align}\label{tt6}
    f(T\xv_{n-\tau_n})-{f^*} 
    &= f(T\xv_{n-\tau_n})-f(T\xv_n) + f(T\xv_n)-f(\xv_{n}) + f(\xv_{n})-{f^*} \nonumber\\
    &\geq -L\|T\xv_{n-\tau_n}-T\xv_n\|-L\Vert T\xv_n-\xv_{n}\Vert + f(\xv_{n})-{f^*} \nonumber\\
    &\geq -L\|\xv_{n-\tau_n}-\xv_n\|-L\|T\xv_n-\xv_n\| + f(\xv_{n})-{f^*.} 
   \end{align}  
   By invoking this relation in the inequality obtained in Lemma \ref{LL}, it follows that, for all $n\in\N_0$,
   \begin{align*}
      2\alpha_n(f(\xv_{n})-{f^*}
       &\leq V_n(\xv^*) - V_{n+1}(\xv^*) + 2L\alpha_n\|\xv_{n-\tau_n}-\xv_n\|\\
       &\quad + 2\alpha_nL\|T\xv_n-\xv_n\|
       + (2+8(\tau+1)^2)C^2\alpha_n^2.    
   \end{align*} 
For any $n\in \mathbb{N}_0$, we define
     $$
     M_n:= f(\xv_{n})-f^* - L\|\xv_n-\xv_{n-\tau_n}\|-L\|T\xv_n-\xv_n\|-(1+4(\tau+1)^2)C^2\alpha_n.
     $$
 Then the above inequality becomes
     $\alpha_n M_n \leq \frac{V_n(\xv^*)}{2} -\frac{V_{n+1}(\xv^*)}{2}.$   
Let $N\in \mathbb{N}_0$ be fixed.
By summing up this inequality from $n=0$ to $N$, we obtain 
\begin{align*}
\sum_{n=0}^N\alpha_nM_n \leq \frac{V_0(\xv^*)}{2} -\frac{V_{N+1}(\xv^*)}{2} \leq \frac{V_0(\xv^*)}{2}. 
\end{align*}

Now, we will show that $\liminf_{n\to \infty}M_n \leq 0$.
Assume the contrary that
$\liminf_{n\to \infty}M_n >0$.
Consequently, there exist $m\in \mathbb{N}_0$ and $\epsilon >0$ such that
$M_n \geq \epsilon$ for all $n\geq m$.
{
For any integer $M> m$, we obtain 
$$
\sum_{n=m+1}^{M} \epsilon \alpha_n 
\leq \sum_{n=m+1}^{M} \alpha_n M_n = \sum_{n=0}^M \alpha_n M_n-\sum_{n=0}^{m}\alpha_nM_n 
\leq \frac{V_0(\xv^*)}{2}-\sum_{n=0}^{m}\alpha_nM_n.
$$
As $M\to \infty$, it follows that $\infty=\sum_{n=m}^\infty\epsilon \alpha_n \leq \frac{V_0(\xv^*)}{2}-\sum_{n=0}^{m}\alpha_nM_n$, 
which is a contradiction.
}
Thus we have $\liminf_{n\to \infty}M_n \leq 0$.
By using $\lim_{n\to \infty}\alpha_n=0$, \eqref{tt2.1} and \eqref{tt2}, we obtain
\begin{align}
    \liminf_{n\to \infty}f(\xv_{n})\leq f^*.
\end{align}
It follows from the boundedness of $\{\xv_{n}\}_{n=0}^\infty$ and the continuity of $f$ that there exists a subsequence $\{\xv_{n_p}\}_{p=0}^\infty$ of $\{\xv_{n}\}_{n=0}^\infty$ such that
\begin{align}\label{C22}
    \lim_{p\to \infty}f(\xv_{n_p})=\liminf_{n\to \infty}f(\xv_{n}) \leq f^*.
\end{align}
Now, since $\{\xv_{n_p}\}_{p=0}^\infty$ is also a bounded sequence, there exists a subsequence $\{\xv_{n_{p_l}}\}_{l=0}^\infty$ of $\{\xv_{n_p}\}_{p=0}^\infty$ such that
\begin{align}\label{C23}
    \lim_{l\to \infty}\xv_{n_{p_l}} = \overline{\xv} \in \mathbb{R}^d.
\end{align}
Invoking the demi-closedness principle of $T$ as in Fact \ref{fne-fact} (4),  \eqref{tt2} and \eqref{C23}, we ensure that $\overline{\xv}\in \fix(T)$.

 Finally, by applying the continuity of $f$ together with the obtained results \eqref{C22} and \eqref{C23}, we have
$$
f(\overline{\xv})=\displaystyle\lim_{l \to \infty}f(\xv_{n_{p_l}}) \leq f^*,
$$
which  implies that $\overline{\xv} \in \mathcal{S}$.
\vskip 2mm

{\bf Case 2:} Suppose that there exists a subsequence $\{V_{n_j}(\xv^*)\}_{j=0}^\infty$ of $\{V_n(\xv^*)\}_{n=0}^\infty$ such that 
$$
V_{n_j}(\xv^*) < V_{n_j+1}(\xv^*),\,\,\,\forall j\in \mathbb{N}_0.
$$

Let $m_0$ be the smallest integer such that $V_{m_0}(\xv^*) < V_{m_0+1}(\xv^*)$.
Thus, by defining the indexing sequence $\{\mu(n)\}_{n=m_0}^\infty$ 
as in Lemma \ref{prop1}, we have 
\begin{align}\label{tt5}
 V_{\mu(n)}(\xv^*) \leq V_{\mu(n)+1}(\xv^*),\,\,\,\forall n \geq m_0.
\end{align}
Then,  taking $n\to\infty$ and using the fact that $\displaystyle\lim_{n\to \infty} \alpha_n=0$, it follows from \eqref{tt1} that
\begin{align}\label{tt7}
    \lim_{n\to \infty} \Vert \xv_{\mu(n)+1}-\xv_{\mu(n)}\Vert = \lim_{n\to \infty}\Vert \xv_{\mu(n)}- \xv_{\mu(n)-\tau_{\mu(n)}}  \Vert=0.
\end{align}
By following the above lines, we also obtain 
   \begin{align}\label{tt8}
       \lim_{n\to \infty}\Vert T\xv_{\mu(n)} -\xv_{\mu(n)} \Vert =0.
   \end{align}
Again, by applying the inequality obtained in Lemma \ref{LL} with \eqref{tt6}, \eqref{tt5} and dividing both sides by $2\alpha_{\mu(n)}>0$, we have 
       $$
       \aligned
       f(\xv_{\mu(n)})-f^* &\leq L\|\xv_{\mu(n)-\tau_{\mu(n)}}-\xv_{\mu(n)}\|+ L\|T\xv_{\mu(n)}-\xv_{\mu(n)}\|\\
      &\quad + (1+4(\tau+1)^2)C^2\alpha_{\mu(n)},\,\,\,\forall n\geq m_0.
      \endaligned
      $$ 
   Subsequently, by using $\lim_{n\to \infty}\alpha_n=0$,  \eqref{tt7} and \eqref{tt8}, we obtain 
       $$
       \limsup_{n\to \infty}f(\xv_{\mu(n)})\leq f^*.
       $$
Fix a subsequence $\{ \xv_{\mu(n_{p})}\}_{p=0}^\infty$ of $\{ \xv_{\mu(n)}\}_{n=m_0}^\infty$ and then it follows that
$$
\limsup_{p \to \infty}f(\xv_{\mu(n_{p})})\leq \limsup_{n\to \infty}f(\xv_{\mu(n)})\leq f^*.
$$
The boundedness of $\{ \xv_{\mu(n_p)}\}_{p=0}^\infty$  ensures that there exists a subsequence $\{ \xv_{\mu(n_{p_l})}\}_{l=0}^\infty$ of $\{ \xv_{\mu(n_p)}\}_{p=0}^\infty$ such that
\begin{align}\label{C36}
    \displaystyle \lim_{l\to \infty} \xv_{\mu(n_{p_l})} = \overline{\xv}\in \mathbb{R}^d.
\end{align}
The demi-closedness principle of $T$ together with the relations \eqref{tt8} and \eqref{C36} imply $\overline{\xv}\in \fix(T).$
Hence the continuity of $f$ yields
$$
f(\overline{\xv}) 
= \lim_{l\to \infty} f(\xv_{\mu(n_{p_l})}) 
=\limsup_{l\to \infty}f(\xv_{\mu(n_{p_l})})
\leq f^*
$$ 
and so $\overline{\xv}\in \mathcal{S}$.
Note that the above argument applies to every subsequence of
$\{\xv_{\mu(n)}\}_{n=m_0}^\infty$. In particular, applying it to
$\{\xv_{\mu(n_j)}\}_{j=0}^\infty$, we obtain that there is a subsequence
$\{\xv_{\mu(n_{j_s})}\}_{s=0}^\infty$ that converges an optimal solution.
On the other hand, for each $j\in\mathbb{N}_0$, we have $n_j\ge m_0$ and
$V_{n_j}(\xv^*)<V_{n_j+1}(\xv^*)$, so the definition of $\mu$ gives
$\mu(n_j)=n_j$. Hence $\xv_{\mu(n_{j_s})}=\xv_{n_{j_s}}$ for all
$s\in\mathbb{N}_0$, that is, $\{\xv_{n_{j_s}}\}_{s=0}^\infty$ is a
subsequence of $\{\xv_n\}_{n=0}^\infty$ that converges to an optimal solution.
Therefore, from {\bf Case 1} and {\bf Case 2}, we conclude that there exists a subsequence of $\{\xv_n\}_{n=0}^\infty$ that converges to an optimal solution in $\mathcal{S}$.
This completes the proof.    
\end{proof}
}

By imposing a stronger assumption on the objective function $f$ to be strictly convex, we can obtain the convergence of the whole sequence $\{\xv_n\}_{n=0}^\infty$ to the unique optimal solution of the problem (\ref{main-pb}) provided that $\mathcal{S}\neq\emptyset$.

{
\begin{corollary}\label{cor:main}
  Let  $\{\xv_n\}_{n=0}^\infty$ be a sequence generated by the \FDSM{} and $a\in (0,1)$.  Suppose that Assumptions \ref{assumption_tau} and \ref{A: xn bbd} hold and  $\mathcal{S}\neq\emptyset$. If the function $f$ is strictly convex and the sequence $\{\alpha_n\}_{n=0}^\infty\subset (0,\infty)$   
satisfies that $\lim_{n\to \infty}\alpha_n=0$ and $\sum_{n=0}^\infty \alpha_n = \infty$, then the sequence $\{\xv_n\}_{n=0}^\infty$ converges to the  unique optimal solution of the problem $(\ref{main-pb})$.
\end{corollary}

\begin{proof}
Suppose that $f$ is strictly convex.
Since $\mathcal{S} \neq \emptyset,$ the strict convexity of $f$ implies that  the problem (\ref{main-pb}) has the unique solution, denoted by $\xv^*,$ i.e., $\mathcal{S}=\{\xv^*\}.$

Now, based on the proving lines of Theorem \ref{thm:con}, we will divide the proof into two cases as follows:
\vskip 2mm

{\bf Case 1:} 
We first note that, for all $n\in \mathbb{N}_0$,
$$  
\|\xv_n-\xv^*\|^2= V_n(\xv^*) - \frac{1}{4(\tau+1)}\sum_{i=1}^{\tau+1} (\tau+2-i)\|\xv_{n-i+1}-\xv_{n-i}\|^2.
$$
By using \eqref{tt2.1}, we get 
$$
\lim_{n\to \infty}\sum_{i=1}^{\tau+1} (\tau+2-i)\|\xv_{n-i+1}-\xv_{n-i}\|^2=0.
$$
Since $\lim_{n\to \infty}V_n(\xv^*)$ exists, it follows that $\lim_{n\to \infty}\|\xv_n-\xv^*\|$ exists.
Since we obtain in {\bf Case 1} of Theorem \ref{thm:con} that there is a subsequence $\{\xv_{n_{p_l}}\}_{l=0}^\infty$ of $\{\xv_n\}_{n=0}^\infty$ such that $\lim_{l\to \infty}\xv_{n_{p_l}}=\xv^*$, i.e., $\lim_{l\to \infty}\|\xv_{n_{p_l}}-\xv^*\|=0$, it follows that 
$$
\lim_{n\to \infty}\|\xv_n-\xv^*\|= \lim_{l\to \infty}\|\xv_{n_{p_l}}-\xv^*\|=0.
$$
Thus the whole sequence $\{\xv_n\}_{n=0}^\infty$ converges to $\xv^*$. 
\vskip 1mm

{\bf Case 2:}
According to the obtained result in {\bf Case 2} of Theorem \ref{thm:con} that, for any subsequence $\{\xv_{\mu(n_p)}\}_{p=0}^\infty$ of $\{\xv_{\mu(n)}\}_{n=m_0}^\infty$, there exists a subsequence $\{\xv_{\mu(n_{p_l})}\}_{l=0}^\infty$ of $\{\xv_{\mu(n_p)}\}_{p=0}^\infty$ such that 
$\lim_{l\to \infty}\xv_{\mu(n_{p_l})}=\xv^*$.

Now, let $\{\xv_{\mu(n_{p_k})}\}_{k=0}^\infty$ be any subsequence of $\{\xv_{\mu(n_p)}\}_{p=0}^\infty$ such that 
$\lim_{k\to \infty}\xv_{\mu(n_{p_k})}=\xv'\in \mathbb{R}^d$.
In a similar fashion in showing that $\overline{\xv} \in \mathcal{S}$ as in Theorem \ref{thm:con}, we also have  $\xv'\in \mathcal{S}$, which yields that $\xv'=\xv^*$. 
Hence the boundedness of $\{ \xv_{\mu(n_p)}\}_{p=0}^\infty$ ensures that $\displaystyle \lim_{p\to \infty}\xv_{\mu(n_p)} = \xv^*$.
Since $\{ \xv_{\mu(n_p)}\}_{p=0}^\infty$ is an arbitrary subsequence of $\{\xv_{\mu(n)}\}_{n=m_0}^\infty$ and the sequence $\{\xv_{\mu(n)}\}_{n=m_0}^\infty$ is bounded, we obtain $\lim_{n\to \infty}\xv_{\mu(n)}=\xv^*$, i.e., $\lim_{n\to \infty}\|\xv_{\mu(n)}-\xv^*\|=0$.
For all $n\geq m_0$, we note that
$$
\|\xv_{\mu(n)+1}-\xv^*\| \leq \|\xv_{\mu(n)+1}-\xv_{\mu(n)}\|+ \|\xv_{\mu(n)}-\xv^*\|.
$$
Taking $n\to \infty$ in the above inequality, we also get $\lim_{n\to \infty}\|\xv_{\mu(n)+1}-\xv^*\|=0$. For all $n\geq m_0$, it follows from 
\eqref{LL6} that
\begin{align*}
   \frac{1}{4(\tau+1)}\sum_{i=1}^{\tau+1}\|\xv_{\mu(n)-i+1}-\xv_{\mu(n)-i}\|^2 
   \leq  \|\xv_{\mu(n)+1}-\xv^*\|^2 + \frac{\|\xv_{\mu(n)+1}-\xv_{\mu(n)}\|^2}{4}
\end{align*}
and so
\begin{align*}
V_{\mu(n)}(\xv^*)
&=\|\xv_{\mu(n)}-\xv^*\|^2 + \frac{1}{4(\tau+1)}\sum_{i=1}^{\tau+1} (\tau+2-i)\|\xv_{\mu(n)-i+1}-\xv_{\mu(n)-i}\|^2 \nonumber\\
&\leq \|\xv_{\mu(n)}-\xv^*\|^2 +\frac{1}{4(\tau+1)}\sum_{i=1}^{\tau+1} (\tau+1)\|\xv_{\mu(n)-i+1}-\xv_{\mu(n)-i}\|^2 \nonumber\\
&\leq \|\xv_{\mu(n)}-\xv^*\|^2+ (\tau+1)\Big(\|\xv_{\mu(n)+1}-\xv^*\|^2 + \frac{\|\xv_{\mu(n)+1}-\xv_{\mu(n)}\|^2}{4}\Big).
\end{align*}
Thus, taking $n\to \infty$, we obtain  $\lim_{n\to \infty}V_{\mu(n)}(\xv^*) =0$.
Also, it follows from \eqref{tt1} that 
\begin{align*}
 V_{\mu(n)+1}(\xv^*) 
\leq V_{\mu(n)}(\xv^*)  +(2+8(\tau+1)^2)C^2\alpha_{\mu(n)}^2 +2L\alpha_{\mu(n)}(\|\xv^*\|+M),\,\,\,\forall n\geq  m_0.
\end{align*}
Also, taking $n\to \infty$, we have $\lim_{n\to \infty}V_{\mu(n)+1}(\xv^*)=0$.
Thus, from Lemma \ref{prop1}, we have 
    $$
    \|\xv_n-\xv^*\|^2 \leq V_n(\xv^*) \leq V_{\mu(n)+1}(\xv^*),\,\,\,\forall n\geq m_0.
    $$
Hence, by taking the limit as $n\to \infty$, 
 it follows that $\lim_{n\to \infty}\Vert \xv_n-\xv^*\Vert =0$,
which  implies that the whole $\{\xv_n\}_{n=0}^\infty$ converges to $\xv^*$. This completes the proof.   
\end{proof}
}

\section{Inexact Version}\label{inexact}

\quad\, In this subsection, we extend the \FDSM{} to a more general version by using an approximate subgradient, which is useful in practice when exact subgradients are difficult to compute. Then we provide some convergence results of generalized methods by the following \FDSM{}.
\vskip 2mm

First, we recall the definition and useful fact of $\epsilon$-subgradient (see the Zalinescu's book \cite{Z-02} for more details).
\vskip 1mm

Let $f:\mathbb{R}^d \to \mathbb{R}$ be convex, let $\xv\in \mathbb{R}^d$  and $\varepsilon \geq 0$. Then the {\it $\varepsilon-$subdifferential} of $f$ at $\xv$ is the set
\begin{align*}
    \partial_\varepsilon f(\xv) := \{\gv\in \mathbb{R}^d : \langle \gv,\yv-\xv \rangle \leq f(\yv)-f(\xv)+\varepsilon,\, \forall  \yv \in \mathbb{R}^d\}.
\end{align*}

The point $\gv\in \partial_\varepsilon f(\xv)$ is called an {\it $\varepsilon-$subgradient of $f$} at $\xv$, which is denoted by $\widetilde{\nabla}_\varepsilon f(\xv).$
It is obvious that $ \partial f(\xv) \subset \partial_\varepsilon f(\xv).$
The convexity of $f$ ensures that $ \partial_\varepsilon f(\xv) \neq \emptyset$ for all $\xv\in \mathbb{R}^d$ and $\varepsilon \geq 0$ \cite[Theorem 2.4.9]{Z-02}. 
Moreover, for any nonempty bounded set $X\subset \mathbb{R}^d,$ also, it follows that $\bigcup_{\xv\in X}\partial_\varepsilon f(\xv)$ is nonempty and bounded \cite[Theorem 2.4.13]{Z-02}.
\vskip 2mm

We also assume the boundedness of delayed sequence as Assumption \ref{assumption_tau}, i.e., there exists a $\tau \geq 0$ such that
$ 0 \leq \tau_n \leq \tau$ for all $n \in \mathbb{N}_0$.
\vskip 2mm

An inexact version of the \FDSM{} can be defined as follows.
\vskip2mm

\begin{algorithm}[H]
	\SetAlgoLined
	\vskip2mm
	\textbf{Initialization}: Given a stepsize $\{\alpha_n\}_{n=0}^\infty \subset (0,\infty)$, $\{ \varepsilon_n\}_{n=0}^\infty \subset [0,\infty)$ and
     initial points $\xv_0,\xv_{-1},\ldots,\xv_{-\tau}\in \mathbb{R}^d.$
	
	\textbf{Iterative Step}: For a current point $\xv_n\in \mathbb{R}^d$, we compute
    $$
\xv_{n+1}:= T\xv_n-\alpha_n {\widetilde{\nabla}}_{\varepsilon_{n-\tau_n}} f(T\xv_{n-\tau_n}),\,\,\,\forall n \in \mathbb{N}_0,
$$
where $\widetilde{\nabla}_{\varepsilon_{n-\tau_n}}f(T\xv_{n-\tau_n})$ is an $\varepsilon_{n-\tau_n}$-subgradient of $f$ at $T\xv_{n-\tau_n}$.
 
	{\bf Update} $n:=n+1$.
	\caption{\bf Fixed-Point Delayed Approximate Subgradient Method (FDASM)}
	\label{alg:FDASM}
		\vskip2mm
\end{algorithm}
\vskip 2mm

We assume the boundedness as the following assumption:

\begin{assumption}\label{A: xn epsilon bbd}
    Let $\{\xv_n\}_{n=0}^\infty$ be a sequence generated by the \FDASM{}. There exists $M_2>0$ such that $\| \xv_n \| \leq M_2$ for all $n\in \mathbb{N}_0.$
\end{assumption}

The convergent results of the sequence generated by the \FDASM{} can be derived in the same manner as the previous subsections. First, we provide the following proposition, which is a consequence from Assumption \ref{A: xn epsilon bbd} and will play a crucial role in proving  the convergences. 

\begin{proposition}\label{LL-ep1}
    Let $\{\xv_n\}_{n=0}^\infty$ be a sequence generated by the \FDASM{}.
    Suppose that 
{ $\sup_{n\in \mathbb{N}_0}\varepsilon_n <\infty$} and Assumption \ref{A: xn epsilon bbd} holds. Then we have the following:
    \begin{enumerate}
        \item[{\rm (1)}] There exists $C_2>0$ such that
   $$
   \|\widetilde{\nabla}_{\varepsilon_n}f(T\xv_n)\| \leq C_2,\,\,\,\forall n\in \mathbb{N}_0.
   $$
        \item[{\rm (2)}] There exists $L>0$ such that 
        $$
        f(\xv)-f(\yv)\leq L\|\xv-\yv\|,\,\,\,\forall \xv, \yv \in \{\xv_n, T\xv_n : n\in \mathbb{N}_0\}.
        $$
    \end{enumerate}
\end{proposition}

\begin{proof}
Let $\xv\in \fix(T)$ be fixed.
Then, for any $n\in \mathbb{N}_0$, it follows from the nonexpansiveness of $T$ and Assumption \ref{A: xn epsilon bbd} that
$$
\|T\xv_n\| = \|T\xv_n-\xv + \xv\|\leq \|\xv_n-\xv\|+\|\xv\|\leq \|\xv_n\|+2\|\xv\|= M_2 + 2\|\xv\|,
$$
which shows that the sequence $\{T\xv_n\}_{n=0}^\infty$ is bounded.

{
To prove (1), let $B= \sup_{n\in \mathbb{N}_0}\varepsilon_n <\infty$.
For any $n\in \mathbb{N}_0$, since $\varepsilon_n \leq B$, we obtain 
$$
\partial_{\varepsilon_n}f(T\xv_n)\subseteq \partial_Bf(T\xv_n).
$$
The boundedness of $\{T\xv_n\}_{n=0}^\infty$ implies that $\bigcup_{n\in \mathbb{N}_0}\partial_Bf(T\xv_n)$ is bounded \cite[Theorem 2.4.13]{Z-02} and so $\bigcup_{n\in \mathbb{N}_0}\partial_{\varepsilon_n}f(T\xv_n)$ is bounded.
Thus there exists a constant $C_2>0$ such that 
$$\|\widetilde{\nabla}_{\varepsilon_n} f(T\xv_n)\| \leq C_2,\,\,\,\forall n\in \mathbb{N}_0.$$
}

To prove (2), we have that the sequences $\{\xv_n\}_{n=0}^\infty$ and $\{T\xv_n\}_{n=0}^\infty$ are bounded. Consequently, the set $\{\xv_n, T\xv_n : n\in \mathbb{N}_0\}$ is bounded.
According to Fact \ref{f-Lip}, $f$ is Lipschitz continuous over the set $\{\xv_n, T\xv_n : n\in \mathbb{N}_0\}$.
This implies that there exists a constant $L>0$ such that
$$
f(\xv)-f(\yv)\leq L\|\xv-\yv\|,\,\,\,\forall \xv,\yv \in \{\xv_n, T\xv_n: n\in \mathbb{N}_0\}.
$$
This completes the proof.
\end{proof}

The following lemma is obtained by a slight modification of the line proof of Lemma \ref{mainlemma}.
Indeed, we proceed under the boundedness of a  generated sequence and subsequently apply the Lipschitz continuity of $f$.

\begin{lemma}\label{mainlemma-en}
Let  $\{\xv_n\}_{n=0}^\infty$ be a sequence generated by the \FDASM{} and $a\in (0,1)$.  
Suppose that Assumptions \ref{assumption_tau} and \ref{A: xn epsilon bbd} hold.
Then, for all $n\in\mathbb{N}_0$ and  $ \xv \in \fix(T),$ we have
\begin{align*}
      2\alpha_n(f(\xv_{n+1})-f(\xv)) 
    \leq& \|\xv_n-\xv\|^2 - \|\xv_{n+1}-\xv\|^2 +2\alpha_n\varepsilon_{n-\tau_n}  + 2LC_2\alpha_n^2 + (16L^2+24C_2^2)\alpha_n^{2-a}  \\
    &+\Big(\frac{3\alpha_n^a}{8}-1\Big)\|\xv_{n+1}-\xv_n\|^2 + \frac{2(\tau+1)\alpha_n^a}{8}\sum_{i=0}^\tau\|\xv_{n-i+1}-\xv_{n-i}\|^2,
\end{align*}
where $C_2$ and $L$ are given in Proposition \ref{LL-ep1}.
\end{lemma}

\begin{proof}
Let $n\in\mathbb{N}_0$  and $\xv\in \fix(T)$.
Following the idea in Lemma \ref{mainlemma} together with Proposition \ref{LL-ep1} (1), we have 
\begin{align}\label{L2en}
    \|\xv_{n+1}-\xv\|^2
    &\leq \|\xv_{n}-\xv\|^2 - \|\xv_{n+1}-\xv_n\|^2 + 24\alpha_n^{2-a}\| \widetilde{\nabla}_{\varepsilon_{n-\tau_n}}f(T\xv_{n-\tau_n})\|^2  +\frac{2\alpha_n^a}{8}\|\xv_{n+1}-\xv_n\|^2\nonumber\\
    &\quad  + \frac{\alpha_n^a}{8}\|\xv_{n+1}-\xv_{n-\tau
    _n}\|^2  + 2\alpha_n \langle \widetilde{\nabla}_{\varepsilon_{n-\tau_n}}f(T\xv_{n-\tau_n}), \xv - T\xv_{n-\tau_n} \rangle \nonumber\\
    &\leq \|\xv_{n}-\xv\|^2 - \|\xv_{n+1}-\xv_n\|^2 + 24\alpha_n^{2-a}C_2^2  +\frac{2\alpha_n^a}{8}\|\xv_{n+1}-\xv_n\|^2 \nonumber\\
    & \quad + \frac{\alpha_n^a}{8}\|\xv_{n+1}-\xv_{n-\tau
    _n}\|^2  + 2\alpha_n \langle \widetilde{\nabla}_{\varepsilon_{n-\tau_n}}f(T\xv_{n-\tau_n}), \xv - T\xv_{n-\tau_n} \rangle
\end{align}
It follows from the definition of approximate subgradient of $f$ at $T\xv_{n-\tau_n}$ that
{\small \begin{align}\label{L5en}
2\alpha_n \langle \widetilde{\nabla}_{\varepsilon_{n-\tau_n}}f(T\xv_{n-\tau_n}),\xv-T\xv_{n-\tau_n} \rangle
&\leq 2\alpha_n(f(\xv) - f(\xv_{n+1})) +2\alpha_n\varepsilon_{n-\tau_n}+ 2\alpha_n(f(\xv_{n+1})-f(T\xv_n)) \nonumber\\
&\quad+ 2\alpha_n(f(T\xv_n)-f(T\xv_{n+1}))  + 2\alpha_n(f(T\xv_{n+1}) - f(T\xv_{n-\tau_n})). 
\end{align}}
From Proposition \ref{LL-ep1}, there exists $L>0$ such that
\begin{align}\label{L8en-0}
f(\xv_{n+1})-f(T\xv_n) 
\leq L\|\xv_{n+1}-T\xv_n\| 
= L\alpha_n\|\widetilde{\nabla}_{\varepsilon_{n-\tau_n}}f(T\xv_{n-\tau_n})\| \leq LC_2\alpha_n.
\end{align}
Invoking Young's inequality, we obtain
\begin{align*}
    8L\alpha_n\|\xv_n-\xv_{n+1}\|=8L\alpha_n^{\frac{2-a}{2}}\alpha_n^\frac{a}{2}\|\xv_n-\xv_{n+1}\|
    \leq \frac{64L^2\alpha_n^{2-a}}{2} + \frac{\alpha_n^2}{2}\|\xv_n-\xv_{n+1}\|^2,
\end{align*}
which is 
\begin{align*}
    2L\alpha_n\|\xv_n-\xv_{n+1}\| \leq 8L^2\alpha_n^{2-a}+\frac{\alpha_n^a}{8}\|\xv_n-\xv_{n+1}\|^2.
\end{align*}

Similarly, we get
$$
2L\alpha_n\|\xv_{n+1}-\xv_{n-\tau_n}\|\leq 8L^2\alpha_n^{2-a}+\frac{\alpha_n^a}{8}\|\xv_{n+1}-\xv_{n-\tau_n}\|^2.
$$
Thus, using Proposition \ref{LL-ep1} (2), the nonexpansiveness of $T$ and the above results, we obtain the following two inequalities:
\begin{align}\label{L6en}
2\alpha_n( f(T\xv_n)-f(T\xv_{n+1}))
&\leq 2L\alpha_n\|T\xv_n-T\xv_{n+1}\| \nonumber\\
&\leq 2L\alpha_n\|\xv_n-\xv_{n+1}\|  \nonumber\\
&\leq 8L^2\alpha_n^{2-a}+\frac{\alpha_n^a}{8}\|\xv_n-\xv_{n+1}\|^2
\end{align}
and
\begin{align}\label{L7en}
   2\alpha_n(f(T\xv_{n+1})-f(T\xv_{n-\tau_n})) 
   &\leq 2L\alpha_n\|T\xv_{n+1}-T\xv_{n-\tau}\| \nonumber\\
   &\leq 2L\alpha_n\|\xv_{n+1}-\xv_{n-\tau_n}\|\nonumber\\
   &\leq 8L^2\alpha_n^{2-a}+\frac{\alpha_n^a}{8}\|\xv_{n+1}-\xv_{n-\tau_n}\|^2.
\end{align}
By substituting \eqref{L8en-0}, \eqref{L6en}, \eqref{L7en} into \eqref{L5en}, we get
\begin{align}\label{L8en}
    2\alpha_n\langle \widetilde{\nabla}_{\varepsilon_{n-\tau_n}}f(T\xv_{n-\tau_n}),\xv-T\xv_{n-\tau_n}\rangle
   & \leq 2\alpha_n(f(\xv) - f(\xv_{n+1})) +2\alpha_n\varepsilon_{n-\tau_n} +2LC_2\alpha_n^2  +16L^2\alpha_n^{2-a}\nonumber\\ 
    &\quad +\frac{\alpha_n^a}{8}\|\xv_n-\xv_{n+1}\|^2  +\frac{\alpha_n^a}{8}\|\xv_{n+1}-\xv_{n-\tau_n}\|^2.
\end{align}
Combining \eqref{L2en} and \eqref{L8en} leads to
\begin{align*}
    2\alpha_n(f(\xv_{n+1})-f(\xv))
    &\leq \|\xv_n-\xv\|^2 - \|\xv_{n+1}-\xv\|^2 +2\alpha_n\varepsilon_{n-\tau_n} + 2LC_2\alpha_n^2 + 16L^2\alpha_n^{2-a}+24\alpha_n^{2-a}C_2^2  \\
    &\quad+\left(\frac{3\alpha_n^a}{8}-1\right)\|\xv_{n+1}-\xv_n\|^2 + \frac{2\alpha_n^a}{8}\|\xv_{n+1}-\xv_{n-\tau_n}\|^2 \nonumber\\
    &\leq \|\xv_n-\xv\|^2 - \|\xv_{n+1}-\xv\|^2 +2\alpha_n\varepsilon_{n-\tau_n}  + 2LC_2\alpha_n^2 + (16L^2+24C_2^2)\alpha_n^{2-a}  \\
    &\quad+\left(\frac{3\alpha_n^a}{8}-1\right)\|\xv_{n+1}-\xv_n\|^2 + \frac{2(\tau+1)\alpha_n^a}{8}\sum_{i=0}^\tau\|\xv_{n-i+1}-\xv_{n-i}\|^2.
\end{align*}
This completes the proof. 
\end{proof}

We can derive an explicit  upper bound of the gap between the sequence $\left\{\min_{0\leq n \leq N}f(\xv_{n+1})\right\}_{N=0}^\infty$ of best-archieved function values and the optimal value $f^*$ as in the following theorem.
\vskip 2mm

\begin{theorem}\label{thm:rate epsilon}
Let  $\{\xv_n\}_{n=0}^\infty$ be a sequence generated by the \FDASM{} and $a\in (0,1)$.  Suppose that Assumptions \ref{assumption_tau} and \ref{A: xn epsilon bbd} hold, 
{$\sup_{n\in \mathbb{N}_0}\varepsilon_n < \infty$}, and the sequence $\{\alpha_n\}_{n=0}^\infty\subset (0,\infty)$ is nonincreasing such that  $(3+2(\tau+1)^2)\alpha_0^a < 8$. 
Then, for all $N\in\mathbb{N}_0$ and  $ \xv^* \in \mathcal{S}$, we have
{\small \begin{align*}
    \min_{0\leq n \leq N}f(\xv_{n+1})-f^* 
   \leq\frac{\Vert \xv_0-\xv^*\Vert^2 +2LC_2\sum_{n=0}^N\alpha_n^2 + (16L^2+24C_2^2)\sum_{n=0}^N\alpha_n^{2-a} + 2\sum_{n=0}^N \alpha_n \varepsilon_{n-\tau_n}}{2\sum_{n=0}^N \alpha_n},
\end{align*}  }  
where $C_2$ and $L$ are given in Proposition  \ref{LL-ep1}.
\end{theorem}
\vskip 2mm

Next, we state a lemma which provides the necessary tool for proving the convergence of the iterations. Subsequently, we state the main theorem of this subsection. The proofs of them are  directly obtained by following the line proofs of Lemma \ref{LL}, Theorem \ref{thm:con} and Corollary \ref{cor:main}, respectively. Thus we will omit the proofs here.


{
\begin{lemma}\label{LL-en}
Let  $\{\xv_n\}_{n=0}^\infty$ be a sequence generated by the \FDASM{}.
Suppose that Assumptions \ref{assumption_tau} and \ref{A: xn epsilon bbd} hold and $\sup_{n\in \mathbb{N}_0}\varepsilon_n < \infty$.
Then, for all $n\in \mathbb{N}_0$ and $\xv\in \fix(T)$, we have
    \begin{align*}
        V_{n+1}(\xv)
        &\leq
        V_n(\xv) - \frac{\|\xv_{n+1}-\xv_n\|^2}{4} - \frac{\|\xv_n-\xv_{n-\tau_n}\|^2}{8(\tau+1)^2} + (2+8(\tau+1)^2)C_2^2\alpha_n^2 \\
         &\quad+ 2\alpha_n( f(\xv)-f(T\xv_{n-\tau_n}))
          + 2\alpha_n\varepsilon_{n-\tau_n},
    \end{align*}
    where $V_n(\xv)$ is defined in $(\ref{vn-def})$ and  
    $C_2$ is given in Proposition \ref{LL-ep1}.
\end{lemma}
}
\vskip 2mm

\begin{theorem}\label{thm:con epsilon}
Let  $\{\xv_n\}_{n=0}^\infty$ be a sequence generated by the \FDASM{}.  
Suppose that Assumptions \ref{assumption_tau}, \ref{A: xn epsilon bbd} hold and assume that $\mathcal{S}\neq\emptyset$. 
Suppose that  the sequence $\{\alpha_n\}_{n=0}^\infty\subset (0,\infty)$ satisfies $\lim_{n\to \infty}\alpha_n=0$ and $\sum_{n=0}^\infty\alpha_n =\infty$  and the sequence $\{\varepsilon_n\}_{n=0}^\infty \subset [0,\infty)$ satisfies { $\lim_{n\to \infty}\varepsilon_n=0$.}
  Then there exists a subsequence of $\{\xv_n\}_{n=0}^\infty$ that converges to an optimal solution in $\mathcal{S}$. Furthermore, if $f$ is strictly convex, then the sequence of $\{\xv_n\}_{n=0}^\infty$ converges to the unique optimal solution of the problem $(\ref{main-pb})$.
\end{theorem}

\section{Distributed Version}\label{distributed}

\quad\, In this section, we investigate a distributed version of the \FDSM{} to deal with a particular situation of  problem (\ref{main-pb}) where the objective function is the sum of a finite number of convex functions and the constraint set is the intersection of a finite number of fixed-point sets:
\begin{align}\label{main-pb-sum}
	\begin{array}{ll}
		\textrm{minimize }\indent \displaystyle f(\xv):= \sum_{j=1}^m f_j(\xv)\\
		\textrm{subject to}\indent \xv\in \displaystyle \bigcap_{j=1}^m \fix(T_j),
	\end{array}%
\end{align}	
where, for all $j=1,\ldots,m$,   $f_j:\mathbb{R}^d \to \mathbb{R}$ is a convex function and $T_j: \mathbb{R}^d \to \mathbb{R}^d$ is a firmly nonexpansive operator with $\bigcap_{j=1}^m \fix(T_j) \neq \emptyset$. 
\vskip 1mm

In view of the network system of $m$ individual workers and a central server, we assume that each worker can communicate only with the central server.
Actually, this problem setting can be form as  the problem (\ref{main-pb}) since the sum of convex functions  $f_j$ is also convex. Moreover, one can define the operator $T:=\frac{1}{m}\sum_{j=1}^m T_j$ which is also a firmly nonexpansive operator with $\fix(T) =  \bigcap_{j=1}^m \fix(T_j)$ \cite[Corollary 2.2.20]{C-12}. However, the computations of a current estimate which is relating to the operator $T$ and the subgradients of the function $f$ can indeed be computationally intensive, especially when the problem involves high-dimensional data $d$ or a large number of workers $m$. 

In this situation, the distributed scheme may offer an efficient approach to tackling such drawbacks by allowing  each individual worker compute the iterates $\xv_{n,j}$ by using the current iterate $\xv_n$ or even the staled iterate which are received from the central server and their own information $T_j$ and $f_j$. Subsequently, each worker sends their update iterate $\xv_{n,j}$ to the central server for updating the next iterate $\xv_{n+1}$. The formal algorithm will be presented shortly (see the \DFDSM{} below).
\vskip 2mm

 We now suppose that the  problem (\ref{main-pb-sum}) has an optimal solution and denotes its solution set by $\mathcal{S}_m^{*}$.
It is akin to the above presented results, we will start by the boundedness assumption on the delayed sequences. 

\begin{assumption}\label{assumption_tau:Dis}
For each worker $j=1,\ldots, m$, the delay sequence $\{\tau_n^j\}_{n=0}^\infty$ is bounded, i.e.,  there exists a nonnegative integer $\tau^j$ such that
$0 \leq \tau_n^j \leq \tau^j$ for all $n\in \mathbb{N}_0$.
\end{assumption}

For simplicity, we set $ \tau_{\max} := \max_{1 \leq j \leq m} \tau^j$
and so  
$$
0 \leq \tau_n^j \leq \tau_{\max},\,\,\,\forall n\in \mathbb{N}_0,\, j=1, \ldots, m.
$$

Now, we are in a position to present the distributed fixed-point delayed subgradient method for solving the problem (\ref{main-pb-sum}).
\vskip2mm

\begin{algorithm}[H]
	\SetAlgoLined
	\vskip2mm
	\textbf{Initialization}: Given a stepsize $\{\alpha_n\}_{n=0}^\infty \subset (0,\infty)$ and  initial points $\xv_0,\xv_{-1},\cdots,\xv_{-\tau_{\max}}\in \mathbb{R}^d.$
	
	\textbf{Iterative Step}: The central server transmit the current point $\xv_n\in \mathbb{R}^d$ to all workers. Each worker $j\in\{1,\ldots,m\}$  computes
    $$
        \xv_{n,j}:= T_j\xv_n-\alpha_n {\widetilde{\nabla}} f_j(T_j\xv_{n-\tau_n^j}),,\,\,\,\forall n\in \mathbb{N}_0,\, j=1, \ldots, m,
    $$
where $\widetilde{\nabla}f_j(T_j \xv_{n-\tau_n^j})$ is a subgradient of $f_j$ at $T_j\xv_{n-\tau_n^j}$ and sends $\xv_{n,j}$ to the central server. 

The central server computes the next iterate $\xv_{n+1}$ as
    $$
        \xv_{n+1} := \frac{1}{m}\sum_{j=1}^m \xv_{n,j},\,\,\,\forall n\in \mathbb{N}_0.
    $$
	{\bf Update} $n:=n+1$.
	\caption{\bf  Distributed Fixed-Point Delayed Subgradient Method (DFDSM)}
	\label{alg:DFDSM}
		\vskip2mm
\end{algorithm}
\vskip2mm

\begin{remark}
    If the delays $\tau_n^j = 0$ for all $n\in \mathbb{N}_0,\, j=1,\ldots,m$, the \DFDSM{} reduces to Algorithm 3.1 of Iiduka \cite{I-16}. It should be clearly state that the assumption of the operators $T_j$ considered in Iiduka's work \cite{I-16} was assumed to be quasi-nonexpansive, whereas, according to the proving requirements, we need to assume here that the firm nonexpansivity of each $T_j$.
\end{remark}

\begin{assumption}\label{A: xn bbd:Dis}
    Let $\{\xv_n\}_{n=0}^\infty$ be a sequence generated by the \DFDSM{}.
    For each $j=1,\ldots,m$, there exists  $M_j>0$ such that $\| \xv_{n,j}\| \leq M_j$ for all $n\in \mathbb{N}_0$.
\end{assumption}

 Again, by letting $\displaystyle M_{\max} := \max_{1\leq j\leq m} M_j$, we obtain that 
    $\| \xv_{n,j} \| \leq M_{\max}$ for all $n\in \mathbb{N}_0, j=1, \ldots, m.$
    Moreover, we also obtain the boundedness of the sequence $\{\xv_n\}_{n=0}^\infty$ through the fact that
    $$
    \| \xv_{n+1}\| = \Big\| \frac{1}{m}\sum_{j=1}^m \xv_{n,j} \Big\| \leq M_{\max},\,\,\,\forall n\in \mathbb{N}_0.
    $$
\vskip2mm

Assumption \ref{A: xn bbd:Dis} leads to the next results.

\begin{proposition}\label{LL-d1}
    Let $\{\xv_n\}_{n=0}^\infty$ be a sequence generated by the \DFDSM{}.
    Suppose that Assumption \ref{A: xn bbd:Dis} holds. 
    Then, for each $j=1, \ldots, m$, we have the following:
    \begin{enumerate}
        \item[{\rm (1)}] There exists $C_j>0$ such that 
        $$
        \max\{ \|\widetilde{\nabla}f_j(\xv_n)\|, \|\widetilde{\nabla}f_j(\xv_{n,j})\|, \|\widetilde{\nabla}f_j(T_j\xv_n)\|, \|\widetilde{\nabla}f_j(T_j\xv_{n,j})\|\} \leq C_j,\,\,\,\forall  n\in \mathbb{N}_0.
        $$ 
        In particular, we denote $C_{\max} := \max_{1\leq j \leq m}C_j$.
        \item[{\rm (2)}] There exists $L_j>0$ such that
        $$
        f_j(\xv)-f_j(\yv) \leq L_j\|\xv-\yv\|,\,\,\,\forall  \xv, \yv \in \{\xv_n, \xv_{n,j}, T_j\xv_n: n \in \mathbb{N}_0\}.
        $$
        In particular, we denote $L_{\max} := \max_{1\leq j \leq m}L_j$.
    \end{enumerate}
\end{proposition}

\begin{proof}
    Let $j=1, \ldots, m$.
    It follows from Assumption \ref{A: xn bbd:Dis} that the sequence
    $\{\xv_{n,j}\}_{n=0}^\infty$ and $\{\xv_n\}_{n=0}^\infty$ are bounded.
    Let {$\xv\in \fix(T_j)$ be fixed.}     Since $T_j$ is nonexpansive, it follows that
    $$
    \|T_j\xv_n\|
    = \|T_j\xv_n-\xv+\xv\|
    \leq \|\xv_n-\xv\|+\|\xv\|
    \leq \max\{M_{\max}, \|\xv_0\|\} + 2\|\xv\|,\,\,\,\forall n\in \mathbb{N}_0,
    $$
     and so $\{T_j\xv_n\}_{n=0}^\infty$ is bounded.
     
    {Similarly, we have 
    $$
    \|T_j\xv_{n,j}\| \leq M_{\max} + 2\|\xv\|,\,\,\,\forall n\in \mathbb{N}_0,
    $$
     and so $\{T_j\xv_{n,j}\}_{n=0}^\infty$ is also bounded.}
    
To prove (1), since $\{\xv_n\}_{n=0}^\infty$, $\{\xv_{n,j}\}_{n=0}^\infty$, $\{T_j\xv_n\}_{n=0}^\infty$ and $\{T_j\xv_{n,j}\}_{n=0}^\infty$ are bounded, it follows from Fact \ref{sg-bbd} (2) that $\{\widetilde{\nabla}f(\xv_n)\}_{n=0}^\infty,$ $\{\widetilde{\nabla}f(\xv_{n,j})\}_{n=0}^\infty$, $\{\widetilde{\nabla}f(T_j\xv_n)\}_{n=0}^\infty$ and $\{\widetilde{\nabla}f(T_j\xv_{n,j})\}_{n=0}^\infty$ are bounded. Thus there exists a constant $C_j>0$ such that
$$
\max\{\|\widetilde{\nabla}f_j(\xv_n)\|, \|\widetilde{\nabla}f_j(\xv_{n,j})\|,
\|\widetilde{\nabla}f_j(T_j\xv_n)\|,
\|\widetilde{\nabla}f(T_j\xv_{n,j})\|\} \leq C_j,\,\,\,\forall  n\in \mathbb{N}_0.
$$

To prove (2), since  $\{\xv_n\}_{n=0}^\infty$, $\{T_j\xv_n\}_{n=0}^\infty$ and $\{\xv_{n,j}\}_{n=0}^\infty$ are bounded, it follows that  the set $\{\xv_n, T_j\xv_n, \xv_{n,j} : n\in \mathbb{N}_0\}$ is bounded.
This implies the Lipschitz continuity of  $f_j$ over the set $\{\xv_n, T_j\xv_n, \xv_{n,j} : n\in \mathbb{N}_0\}$ from Fact \ref{f-Lip}, i.e., there exists $L_j>0$ such that
{
$$
f_j(\xv)-f_j(\yv)\leq L_j\|\xv-\yv\|,\,\,\,\forall \xv, \yv \in \{\xv_n, T_j\xv_n, \xv_{n,j} : n\in \mathbb{N}_0\}.
$$}
This completes the proof.
\end{proof}    

In the same manner as the above results, we provide the following technical lemma. The proof of this lemma follows the same lines as Lemma \ref{mainlemma} with slight modifications by proceed under the boundedness assumption of the sequence $\{\xv_{n,j}\}_{n=0}^\infty$ (for all $j=1, \ldots, m$) in order to invoking the Lipschitz continuity of each $f_j$.

 More precisely, we utilize the Lipschitz continuity of $f$ to derive the sequence of the best-archieved function values $\left\{ \min_{0\leq n \leq N} f(\xv_n) \right\}_{N=0}^\infty$.

\begin{lemma}\label{mainlemma:Dis}
Let  $\{\xv_n\}_{n=0}^\infty$ be a sequence generated by the \DFDSM{} and $a\in (0,1)$.  Suppose that Assumption \ref{assumption_tau:Dis} and \ref{A: xn bbd:Dis} hold. Then, for all $n\in\mathbb{N}_0$ and  $ \xv \in \bigcap_{j=1}^m\fix(T_j),$ we have
{\small \begin{align*}
    \frac{2\alpha_n}{m}(f(\xv_n)-f(\xv)) 
    &\leq \|\xv_n-\xv\|^2 - \| \xv_{n+1}-\xv\|^2 
    +\left(\frac{8\alpha_n^a}{8}-1\right)\frac{1}{m}\sum_{j=1}^m \|\xv_{n}-\xv_{n,j}\|^2    \nonumber \\
    &\quad +2\alpha_n^2C_{\max}^2 + 40\alpha_n^{2-a}C_{\max}^2 +8L_{\max}^2\alpha_n^{2-a} 
    +\frac{4\alpha_n^a(\tau_{\max}+1)}{8}\sum_{i=0}^{\tau_{\max}} \Vert \xv_{n-i+1}-\xv_{n-i}\Vert^2,
\end{align*}  }
where $C_{\max}$ and $L_{\max}$ are given in Proposition \ref{LL-d1}.
\end{lemma}

\begin{proof} 
    Let $n\in\mathbb{N}_0$  and $\xv\in \displaystyle \bigcap_{j=1}^m\fix(T_j)$. Since the proving lines of the following relation go along the lines of Lemma \ref{mainlemma} with some modifications, we will omit its proof and state here that: for each $j=1,\ldots,m$, we have
\begin{align}\label{D1}
    \|\xv_{n,j}-\xv\|^2
    &\leq \|\xv_n-\xv\|^2 
    + \left(\frac{3\alpha_n^a}{8}-1\right)\|\xv_{n,j}-\xv_n\|^2
    + \frac{2\alpha_n^a}{8}\|\xv_{n,j}-\xv_{n-\tau_n^j}\|^2 +2\alpha_n(f_j(\xv)-f_j(\xv_{n,j}))\nonumber\\
    &\quad  +2\alpha_n^2\|\widetilde{\nabla}f_j(\xv_{n,j})||\|\widetilde{\nabla}f_j(T_j\xv_{n-\tau_n^j})\|+24\alpha_n^{2-a}\|\widetilde{\nabla}f_j(T_j\xv_{n-\tau_n^j})\|^2  \nonumber\\
    &\quad +8\alpha_n^{2-a}\|\widetilde{\nabla}f_j(T_j\xv_n)\|^2 + 8\alpha_n^{2-a}\|\widetilde{\nabla}f_j(T_j\xv_{n,j})\|^2 \nonumber\\
    &\leq \|\xv_n-\xv\|^2
     + \left(\frac{3\alpha_n^a}{8}-1\right)\|\xv_{n,j}-\xv_n\|^2 +2\alpha_n(f_j(\xv)-f_j(\xv_{n,j})) 
    +2\alpha_n^2C_{\max}^2  \nonumber\\
    &\quad  + 40\alpha_n^{2-a}C_{\max}^2 +\frac{2\alpha_n^a}{8}\|\xv_{n,j}-\xv_{n-\tau_n^j}\|^2.
\end{align}
{
For each $j=1, \ldots, m$, we have
\begin{align*}
    \|\xv_{n,j}-\xv_{n-\tau_n^j}\|^2 
    &\leq 2\|\xv_{n,j}-\xv_n\|^2 + 2\|\xv_n-\xv_{n-\tau_n^j}\|^2\\
    &\leq 2\|\xv_{n,j}-\xv_n\|^2 + 2(\tau_{\max}+1)\sum_{i=0}^{\tau_{\max}}\|\xv_{n-i+1}-\xv_{n-i}\|^2.
\end{align*}}
 For each $j=1, \ldots, m$, it follows from Proposition \ref{LL-d1} (2) and Young’s inequality that
\begin{align*}
2\alpha_n(f_j(\xv)-f_j(\xv_{n,j}))
&=2\alpha_n(f_j(\xv)-f_j(\xv_n)) + 2\alpha_n(f_j(\xv_n)-f_j(\xv_{n,j}))\nonumber\\
&\leq 2\alpha_n(f_j(\xv)-f_j(\xv_n)) + 2L_{\max}\alpha_n\|\xv_n-\xv_{n,j}\| \nonumber\\
&\leq 2\alpha_n(f_j(\xv)-f_j(\xv_n)) + 8L_{\max}^2\alpha_n^{2-a}+\frac{\alpha_n^a}{8}\|\xv_n-\xv_{n,j}\|^2.
\end{align*}
Then the inequality \eqref{D1} becomes
\begin{align*}
    \|\xv_{n,j}-\xv\|^2
   &\leq \|\xv_n-\xv\|^2
     + \Big(\frac{8\alpha_n^a}{8}-1\Big)\|\xv_{n,j}-\xv_n\|^2  
    +2\alpha_n^2C_{\max}^2 + 40\alpha_n^{2-a}C_{\max}^2+8L_{\max}^2\alpha_n^{2-a} \nonumber\\
    &\quad+\frac{4\alpha_n^a(\tau_{\max}+1)}{8}\sum_{i=0}^{\tau_{\max}}\|\xv_{n-i+1}-\xv_{n-i}\|^2+2\alpha_n(f_j(\xv)-f_j(\xv_{n})).
\end{align*}
By summing the above inequality for $j=1,\ldots,m$, invoking the definition of $\displaystyle \xv_{n+1}:= \frac{1}{m}\sum_{j=1}^m \xv_{n,j}$ and  $\displaystyle f:=\sum_{j=1}^mf_j$ and using the convexity of $\| \cdot \|^2$, we obtain
{\small \begin{align}\label{L11D}
   \| \xv_{n+1}-\xv \|^2 &\leq \frac{1}{m}\sum_{j=1}^m \| \xv_{n,j}-\xv \|^2 \nonumber \\
 &\leq \| \xv_n-\xv\|^2 + \Big(\frac{8\alpha_n^a}{8}-1\Big)\frac{1}{m}\sum_{j=1}^m \|\xv_{n}-\xv_{n,j}\|^2   +2\alpha_n^2C_{\max}^2 + 40\alpha_n^{2-a}C_{\max}^2 +8L_{\max}^2\alpha_n^{2-a} \nonumber\\
    &\quad+\frac{4\alpha_n^a(\tau_{\max}+1)}{8}\sum_{i=0}^{\tau_{\max}} \Vert \xv_{n-i+1}-\xv_{n-i}\Vert^2 +\frac{2\alpha_n}{m}(f(\xv)-f(\xv_n)),
\end{align} }
which is the desired inequality. This completes the proof.
\end{proof}

Now, we can obtain that the sequence of best-achieved function values $\left\{\min_{0\leq n \leq N} f(\xv_n)\right\}_{N=0}^\infty$ approximates the optimal value $f^*$ as the following theorem. Again, since its proof is directly obtained by following the line proofs of Theorem \ref{thm:rate} { along with the inequality $\|\xv_{n+1}-\xv_n\| \leq \frac{1}{m}\sum_{j=1}^m \|\xv_n-\xv_{n,j}\|$}, we will omit the proof here.

\begin{theorem}\label{thm:rate: Distributed}
Let  $\{\xv_n\}_{n=0}^\infty$ be a sequence generated by the \DFDSM{} and $a\in (0,1)$.  
Suppose that Assumptions \ref{assumption_tau:Dis} and \ref{A: xn bbd:Dis} hold and  the sequence $\{\alpha_n\}_{n=0}^\infty\subset (0,\infty)$ is nonincreasing and satisfies 
{ $(8+4(\tau_{\max}+1)^2)\alpha_0^a < 8$.} 
Then, for all $N\in\mathbb{N}_0$ and  $ \xv^* \in \mathcal{S}_m^{*}$, we have
{\small \begin{align*}
     \min_{0\leq n \leq N}f(\xv_{n})-f^* 
    \leq \frac{m\Vert \xv_0-\xv^*\Vert^2 +2mC_{\max}^2\sum_{n=0}^N\alpha_n^2+ m(40C_{\max}^2+8L_{\max}^2)\sum_{n=0}^N \alpha_n^{2-a}}{2\sum_{n=0}^N \alpha_n},
\end{align*} }
where $C_{\max}$ and $L_{\max}$ are given in Proposition \ref{LL-d1}.
\end{theorem}

In order to establish the convergence in iterations, we need the following lemma,
which follows  from Lemma \ref{LL}:

{
\begin{lemma}\label{LL-Dis}
    Let $\{\xv_n\}_{n=0}^\infty$ be a sequence generated by the \DFDSM{}.
    Suppose that Assumptions \ref{assumption_tau:Dis} and \ref{A: xn bbd:Dis} hold.     Then, for all $n\in \mathbb{N}_0$ and $\xv\in \bigcap_{j=1}^m\fix(T_j)$, we have 
 \begin{align*}
    V_{n+1}(\xv)
    &\leq V_n(\xv) - \frac{1}{4}\Big(\frac{1}{m}\sum_{j=1}^m \|\xv_{n,j}-\xv_n\|^2 \Big) -\frac{1}{8(\tau_{\max}+1)^2}\Big(\frac{1}{m}\sum_{j=1}^m\|\xv_{n-\tau_n^j}-\xv_n\|^2 \Big) \\
    &\quad+ (2+8(\tau_{\max}+1)^2)\alpha_n^2 C_{\max}^2  
    + \frac{2\alpha_n}{m}\sum_{j=1}^m ( f_j(\xv)-f_j(T_j\xv_{n-\tau_n^j})),
\end{align*}
    where 
    $$
    V_n(\xv):= \|\xv_n-\xv\|^2 + \frac{1}{4(\tau_{\max}+1)}\displaystyle \sum_{i=1}^{\tau_{\max}+1}(\tau_{\max}+2-i)\|\xv_{n-i+1}-\xv_{n-i}\|^2
    $$
     and $C_{\max}$ is given in Proposition \ref{LL-d1}.
\end{lemma}
}

Finally, we are in a position to state the main theorem of this section. The proof is also directly obtained by following the line proofs of Theorem \ref{thm:con} and Corollary \ref{cor:main}. We also omit its proof.

\begin{theorem}\label{thm:con:Dis}
   Let  $\{\xv_n\}_{n=0}^\infty$ be a sequence generated by the \DFDSM{}. 
Suppose that Assumptions \ref{assumption_tau:Dis} and \ref{A: xn bbd:Dis} hold. Suppose that  the sequence $\{\alpha_n\}_{n=0}^\infty\subset (0,\infty)$ satisfies $\lim_{n\to \infty}\alpha_n=0$ and $\sum_{n=0}^\infty\alpha_n=\infty$.
Then, there exists a subsequence of $\{\xv_n\}_{n=0}^\infty$ that converges to an optimal solution in { $\mathcal{S}_m^*$.}  Furthermore, if the whole function $f$ is strictly convex, then then the whole sequence $\{x_n\}_{n=0}^\infty$ converges to the unique optimal solution of the problem $(\ref{main-pb-sum})$.
\end{theorem}

\section{Image Inpainting Problems} \label{Image Inpainting Problems}

\quad\, In this section, we will apply the \FDSM{} to solve the image inpainting problem. All the experiments were performed in MATLAB(R2023b) on a MacBook Air 13.3-inch with Apple M1 chip processor and 8GB memory. 
\vskip 2mm

For an  ideal image $\xv\in \mathbb{R}^{256 \times 256}$ whose entry $\xv_{i,j}$ is the pixel value in $i$-th row and $j$-th column, for all $i, j = 1,\ldots,256$, we represent the vector $\xv\in \mathbb{R}^d$, where $d=256^2=65536$, the vector generated by vectorizing the image $X$. 
Given the damaged image $\bv\in \mathbb{R}^d$ and the diagonal matrix $B:=\text{diag}(\overline{\bv}) \in \mathbb{R}^{d\times d}$ such that $\overline{\bv}\in \mathbb{R}^d$ is defined by
\begin{align*}
\overline{\bv}_{i,j}=\begin{cases}
0, & \text{ if $\bv_{i,j}=0,$ }\\
1, & \text{ if $\bv_{i,j} \neq 0$},
\end{cases}
\end{align*}
the  image inpainting problem is to find an ideal image $\overline{\xv}\in \mathbb{R}^d$ such that
\begin{align}\label{inverse}
    B\overline{\xv}=\bv.
\end{align}
Since it is well known that the problem \eqref{inverse} is the ill-posed linear inverse problem, in this section, we will investigate a strategy for solving it by considering the following nonsmooth convex constrained minimization problem of the form:
\begin{align}\label{inpainting}
	\begin{array}{ll}
		\textrm{minimize }\indent \Vert W\xv \Vert_1\\
		\textrm{subject to}\indent \xv\in \mathrm{argmin}_{\uv\in \R^d} \frac{1}{2}\Vert B\uv-\bv\Vert^2,
	\end{array}%
\end{align}
where the objective function is a nonsmooth sparsity-promoting potential function $\Vert W\xv \Vert_1$ and $W$ is a dictionary transform, see the book of Starck et al. \cite{SMF-10} for further details. 
\vskip 2mm

To investigate the solving performance, we will consider the following five dictionary transforms will be mentioned as follows. For all $\xv\in \mathbb{R}^d$,

\begin{enumerate}
\item[\bf(1)]  The transform $R:\mathbb{R}^d \to \mathbb{R}^d$ is       defined by 
    \begin{align*}
            (R\xv)_{i,j}=\begin{cases}
			\xv_{i+1,j}-\xv_{i,j}, & \text{if $i<256,$ }\\
            0, & \text{otherwise}.
		 \end{cases}
    \end{align*}  

\item[\bf(2)] The transform $C:\mathbb{R}^d \to \mathbb{R}^d$ is       defined by
    \begin{align*}
            (C\xv)_{i,j}=\begin{cases}
			\xv_{i,j+1}-\xv_{i,j}, & \text{if $j<256,$ }\\
            0, & \text{otherwise}.
	   \end{cases}
    \end{align*}
  
\item[\bf(3)]  The transform $H:\mathbb{R}^d \to \mathbb{R}^{256 \times 256}$ is defined by
    \begin{align}\label{H}
        H\xv := (W_{256,8}W_{256,7}\cdots W_{256,2} W_{256})X(W_{256,8}W_{256,7}\cdots W_{256,2} W_{256})^T,    
    \end{align}
where $X \in \mathbb{R}^{256\times 256}$ is an image $\xv$ in the form of matrix
and, for some positive integer $k$ such that $1\leq k \leq 8,$ a matrix $W_{256,k}\in \mathbb{R}^{256 \times 256}$ is given by
\begin{align*}
W_{256,k} := \begin{cases}
			W_{256}, & \text{ if $k=1,$ }\\
            \text{diag}(W_{2^{8-k+1}}, I_{2^{8-k+1}}, I_{2^{8-k+2}}, \ldots, I_{2^{8-1}}), & \text{ if $2 \leq k \leq 8$},
		 \end{cases}
\end{align*}
where the matrix $\text{diag}(D_1, \ldots, D_l)$ with positive integer $l$ contains the blocks $D_1, \ldots, D_l$ along its diagonal and zeros in all other positions and, for an even positive integer $M$,  the matrix $W_M\in \mathbb{R}^{M\times M}$ is given by
\begin{align*}
    W_M := \sqrt{2}\left[\frac{H_{\frac{M}{2}}}{-G_\frac{M}{2}} \right],
\end{align*}
where the submatrix $H_{\frac{M}{2}}$ and $G_{\frac{M}{2}}$ are defined as in the relation $(4.7)$ of book of van Fleet \cite[page 134]{P-19}.
The operator $H$ is the well-known Haar wavelet transform, which is an important tool for image processing \cite{P-19}.

\item[\bf(4)] The transform $L:\mathbb{R}^d \to \mathbb{R}^d        \times \mathbb{R}^d$ is defined by
    \begin{align*}
        L\xv:=(R\xv,C\xv),
    \end{align*}
    which is the combination of $R$ and $C$ given in {\bf(1)} and {\bf(2)}, respectively, so that $\|L\xv\|_1$ is known as the anisotropic TV norm\cite{C-04,B-14}. 

\item[\bf(5)] The transform $G:\mathbb{R}^d \to \mathbb{R}^{d} \times \mathbb{R}^d$ is given by
\begin{align*}
G\xv:= (H\xv, L\xv),
\end{align*}
where $H\xv$ is the vectorized form of the image $H\xv \in \mathbb{R}^{256 \times 256}$ given in \eqref{H}. The transform $G$ is defined as the combination of the transforms $L$ and $H$, i.e., 
$\|G\cdot \|_1 = \|L\cdot \|_1 + \|H\cdot \|_1.$
\end{enumerate}

For applying the \FDSM{} to solve the problem \eqref{inpainting},  we set the objective function $f(\xv)=\Vert W\xv \Vert_1$ and the operator $T:\mathbb{R}^d \to \mathbb{R}^d$  by 
$$
T\xv := \xv - \frac{1}{\Vert B \Vert^2}B^T (B\xv-\bv),\,\,\,\forall \xv\in\mathbb{R}^d.
$$

  Note that $T$ is a firmly nonexpansive operator with 
$$
\fix(T) = \mathrm{argmin}_{\uv\in \mathbb{R}^d} \frac{1}{2}\Vert B\uv-\bv\Vert^2
$$
(see \cite[Lemma 4.6.2]{C-12}).

We set the initial vector $\xv_0$ be the zero vector. We consider the method's performance for the delay bounds  $\tau=0, 1, 3, 5, 10$ and $20$ and set the time-varying delayed sequence by
$\tau_n = n \mod{(\tau+1)}$. 
{ Note that the \FDSM{} with $\tau=0$ is Algorithm 3.1 of \cite{I-16} with the number of users $I=1$.}
We set the step size 
 $\alpha_n = \frac{a_0}{n+1} \left(\frac{8}{3+2(\tau+1)^2}\right)^{1/a}$,
where $a, a_0 \in \{0.1,0.2,\ldots,0.9\}$. 
We conducted experiments for reconstructing three RGB images sized $256 \times 256$, the original images of Cafe image, Blue Lagoon image, and Pagoda image are shown in Figure \ref{FI-original}, whose noisy images are obtained by randomly masking 50\% of all pixels to black.
\begin{figure}[H]
	\begin{center}
        \begin{minipage}[c]{0.32\linewidth}
		\centering
		{\includegraphics[scale = 0.39]{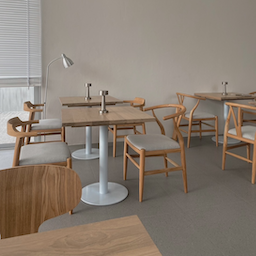}}	\\
	     	(a)  The Cafe image
        \end{minipage}
        \begin{minipage}[c]{0.32\linewidth}
		\centering
		{\includegraphics[scale = 0.39]{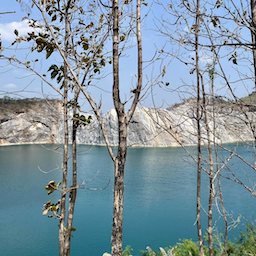}}	\\
	     	(b)  The Blue Lagoon image    
        \end{minipage}
        \begin{minipage}[c]{0.32\linewidth}
		\centering
		{\includegraphics[scale = 0.39]{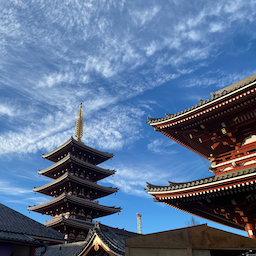}}	\\
	     	(c)  The Pagoda image 
        \end{minipage}
            \caption{\label{FI-original}Three original images.}
	\end{center}
\end{figure}
\vspace{-0.2cm}

We measure the quality of the reconstructing images by the Peak Signal-to-Noise Ratio (PSNR).
We examine the performance of the proposed method by considering the effects of the five objective functions according to the different dictionary transforms discussed above. We terminate the method by either the computational runtime reaches $10$ seconds or the number of iterations reaches $500$ iterations, whichever came first. The results are given in Tables \ref{TI1} - \ref{TI3}.

\subsection{The Cafe Image}
\vskip 2mm

\quad\, In this subsection, we present the experiment results for the Cafe image. We firstly investigate the influence of the parameters $a$ and $a_0$ given in the step-size sequence on the PSNR value and the number of iterations  $\#$(iter) for several time-varying delays' bounds and for different dictionary transforms. The best choices of fine-tuning parameters $a$ and $a_0$ with the PSNR values and the numbers of iterations are presented in Table \ref{TI1}. 

\begin{table}[H]
\centering
\small
\begin{tabular}{ccccccccccccc}
\toprule
\multirow{2}{*}{Transform} & \multicolumn{4}{c}{{$\tau=0$ \cite[Alg. 3.1]{I-16}}}      & \multicolumn{4}{c}{$\tau=1$}      & \multicolumn{4}{c}{$\tau=3$}      \\
\cmidrule(rl){2-5} \cmidrule(rl){6-9} \cmidrule(rl){10-13}
                    & $a$   & $a_0$  & PSNR    & $\#$(iter) & $a$   & $a_0$  & PSNR    & $\#$(iter) & $a$   & $a_0$  & PSNR    & $\#$(iter) \\
\midrule
$R$                   & 0.7 & 0.5 & 29.9419 & 340  & 0.9 & 0.9 & 29.9254 & 372  & 0.9 & 0.9 & 23.3780 & 384  \\
$C$                   & 0.5 & 0.3 & 30.1714 & 338  & 0.9 & 0.9 & 30.0426 & 374  & 0.9 & 0.9 & 23.2596 & 385  \\
$H$                   & 0.7 & 0.2 & 29.0804 & 253  & 0.5 & 0.7 & 29.1040 & 305  & 0.9 & 0.9 & 26.9789 & 339  \\
$L$                  & 0.5 & 0.1 & 33.0416 & 280  & 0.4 & 0.5 & 33.0923 & 333  & 0.9 & 0.8 & 32.9870 & 353  \\
$G$                 & 0.9 & 0.1 & 32.5660 & 201  & 0.2 & 0.5 & 32.5537 & 271  & 0.9 & 0.8 & 32.3546 & 319 
\end{tabular}

\begin{tabular}{ccccccccccccc}
\midrule
\multirow{2}{*}{Transform} 
                    & \multicolumn{4}{c}{$\tau=5$}      & \multicolumn{4}{c}{$\tau=10$}     & \multicolumn{4}{c}{$\tau=20$}     \\  
\cmidrule(rl){2-5} \cmidrule(rl){6-9} \cmidrule(rl){10-13}
                    & $a$   & $a_0$  & PSNR    & $\#$(iter) & $a$   & $a_0$  & PSNR    & $\#$(iter) & $a$   & $a_0$  & PSNR    & $\#$(iter) \\
\midrule
$R$                   & 0.9 & 0.9 & 16.4472 & 392  & 0.9 & 0.9 & 11.1249 & 396  & 0.9 & 0.9 & 9.3948  & 393  \\
$C$                   & 0.9 & 0.9 & 16.4131 & 391  & 0.9 & 0.9 & 11.1148 & 396  & 0.9 & 0.9 & 9.3938  & 398  \\
$H$                   & 0.9 & 0.9 & 16.3118 & 355  & 0.9 & 0.9 & 10.6086 & 366  & 0.9 & 0.9 & 9.2739  & 374  \\
$L$                  & 0.9 & 0.9 & 30.4198 & 367  & 0.9 & 0.9 & 13.9825 & 371  & 0.9 & 0.9 & 9.9458  & 380  \\
$G$                 & 0.9 & 0.9 & 29.0053 & 339  & 0.9 & 0.9 & 14.3807 & 359  & 0.9 & 0.9 & 10.0546 & 370 \\
\bottomrule
\end{tabular}
  \caption{\label{TI1}The best choices of parameters $a$ and $a_0$ with the PSNR values and the numbers of iterations for the Cafe image.}
\end{table}

It is observed from Table \ref{TI1} that the parameters $a$ and $a_0$ seem to be steady at higher $\tau$ ($\tau= 5, 10, 20$) with the same value $a=a_0=0.9$.

Focusing on the performance comparisons of the dictionary transforms and the delay's bounds, we observe that the transform $L$ gives the highest PSNR value for the delay bound $\tau=0, 1, 3$ and $5$ compared to other transforms. However, for the delay's bounds $\tau=10$ and $20$, we observe the highest PSNR value with the transform $G$. 
The best PSNR value of $33.0923$ is obtained when the dictionary transform is $L$ and the delays' bound $\tau=1$.

To illustrate more insightful the behaviors of the PSNR values obtained from each dictionary transform and the delay bound, we plot each curve by putting the possibly best choices $a$ and $a_0$  given in Table \ref{TI1} when performing  \FDSM{}. We show the PSNR curves within $5$ seconds in Figure \ref{FI1}.

\begin{figure}[H]
	\begin{center}
        \begin{minipage}[c]{0.3\linewidth}
		\centering
		{\includegraphics[scale = 0.3]{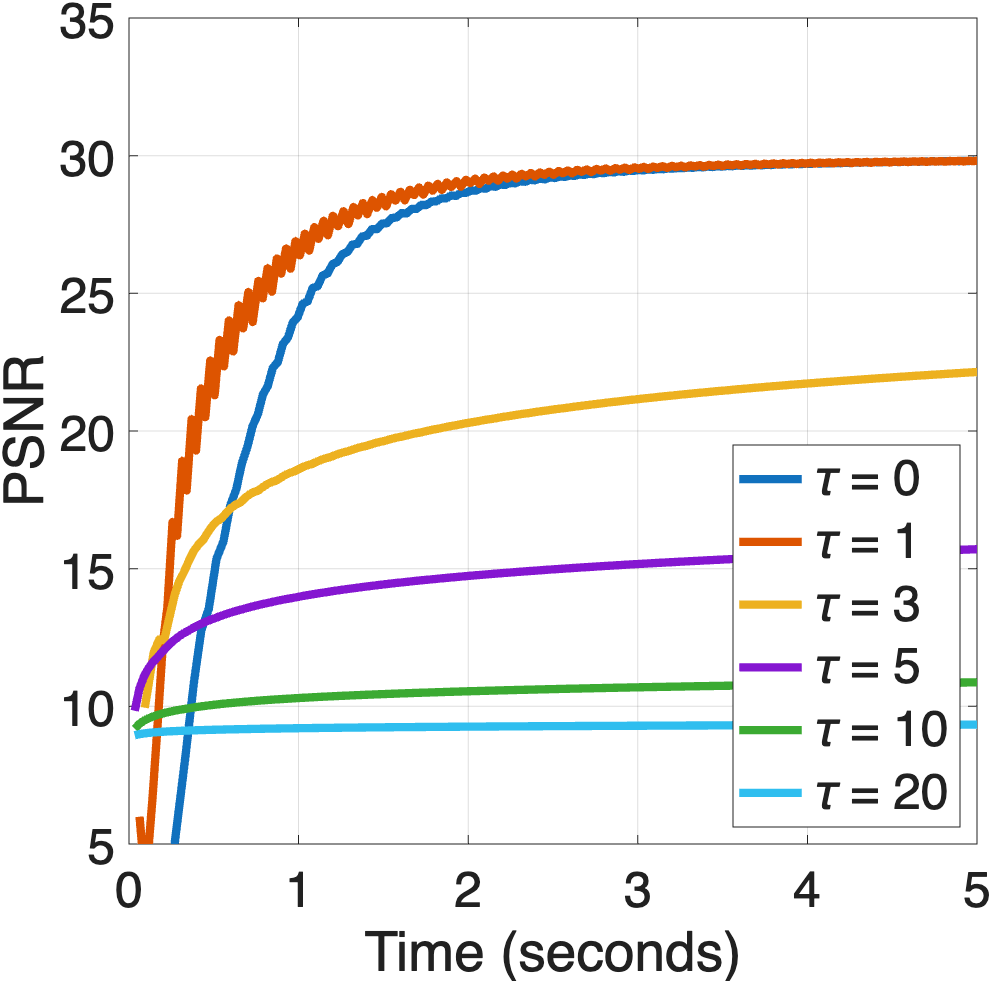}}	\\
	     	(a)  Transform $R$
        \end{minipage}
        \begin{minipage}[c]{0.3\linewidth}
		\centering
		{\includegraphics[scale = 0.3]{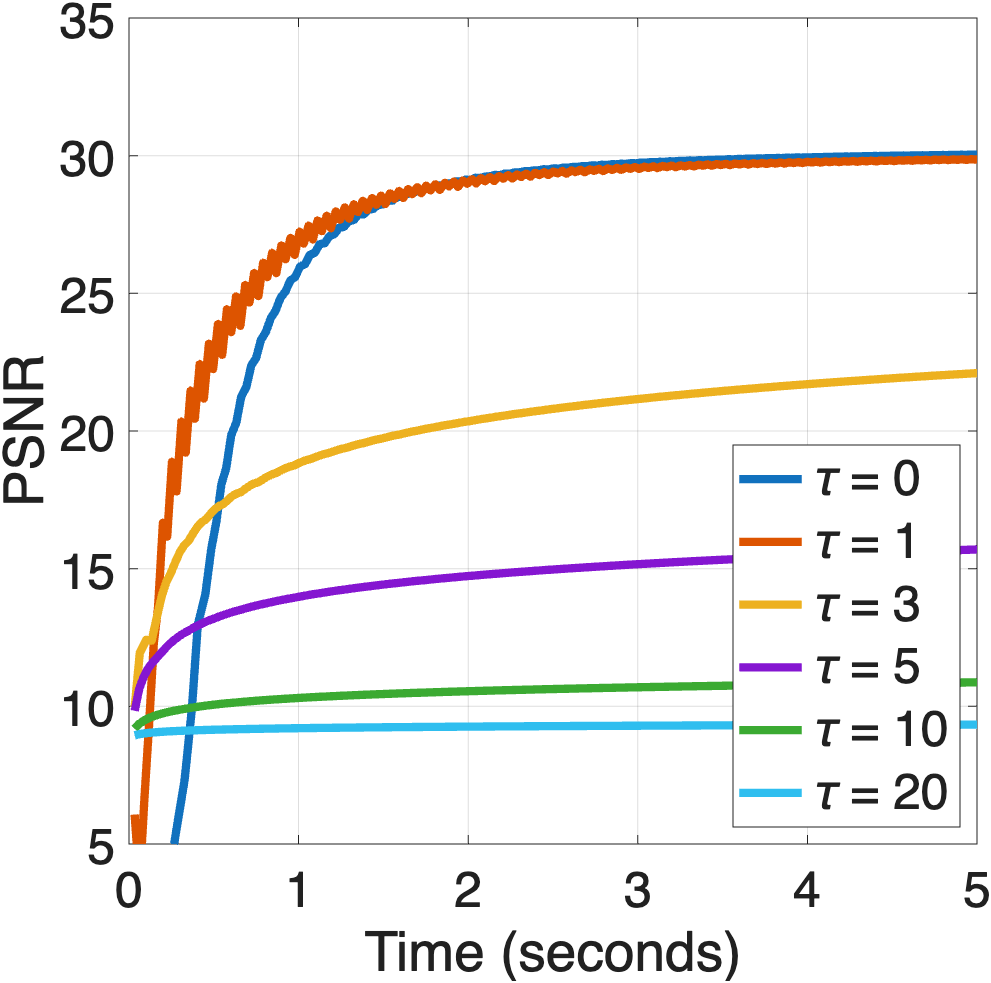}}	\\
	     	(b)  Transform $C$     
        \end{minipage}
        \begin{minipage}[c]{0.3\linewidth}
		\centering
		{\includegraphics[scale = 0.3]{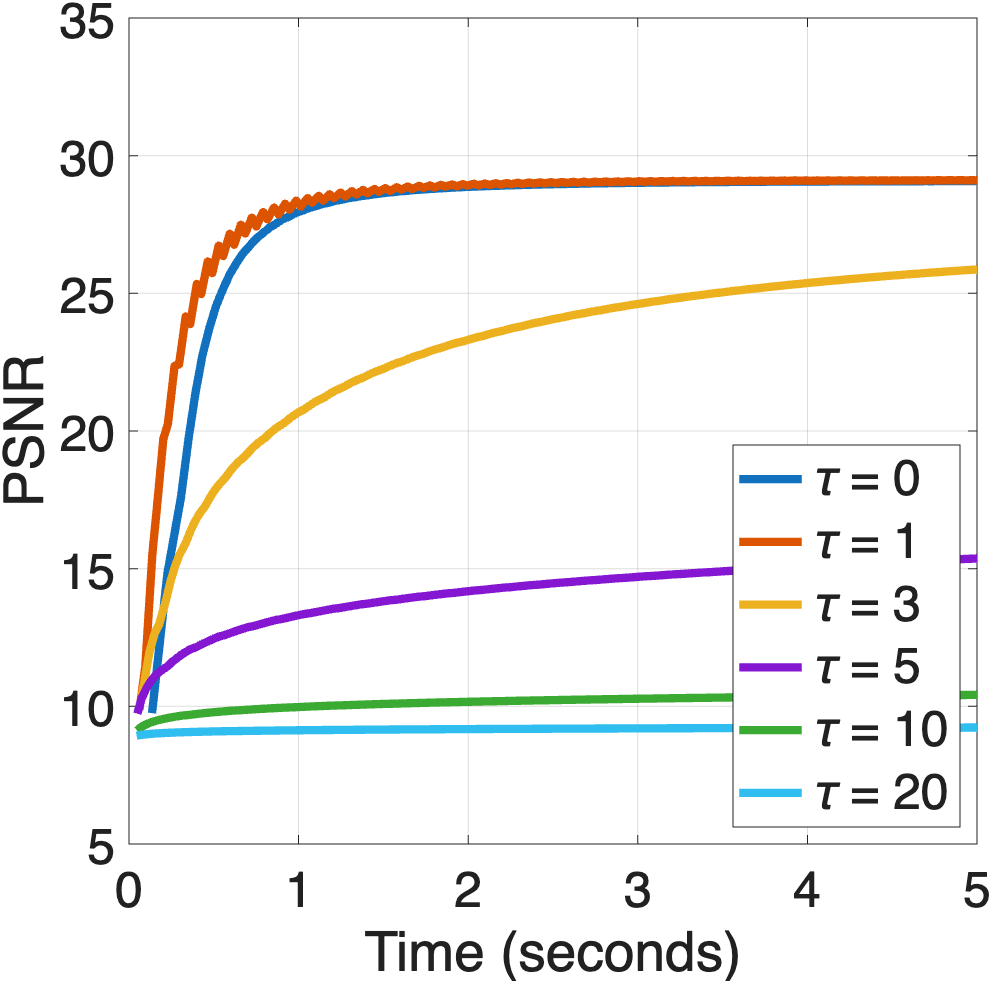}}	\\
	     	(c)  Transform $H$ 
        \end{minipage}

             \vspace{0.3cm}
             
        \begin{minipage}[c]{0.3\linewidth}
		\centering
		{\includegraphics[scale = 0.3]{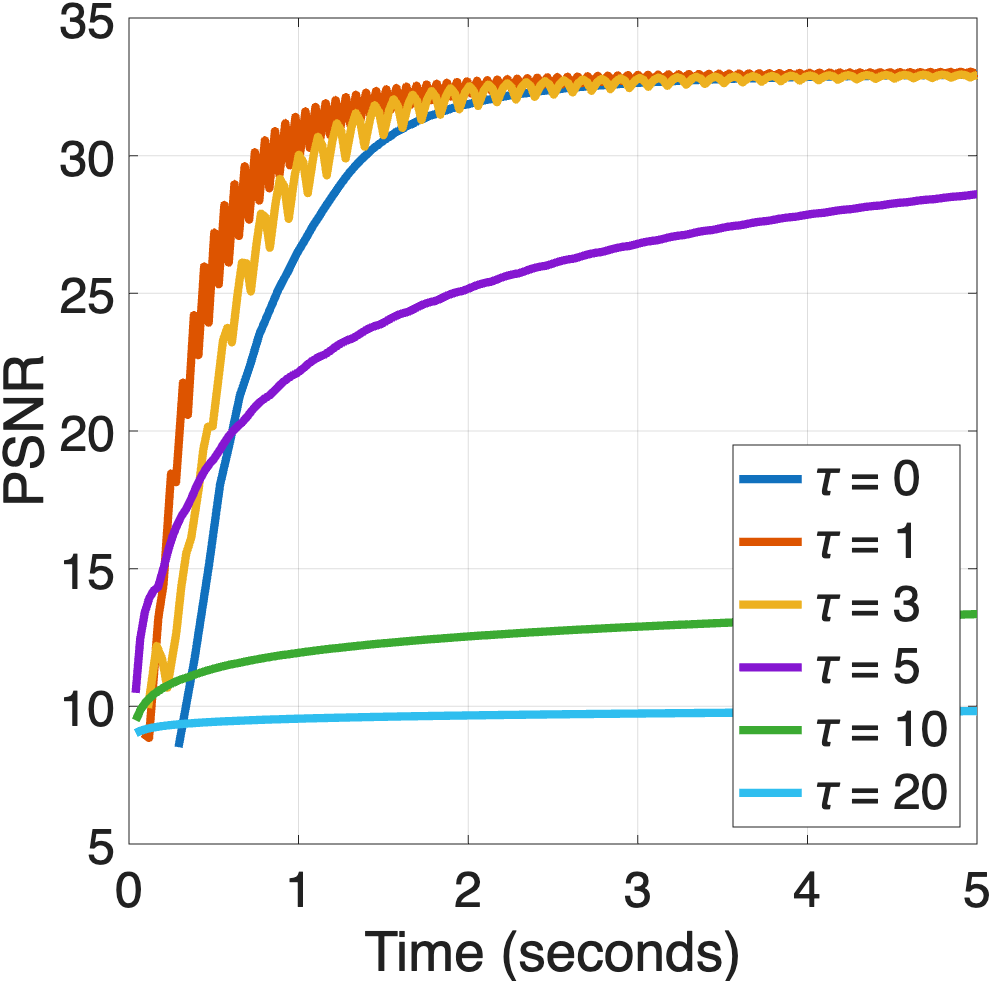}}	\\
	     	(d)  Transform $L$ 
        \end{minipage}
        \begin{minipage}[c]{0.3\linewidth}
		\centering
		{\includegraphics[scale = 0.3]{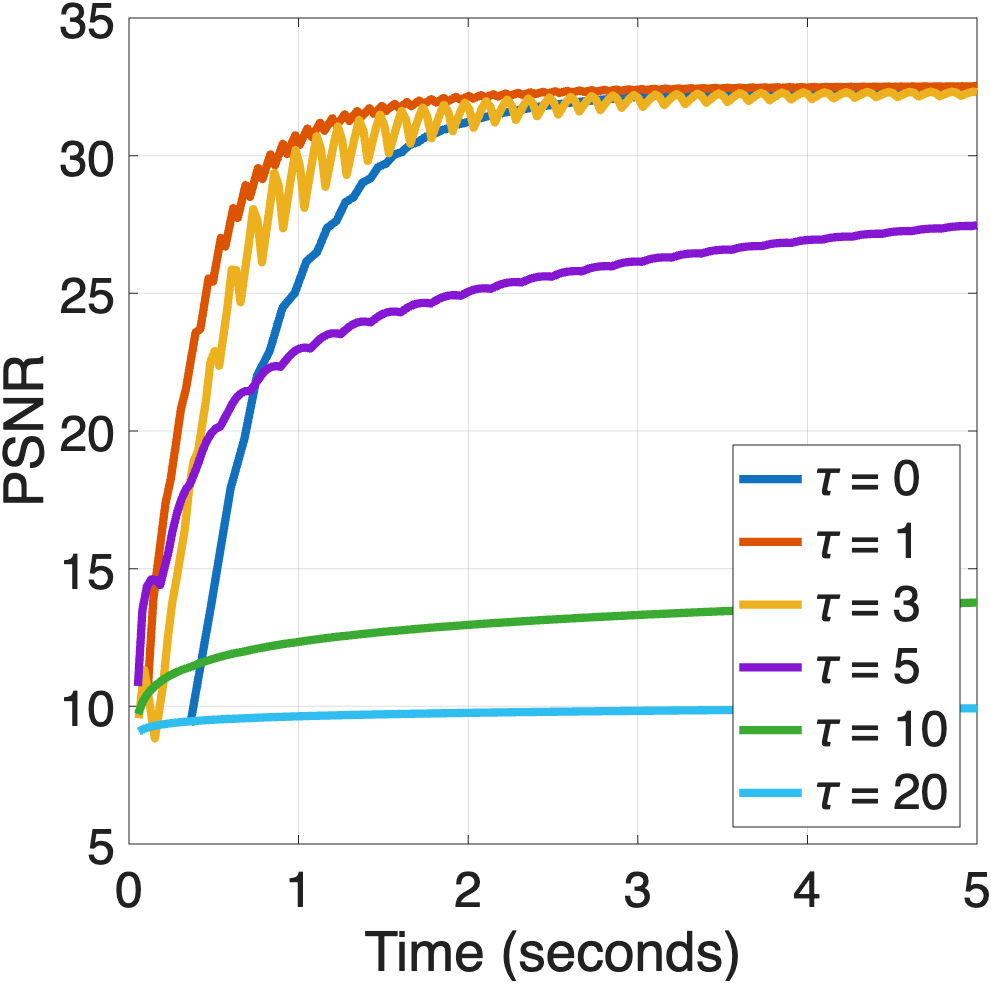}}	\\
	     	(e) Transform $G$ 	     
        \end{minipage}
            \caption{\label{FI1}Behaviors of the PSNR values for the Cafe image.}
	\end{center}
\end{figure}

It can be observed from Figure \ref{FI1} that the PSNR performed by the \FDSM{} with $\tau=1$ increases to the possibly best PSNR value faster than the non-delayed counterpart ($\tau=0$) and other delayed bounds. Moreover, for the choices of the transform $L$ and $G$, we further notice that the \FDSM{} with $\tau=1$ and $\tau=3$ increases to the possibly best PSNR value faster than the non-delayed counterpart ($\tau=0$). It is not surprising that the PSNR values performed by the \FDSM{} with $\tau=10,$ and $20$ are not increasing well for all types of transforms.

Focusing on \FDSM{} with $\tau=1$ with different type of transforms, we notice that the transform $L$ yields the best PSNR values with the shortest computational runtime. It is slightly surprising that the results obtained by the transform $G$ are not better than the results obtained by $L$, even though $G$ is the combination of $H$ and $L$. 

We show the visual restoration results of the Cafe image using the FDSM with different types of transforms and  $\tau=1$ within $10$ seconds in Figure \ref{FFI1}.

\begin{figure}[H]
	\begin{center}
        \begin{minipage}[c]{0.3\linewidth}
		\centering
		{\includegraphics[scale = 0.45]{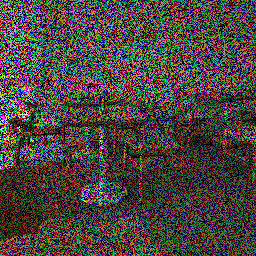}}	\\
	     	(a) Damaged image 	     
        \end{minipage}
        \begin{minipage}[c]{0.3\linewidth}
		\centering
		{\includegraphics[scale = 0.45]{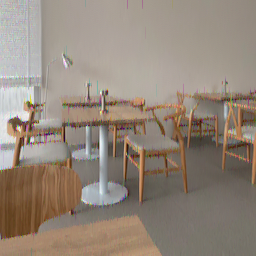}}	\\
	     	(b)  Transform $R$
        \end{minipage}
        \begin{minipage}[c]{0.3\linewidth}
		\centering
		{\includegraphics[scale = 0.45]{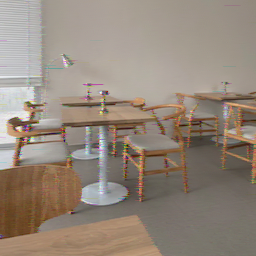}}	\\
	     	(c)  Transform $C$ 
        \end{minipage}

        \vspace{0.3cm}
        
        \begin{minipage}[c]{0.3\linewidth}
		\centering
		{\includegraphics[scale = 0.45]{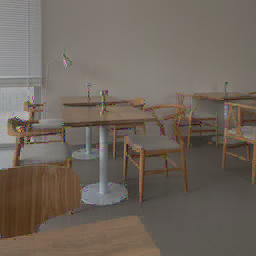}}	\\
	     	(d)  Transform $H$ 
        \end{minipage}
        \begin{minipage}[c]{0.3\linewidth}
		\centering
		{\includegraphics[scale = 0.45]{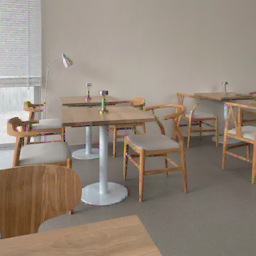}}	\\
	     	(e)  Transform $L$ 
        \end{minipage}
        \begin{minipage}[c]{0.3\linewidth}
		\centering
		{\includegraphics[scale = 0.45]{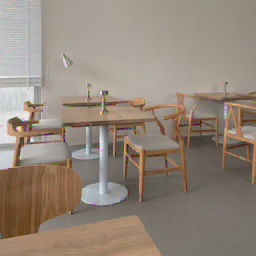}}	\\
	     	(f) Transform $G$ 	     
        \end{minipage}
            \caption{\label{FFI1} 
              The Cafe image restored  by the FDSM with $\tau=1$ within $10$ seconds.}
	\end{center}
\end{figure}

\subsection{The Blue Lagoon Image}
\vskip 2mm

\quad\, In this subsection, we also examine the experiment results in the same manner as the previous subsection with the different image, namely the Blue Lagoon image as shown in Figure \ref{FI-original} (b).

We also start with the influences of the parameters $a$ and $a_0$ given in the step-size sequence where the best choices of fine-tuning parameters $a$ and $a_0$ with the PSNR values and the numbers of iterations are shown in Table  \ref{TI2}. 

\begin{table}[H]
\centering
\small
\begin{tabular}{ccccccccccccc}
\toprule
\multirow{2}{*}{Transform} & \multicolumn{4}{c}{{$\tau=0$ \cite[Alg. 3.1]{I-16}}}      & \multicolumn{4}{c}{$\tau=1$}      & \multicolumn{4}{c}{$\tau=3$}      \\
\cmidrule(rl){2-5} \cmidrule(rl){6-9} \cmidrule(rl){10-13}
                    & $a$   & $a_0$  & PSNR    & $\#$(iter) & $a$   & $a_0$  & PSNR    & $\#$(iter) & $a$   & $a_0$  & PSNR    & $\#$(iter) \\
\midrule
$R$                   & 0.6 & 0.3 & 21.1844 & 329  & 0.9 & 0.9 & 21.1593 & 356  & 0.9 & 0.9 & 18.1532 & 369  \\
$C$                   & 0.7 & 0.3 & 20.0585 & 332  & 0.9 & 0.9 & 20.0276 & 361  & 0.9 & 0.9 & 17.6748 & 371  \\
$H$                   & 0.9 & 0.2 & 19.7150 & 253  & 0.3 & 0.9 & 19.6889 & 308  & 0.9 & 0.9 & 18.2422 & 343  \\
$L$                  & 0.9 & 0.1 & 21.9373 & 292  & 0.2 & 0.7 & 21.9574 & 337  & 0.8 & 0.8 & 21.9401 & 358  \\
$G$                 & 0.9 & 0.1 & 21.5513 & 204  & 0.4 & 0.2 & 21.7118 & 272  & 0.9 & 0.6 & 21.4944 & 317 
\end{tabular}
\begin{tabular}{ccccccccccccc}
\toprule
\multirow{2}{*}{Transform} & \multicolumn{4}{c}{$\tau=5$}      & \multicolumn{4}{c}{$\tau=10$}     & \multicolumn{4}{c}{$\tau=20$}    \\
\cmidrule(rl){2-5} \cmidrule(rl){6-9} \cmidrule(rl){10-13}
                    & $a$   & $a_0$  & PSNR    & $\#$(iter) & $a$   & $a_0$  & PSNR    & $\#$(iter) & $a$   & $a_0$  & PSNR   & $\#$(iter) \\
\midrule
$R$                   & 0.9 & 0.9 & 13.4783 & 376  & 0.9 & 0.9 & 9.2611  & 380  & 0.9 & 0.9 & 7.8716 & 382  \\
$C$                   & 0.9 & 0.9 & 13.3464 & 377  & 0.9 & 0.9 & 9.2458  & 378  & 0.9 & 0.9 & 7.8690 & 379  \\
$H$                   & 0.9 & 0.9 & 12.6921 & 356  & 0.9 & 0.9 & 8.8117  & 368  & 0.9 & 0.9 & 7.7684 & 375  \\
$L$                  & 0.9 & 0.9 & 20.8128 & 367  & 0.9 & 0.9 & 11.4293 & 375  & 0.9 & 0.9 & 8.3185 & 380  \\
$G$                 & 0.9 & 0.9 & 20.1136 & 339  & 0.9 & 0.9 & 11.6457 & 361  & 0.9 & 0.9 & 8.3931 & 373 \\
\bottomrule
\end{tabular}
 \caption{\label{TI2}The best choices of parameters $a$ and $a_0$ with the PSNR values and the numbers of iterations for the Blue Lagoon image.}
\end{table}

Form Table \ref{TI2}, we notice that the trends of parameters $a$ and $a_0$ are in the same fashion to the results obtained in Table \ref{TI1}. 
We observe that the dictionary transform $L$ also yields the highest PSNR values for the delay’s bounds $\tau=0, 1, 3$ and $5$. Again, the best PSNR value of $21.9574$ is obtained when the dictionary transform is $L$ and the delay bound $\tau=1$.

We present the behaviors of the PSNR values within $5$ seconds obtained from each dictionary transform and the delay bound in the following Figure \ref{FI2}.

\begin{figure}[H]
	\begin{center}
        \begin{minipage}[c]{0.3\linewidth}
		\centering
		{\includegraphics[scale = 0.3]{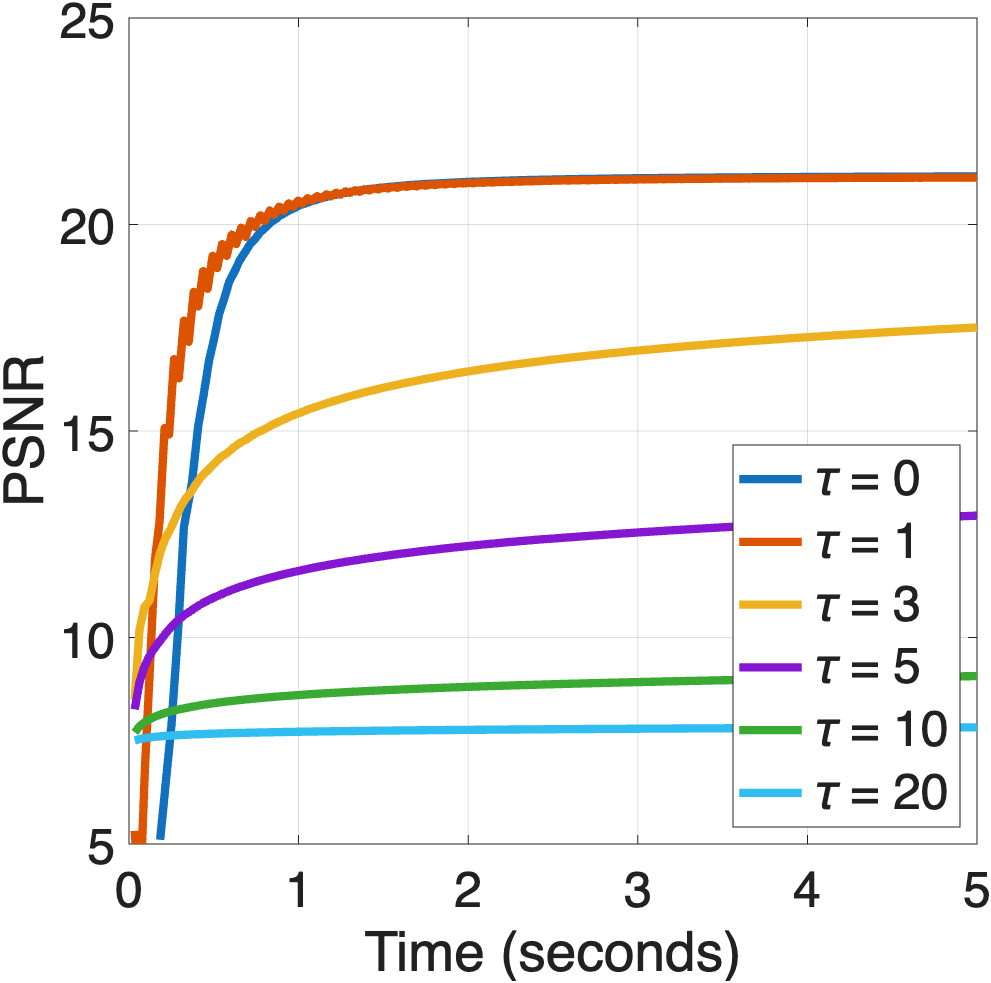}}	\\
	     	(a)  Transform $R$
        \end{minipage}
        \begin{minipage}[c]{0.3\linewidth}
		\centering
		{\includegraphics[scale = 0.3]{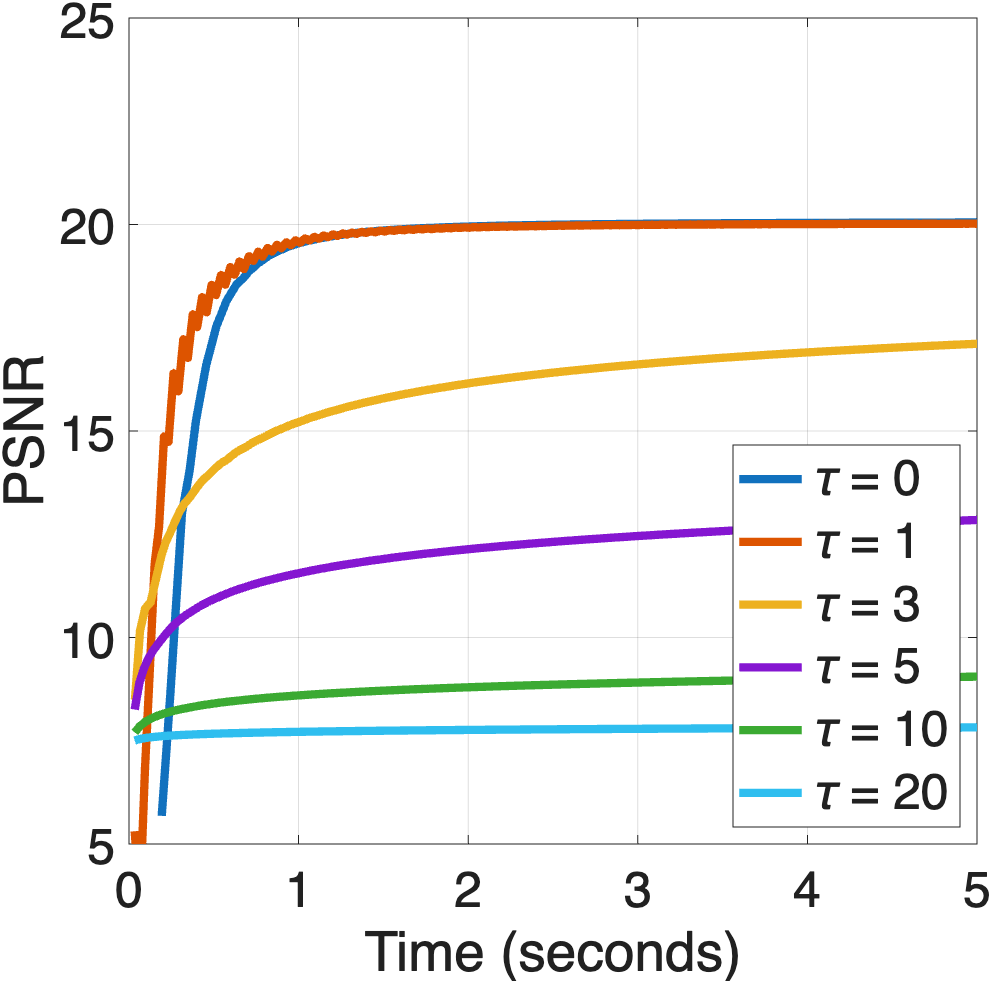}}	\\
	    (b)  Transform $C$     
        \end{minipage}
        \begin{minipage}[c]{0.3\linewidth}
		\centering
		{\includegraphics[scale = 0.3]{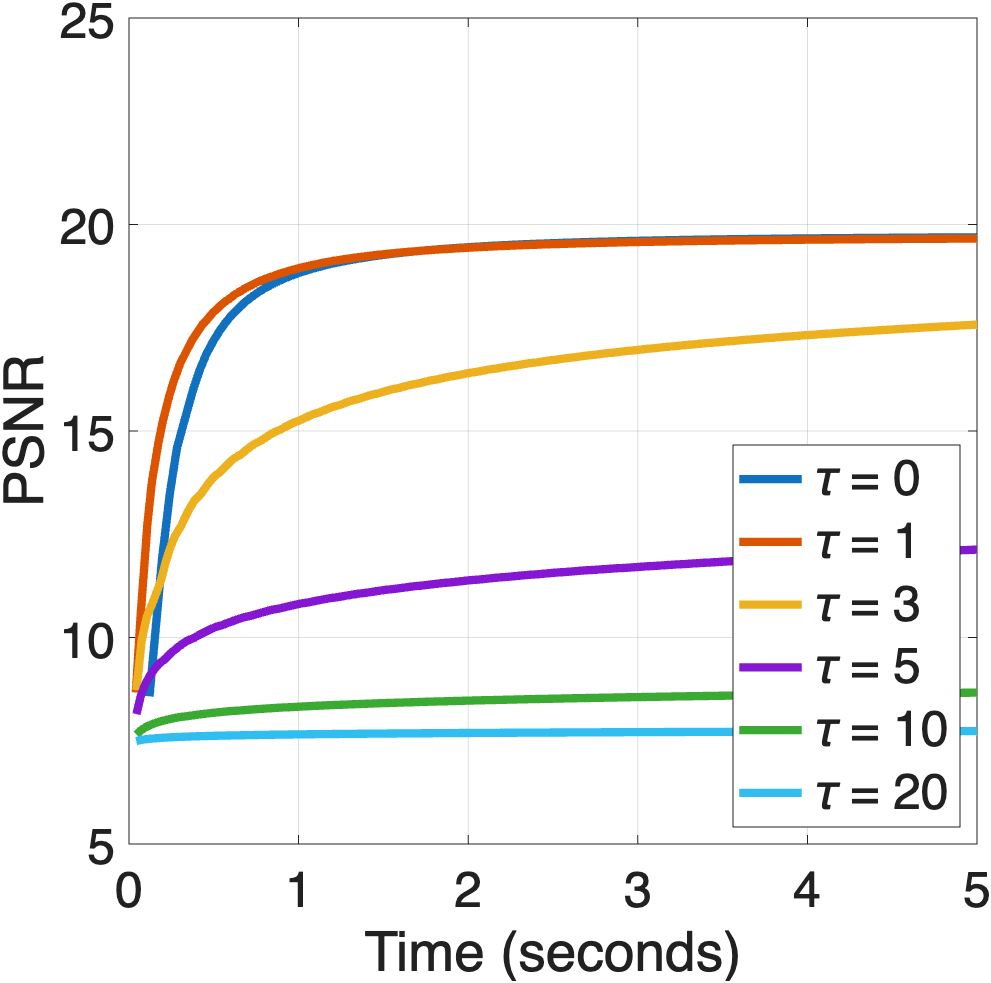}}	\\
	     	(c)  Transform $H$ 
        \end{minipage}
       
        \vspace{0.3cm}
        
        \begin{minipage}[c]{0.3\linewidth}
		\centering
		{\includegraphics[scale = 0.3]{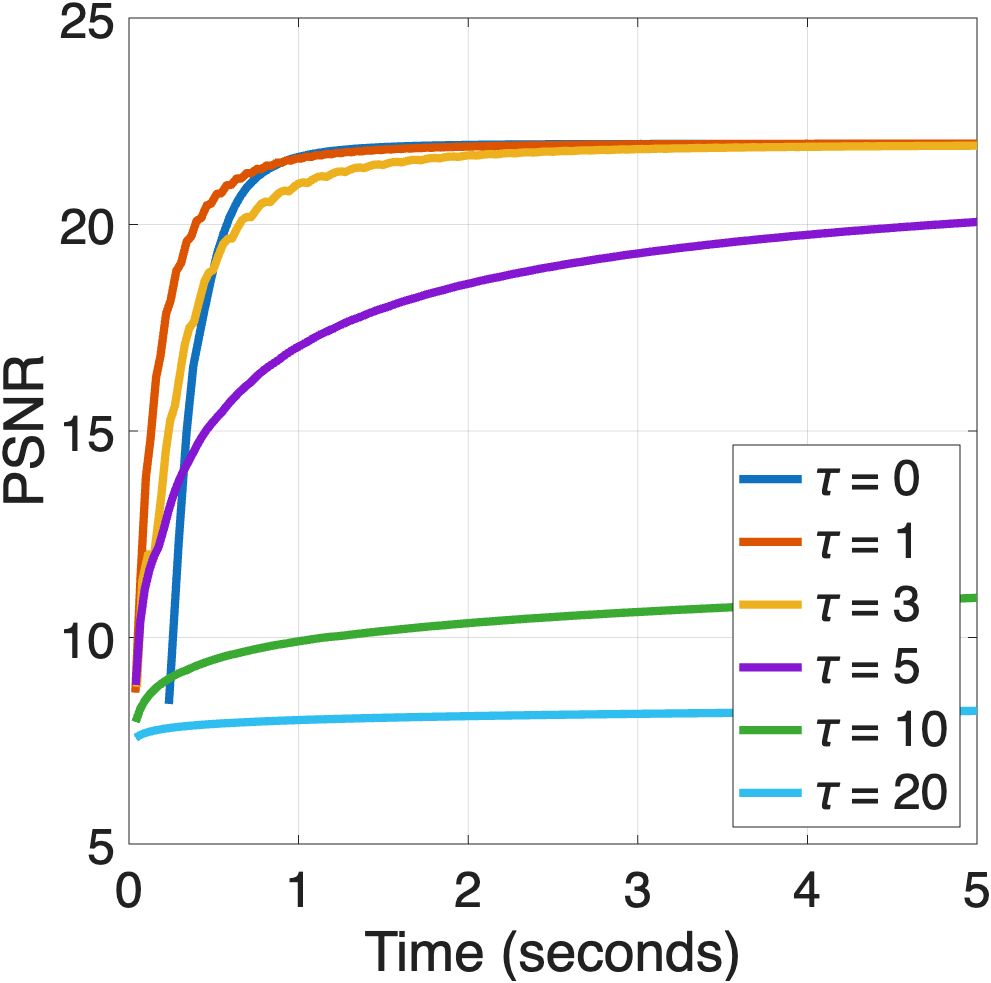}}	\\
	     	(d)  Transform $L$ 
        \end{minipage}
        \begin{minipage}[c]{0.3\linewidth}
		\centering
		{\includegraphics[scale = 0.3]{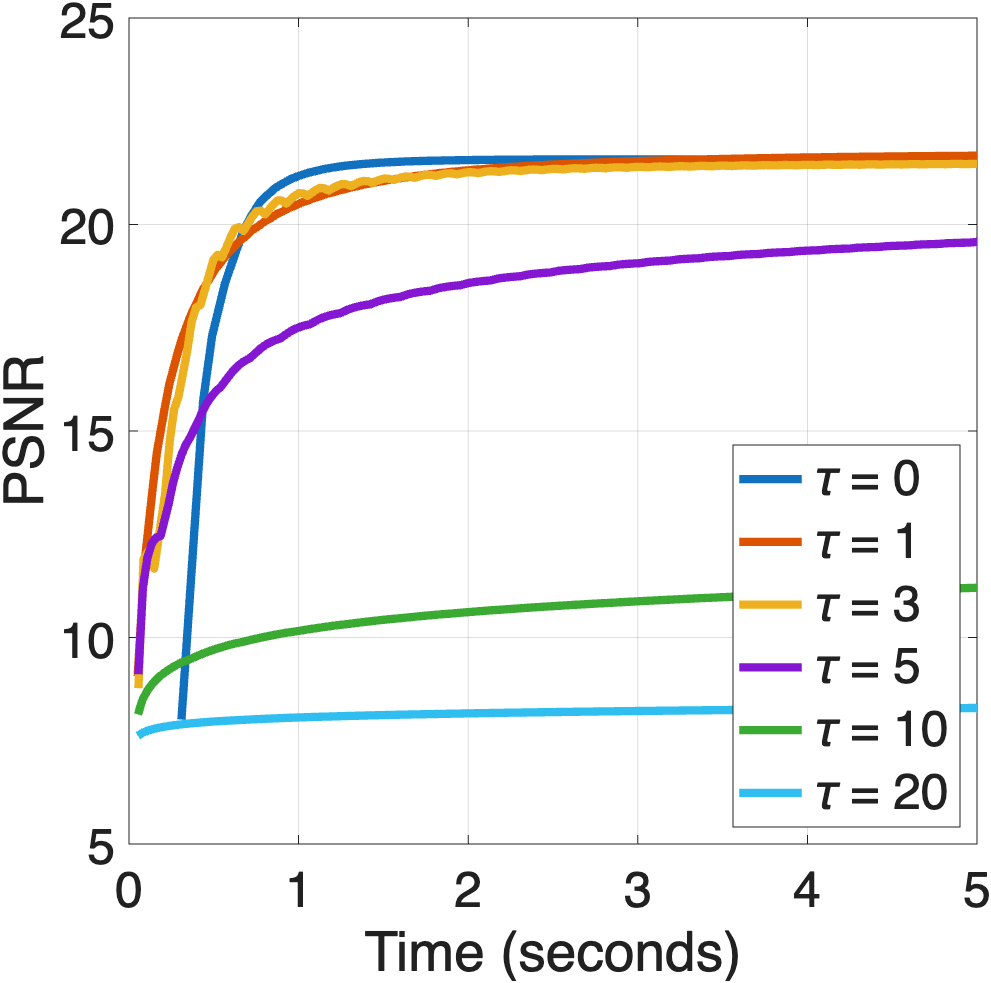}}	\\
	     	(e) Transform $G$ 	     
        \end{minipage}
            \caption{\label{FI2}Behaviors of the PSNR values for the Blue Lagoon image.}
	\end{center}
\end{figure}

It is observed from Figure \ref{FI2} that, for overall experiment, the \FDSM{} with $\tau=0, 1$ and  $3$ yield the PSNR behaviors in the same directions as in Figure \ref{FI1}. Again, for each delay bound $\tau=0, 1, 3$ and  $5$, we notice that the transform $L$ gives the best PSNR values with the shortest computational runtime compared with other transforms. 
We present the visual restoration outcomes of the Blue Lagoon image obtained from the FDSM with different types of transforms and $\tau=1$ at $10$ seconds in Figure \ref{FFI2}.

\begin{figure}[H]
	\begin{center}
        \begin{minipage}[c]{0.3\linewidth}
		\centering
		{\includegraphics[scale = 0.45]{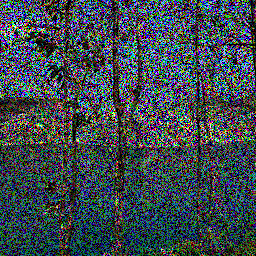}}	\\
	     	(a) Damaged image 	     
        \end{minipage}
        \begin{minipage}[c]{0.3\linewidth}
		\centering
		{\includegraphics[scale = 0.45]{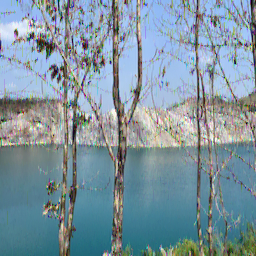}}	\\
	     	(b)  Transform $R$
        \end{minipage}
        \begin{minipage}[c]{0.3\linewidth}
		\centering
		{\includegraphics[scale = 0.45]{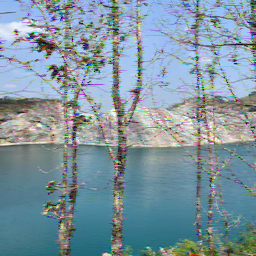}}	\\
	     	(c)  Transform $C$ 
        \end{minipage}

        \vspace{0.3cm}
        
        \begin{minipage}[c]{0.3\linewidth}
		\centering
		{\includegraphics[scale = 0.45]{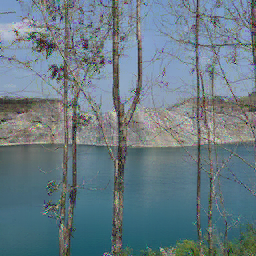}}	\\
	     	(d)  Transform $H$ 
        \end{minipage}       
        \begin{minipage}[c]{0.3\linewidth}
		\centering
		{\includegraphics[scale = 0.45]{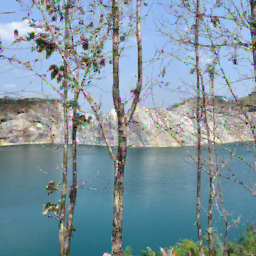}}	\\
	     	(e)  Transform $L$ 
        \end{minipage}
        \begin{minipage}[c]{0.3\linewidth}
		\centering
		{\includegraphics[scale = 0.45]{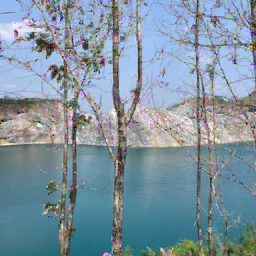}}	\\
	     	(f) Transform $G$ 	     
        \end{minipage}
            \caption{\label{FFI2} 
            The Blue Lagoon image restored by the  FDSM with $\tau=1$ within $10$ seconds.}
	\end{center}
\end{figure}


\subsection{The Pagoda Image}
\vskip 2mm

\quad\, In this last experimental subsection, we also investigate the performance of the \FDSM{} for various delay bounds and various type of transforms. We focus on the Pagoda image whose original version is shown in  Figure  \ref{FI-original} (c). We again start by examining the best choices of parameters $a$ and $a_0$ with the PSNR values and the numbers of iterations in Table \ref{TI3}.

\begin{table}[H]
\centering
\small
\begin{tabular}{ccccccccccccc}
\toprule
\multirow{2}{*}{Transform} & \multicolumn{4}{c}{{$\tau=0$ \cite[Alg. 3.1]{I-16}}}      & \multicolumn{4}{c}{$\tau=1$}      & \multicolumn{4}{c}{$\tau=3$}      \\
\cmidrule(rl){2-5} \cmidrule(rl){6-9} \cmidrule(rl){10-13}
                    & $a$   & $a_0$  & PSNR    & $\#$(iter) & $a$   & $a_0$  & PSNR    & $\#$(iter) & $a$   & $a_0$  & PSNR    & $\#$(iter) \\
\midrule
$R$                   & 0.5 & 0.2 & 21.2993 & 331  & 0.7 & 0.8 & 21.0779 & 364  & 0.9 & 0.9 & 18.8203 & 373  \\
$C$                   & 0.9 & 0.3 & 23.5388 & 326  & 0.8 & 0.9 & 23.3444 & 358  & 0.9 & 0.9 & 19.7886 & 370  \\
$H$                   & 0.8 & 0.2 & 20.9384 & 252  & 0.9 & 0.5 & 20.8965 & 307  & 0.9 & 0.9 & 19.6170 & 343  \\
$L$                  & 0.9 & 0.1 & 23.9296 & 254  & 0.3 & 0.3 & 23.7177 & 293  & 0.8 & 0.7 & 23.7608 & 313  \\
$G$                 & 0.9 & 0.1 & 23.5195 & 204  & 0.2 & 0.5 & 23.4738 & 275  & 0.9 & 0.6 & 23.1919 & 321 
\end{tabular}

\begin{tabular}{ccccccccccccc}
\toprule
\multirow{2}{*}{Transform} & \multicolumn{4}{c}{$\tau=5$}      & \multicolumn{4}{c}{$\tau=10$}      & \multicolumn{4}{c}{$\tau=20$}     \\
\cmidrule(rl){2-5} \cmidrule(rl){6-9} \cmidrule(rl){10-13}
                    & $a$   & $a_0$  & PSNR    & $\#$(iter) & $a$   & $a_0$  & PSNR    & $\#$(iter) & $a$   & $a_0$  & PSNR   & $\#$(iter) \\
\midrule                   
$R$                   & 0.9 & 0.9 & 14.4665 & 380  & 0.9 & 0.9 & 10.3832 & 383  & 0.9 & 0.9 & 9.0003 & 385  \\
$C$                   & 0.9 & 0.9 & 14.6539 & 377  & 0.9 & 0.9 & 10.4046 & 381  & 0.9 & 0.9 & 9.0031 & 383  \\
$H$                   & 0.9 & 0.9 & 13.7716 & 355  & 0.9 & 0.9 & 9.9290  & 368  & 0.9 & 0.9 & 8.8909 & 376  \\
$L$                  & 0.9 & 0.9 & 22.4006 & 321  & 0.9 & 0.9 & 12.4695 & 329  & 0.9 & 0.9 & 9.4391 & 332  \\
$G$                 & 0.9 & 0.9 & 22.0240 & 340  & 0.9 & 0.9 & 12.8467 & 361  & 0.9 & 0.9 & 9.5379 & 372 \\
\bottomrule
\end{tabular}
 \caption{\label{TI3}The best choices of parameters $a$ and $a_0$ with the PSNR values and the numbers of iterations for the Pagoda image.}
\end{table}

Again, it is also observed from Table \ref{TI3} that the behavior of $a$ and $a_0$ are in the similar fashions with two above results. We also notice that, for each $\tau=0, 1, 3$ and $5$, the transform $L$ also yields the best PSNR values, where the overall best PSNR value of $23.9296$ is observed for $\tau=0$. For the case $\tau=0$  and $1$, we observe that the transforms $C$ and $G$ yieled very similar results.
\vskip 2mm

Next, we present the behaviors of the PSNR values within 5 seconds obtained from each dictionary
transform and the delays’ bound in Figures  \ref{FI3}.


\begin{figure}[H]
	\begin{center}
        \begin{minipage}[c]{0.3\linewidth}
		\centering
		{\includegraphics[scale = 0.3]{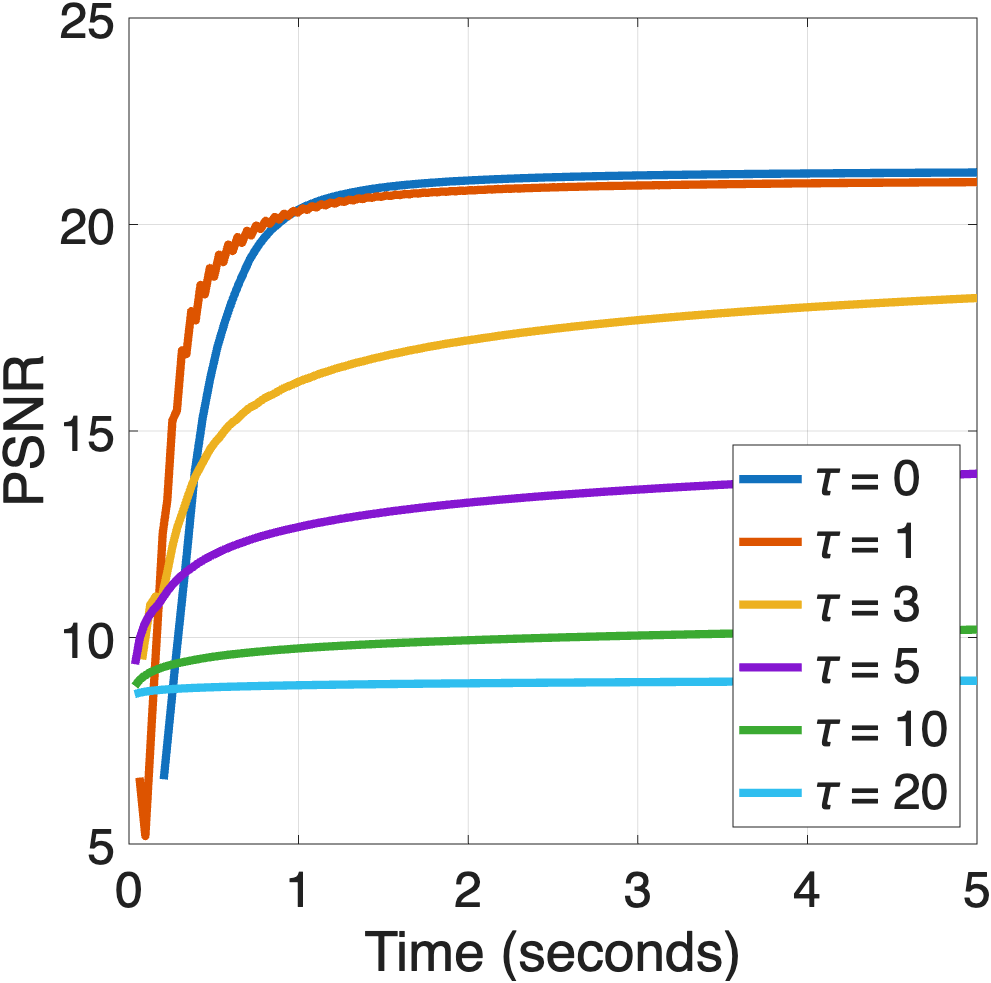}}	\\
	     	(a)  Transform $R$
        \end{minipage}
        \begin{minipage}[c]{0.3\linewidth}
		\centering
		{\includegraphics[scale = 0.3]{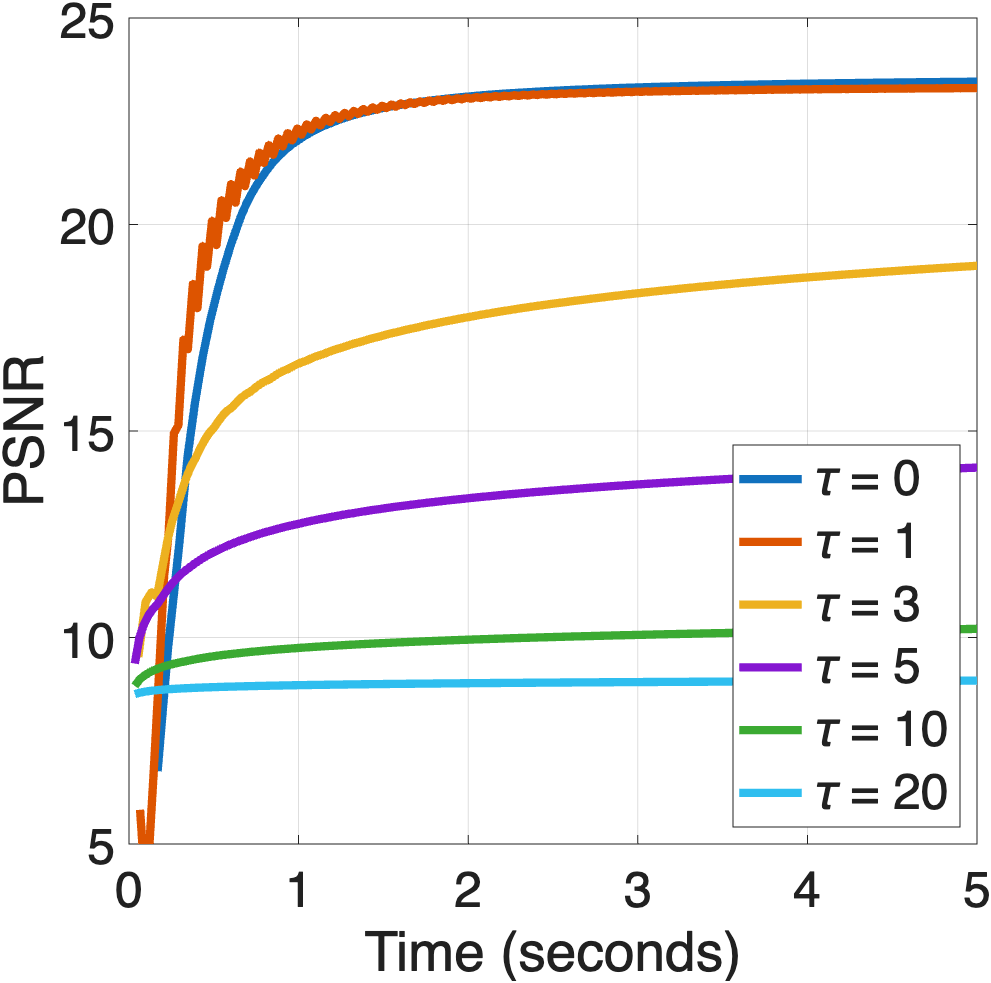}}	\\
	     	(b)  Transform $C$     
        \end{minipage}
        \begin{minipage}[c]{0.3\linewidth}
		\centering
		{\includegraphics[scale = 0.3]{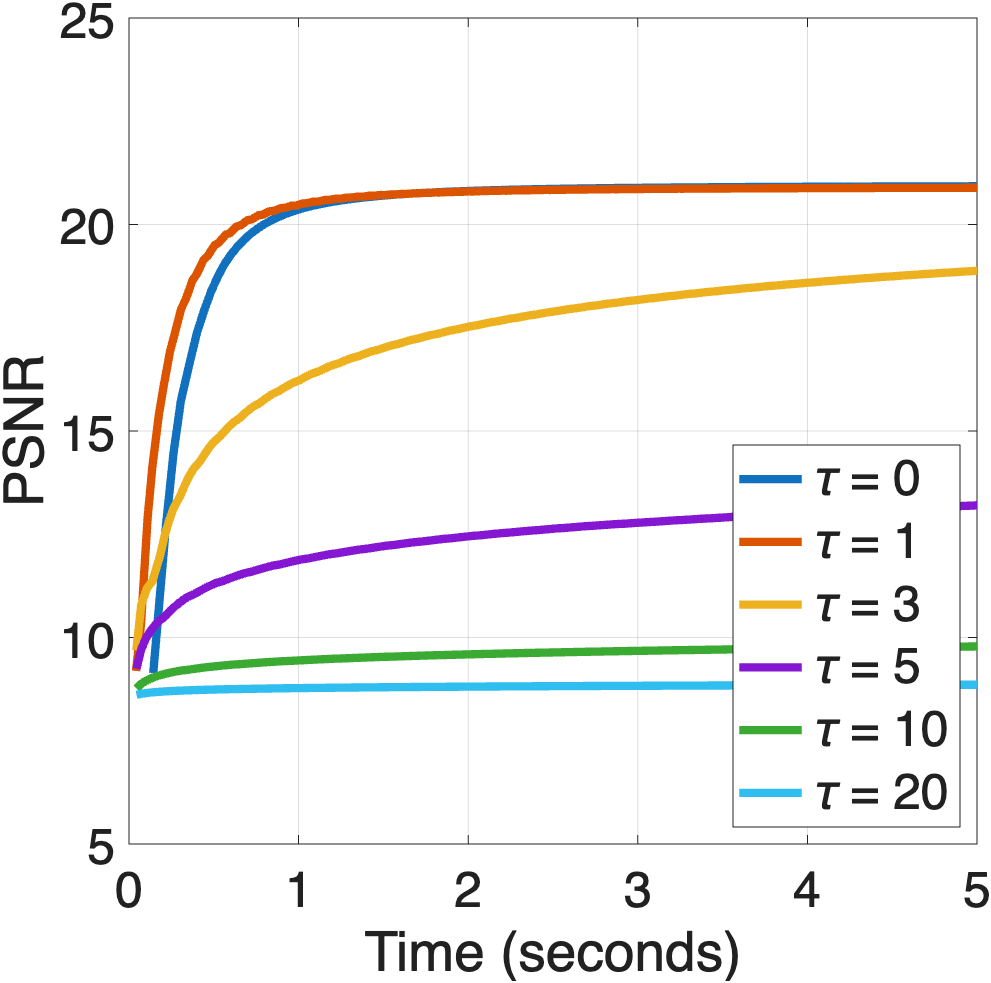}}	\\
	     	(c)  Transform $H$ 
        \end{minipage}
       
        \vspace{0.3cm}
        
        \begin{minipage}[c]{0.3\linewidth}
		\centering
		{\includegraphics[scale = 0.3]{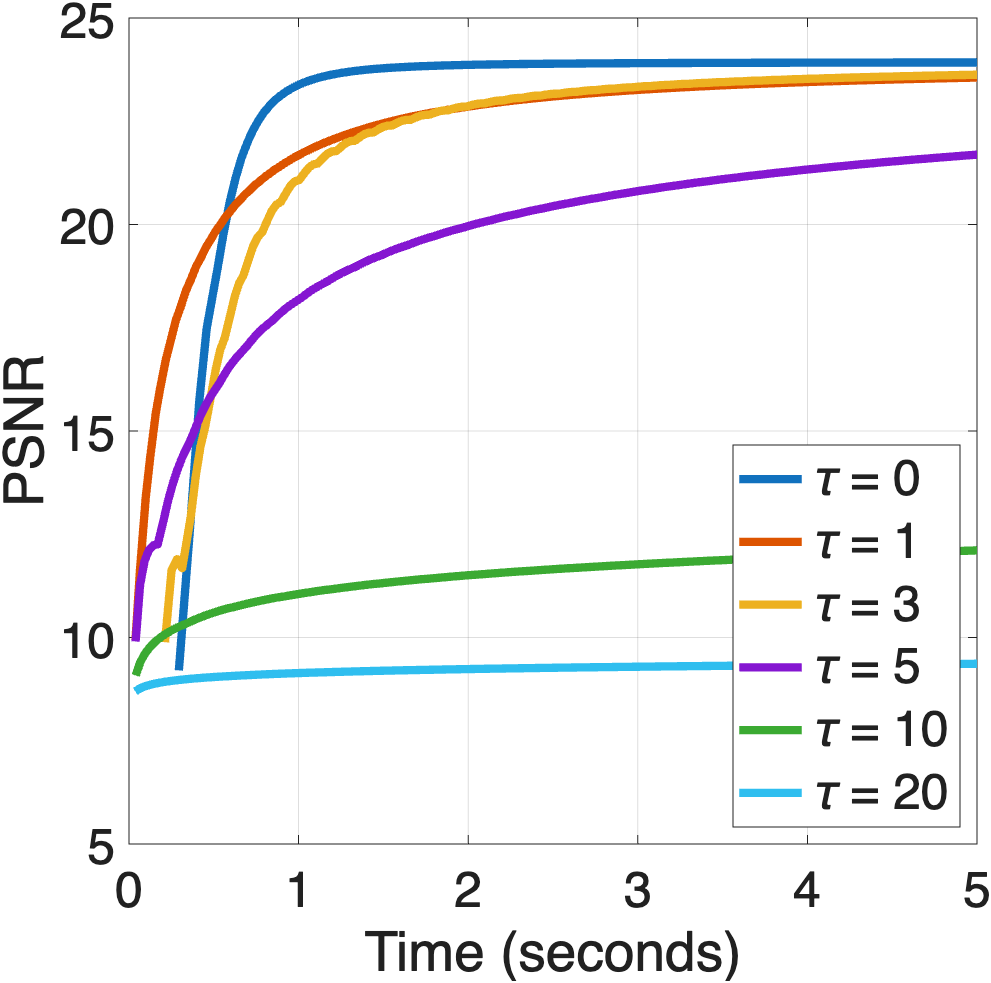}}	\\
	     	(d)  Transform $L$ 
        \end{minipage}
        \begin{minipage}[c]{0.3\linewidth}
		\centering
		{\includegraphics[scale = 0.3]{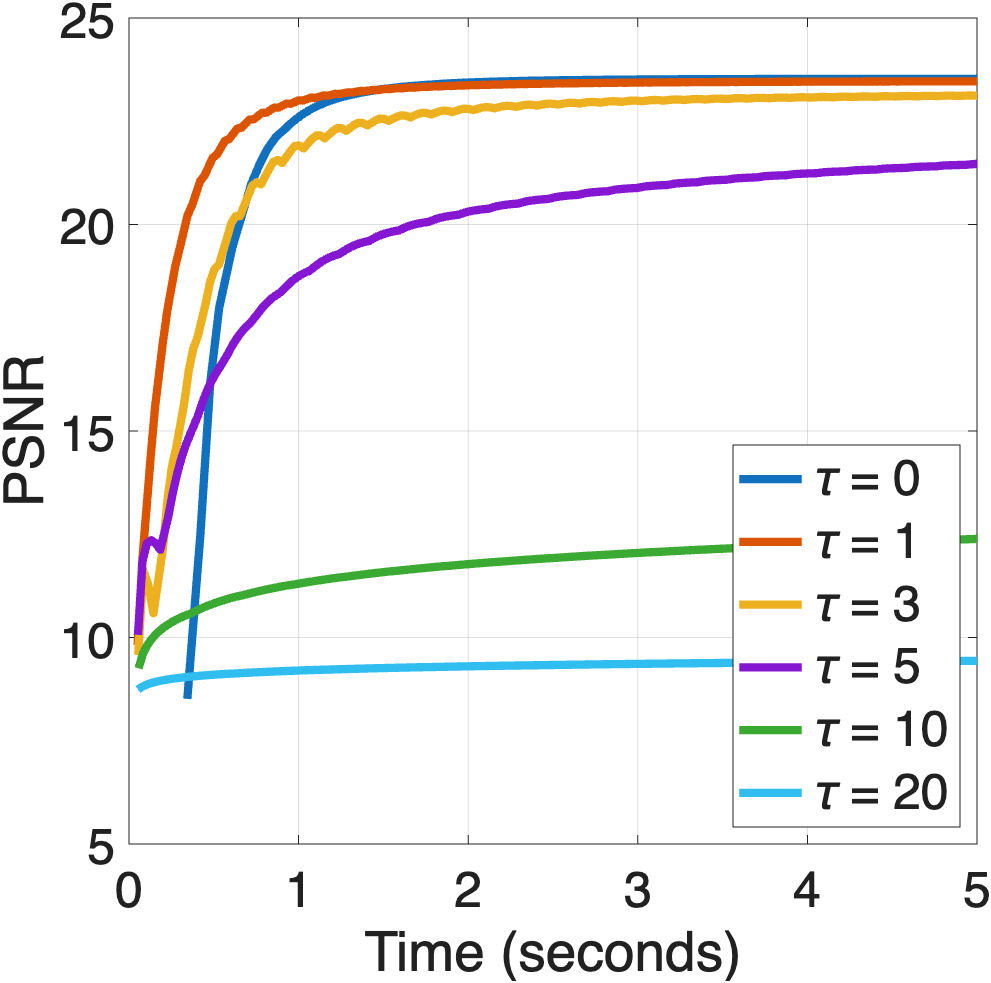}}	\\
	     	(e) Transform $G$ 	     
        \end{minipage}
            \caption{\label{FI3}Behaviors of the PSNR values for the Pagoda image.}
	\end{center}
\end{figure}

It is observed from Figure \ref{FI3} that, for all transforms except the transform $L$, the PSNR of the \FDSM{} with $\tau=1$ increases more rapidly than the other cases. For the case of transform $L$,  the \FDSM{} with $\tau=0$ seems increasingly after a half of a second and, subsequently, is stable within a second. It is worth noting that the transform $C$ can give the very close results to the combining transforms $L$ and $G$, even if it is constructed in a very simple manner. 

From all of these numerical experiments, we observe that, for some tested images and some dictionary transforms, the \FDSM{} with time-varying delays can yield the better reconstructed results than the non-delayed counterpart with the less subgradients' computations. This indicates the advantage of the fixed-point delayed subgradient method over the traditional fixed-point subgradient method ($\tau=0$) in the comparable PSNR with fewer subgradients evaluations.
Finally, the visual restoration of the Pagoda  image, produced by the FDSM with different type of transforms and $\tau=1$ within $10$ seconds, is presented in Figure \ref{FFI3}.

\begin{figure}[H]
	\begin{center}
        \begin{minipage}[c]{0.3\linewidth}
		\centering
		{\includegraphics[scale = 0.45]{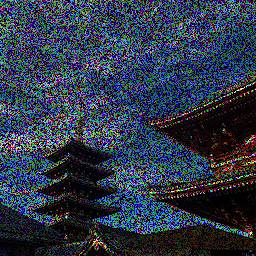}}	\\
	     	(a) Damaged image 	     
        \end{minipage}
        \begin{minipage}[c]{0.3\linewidth}
		\centering
		{\includegraphics[scale = 0.45]{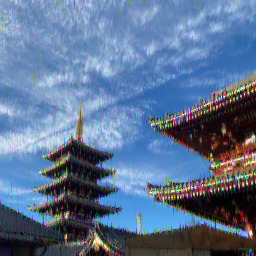}}	\\
	     	(b)  Transform $R$
        \end{minipage}
        \begin{minipage}[c]{0.3\linewidth}
		\centering
		{\includegraphics[scale = 0.45]{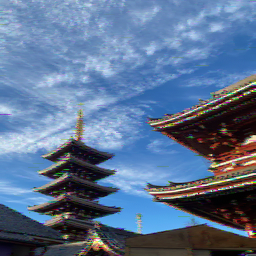}}	\\       
	     	(c)  Transform $C$ 
        \end{minipage}
        \begin{minipage}[c]{0.3\linewidth}
		\centering
		{\includegraphics[scale = 0.45]{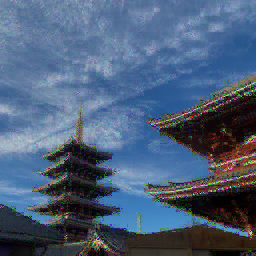}}	\\

        \vspace{0.3cm}
        
	     	(d)  Transform $H$ 
        \end{minipage}        
        \begin{minipage}[c]{0.3\linewidth}
		\centering
		{\includegraphics[scale = 0.45]{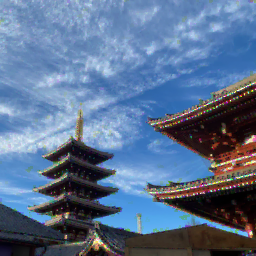}}	\\
	     	(e)  Transform $L$ 
        \end{minipage}
        \begin{minipage}[c]{0.3\linewidth}
		\centering
		{\includegraphics[scale = 0.45]{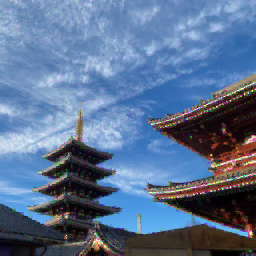}}	\\
	     	(f) Transform $G$ 	     
        \end{minipage}
            \caption{\label{FFI3} 
            The Pagoda  image restored by the FDSM with $\tau=1$ within $10$ seconds.}
	\end{center}
\end{figure}

{
\subsection{Comparison with existing methods}\label{further Experiments}
\vskip 2mm

\quad\, In this subsection, we compare the performance of the proposed method (Algorithm~\ref{alg:FDSM}) with some existing methods, namely the incremental proximal gradient method with a smooth penalty term (IPGM)~\cite[Algorithm 1]{PN-21}, the proximal gradient method (PGM)~\cite[Theorem 26.14]{BC-11} and the fast iterative shrinkage-thresholding algorithm (FISTA)~\cite{BT-09}.

Note that the IPGM is directly applicable to Problem~(\ref{inpainting}), whereas the PGM and the FISTA are not, since they are designed to solve unconstrained optimization problems. In order to apply the PGM and the FISTA to the image inpainting problem, we follow the traditional strategy of solving~(\ref{inverse}) via the following formulation:
\begin{align}\label{pen-pb}
	\begin{array}{ll}
		\textrm{minimize }\indent \frac{1}{2}\Vert B\xv-\bv\Vert^2+r\Vert H\xv \Vert_1\\
		\textrm{subject to}\indent \xv\in \R^d,
	\end{array}%
\end{align}
where $r>0$ is the penalization parameter and the transform $H$ is given as above.

To implement the IPGM, we set $\alpha_n =\frac{a}{n}$ and $\beta_n = bn$ for all $n\geq 1$, where $a,b \in (0,2)$.

For the PGM, we set $r=\frac{1}{255}$, $\gamma \in (0,2)$ and $\lambda_n = \lambda \in (0, \delta)$ for all $n\geq 1$, where $\delta = 2-\frac{\gamma}{2}$.

For the FISTA, we set $r\in (0,2)$.
The initial point is set to the zero vector, and all methods are terminated after a running time of $10$ seconds.

To achieve the best performance of each method, we carefully tune the parameters and report the best parameter choices in Table~\ref{T add methods}.

\begin{table}[H]
\centering
\small
\begin{tabular}{ccccccccccccc}
\toprule
\multirow{2}{*}{Image} & \multicolumn{4}{c}{IPGM}      & \multicolumn{4}{c}{PGM}      & \multicolumn{3}{c}{FISTA}      \\
\cmidrule(rl){2-5} \cmidrule(rl){6-9} \cmidrule(rl){10-12}
                    & $a$   & $b$  & PSNR    & $\#$(iter) & $\gamma$   & $\lambda$  & PSNR    & $\#$(iter) & $r$   & PSNR    & $\#$(iter) \\
\midrule
The Cafe                   & 0.5 & 0.1 & 29.2228 & 1405  & 0.4 & 1.6 & 29.0813 & 1328  & 0.009 & 28.9569 & 1316   \\
The Blue Lagoon & 0.4 & 0.1 & 19.8267 & 1368  & 0.4 & 1.1 & 19.7472  & 1342  & 0.02 & 19.3385 & 1329   \\
The Pagoda                   & 0.4 & 0.1 & 20.9476 & 1382  & 0.5 & 1 & 20.8191 & 1331  & 0.03 & 20.6196 & 1322  \\
\bottomrule
\end{tabular}
\caption{\label{T add methods}The choices of parameters of the IPGM, the PGM, and the FISTA for the three tested images.}
\end{table}

Figure \ref{G add method} presents the evolution of the PSNR values versus computational time within $10$ seconds obtained from the above methods and \FDSM{} with $\tau=0, 1, 3, 5, 10, 20$.

\begin{figure}[H]
	\begin{center}
        \begin{minipage}[t]{0.3\linewidth}
		\centering
		{\includegraphics[scale = 0.3]{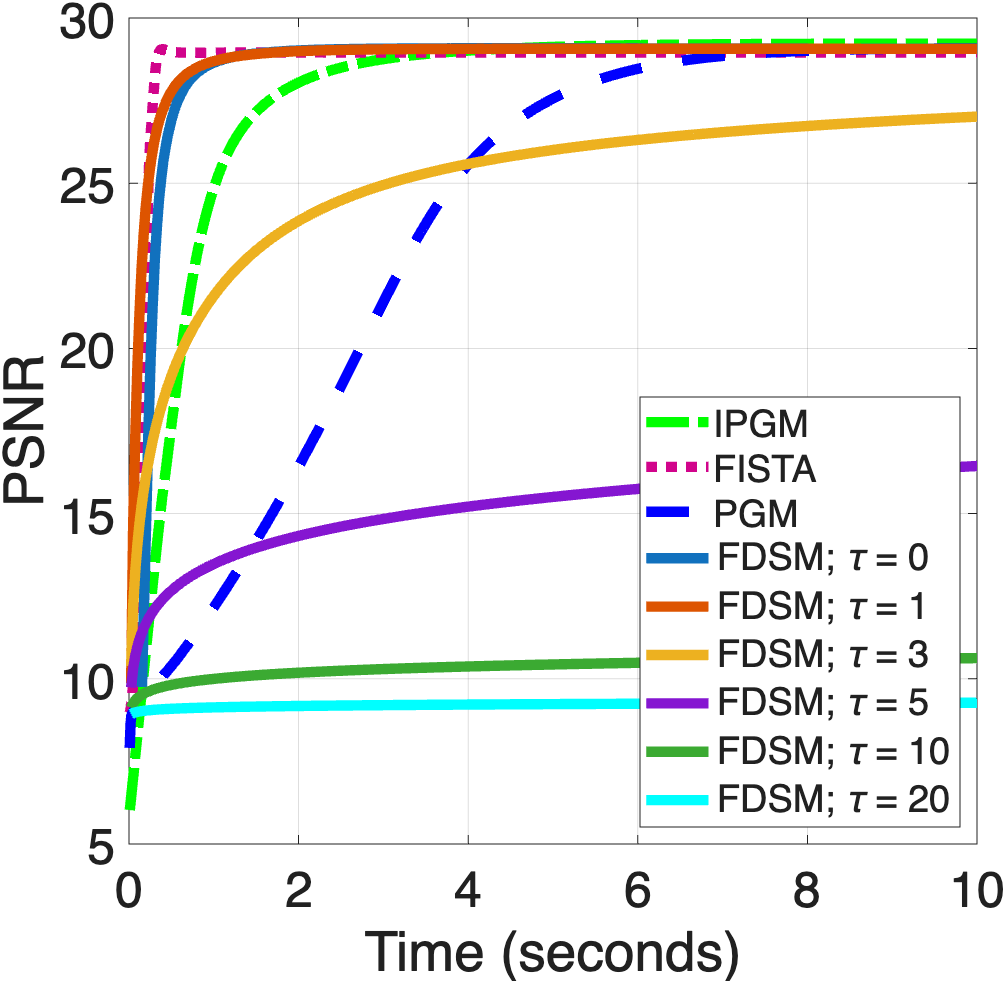}}	\\
	     	(a)   the Cafe image
        \end{minipage}
        \begin{minipage}[t]{0.3\linewidth}
		\centering
		{\includegraphics[scale = 0.3]{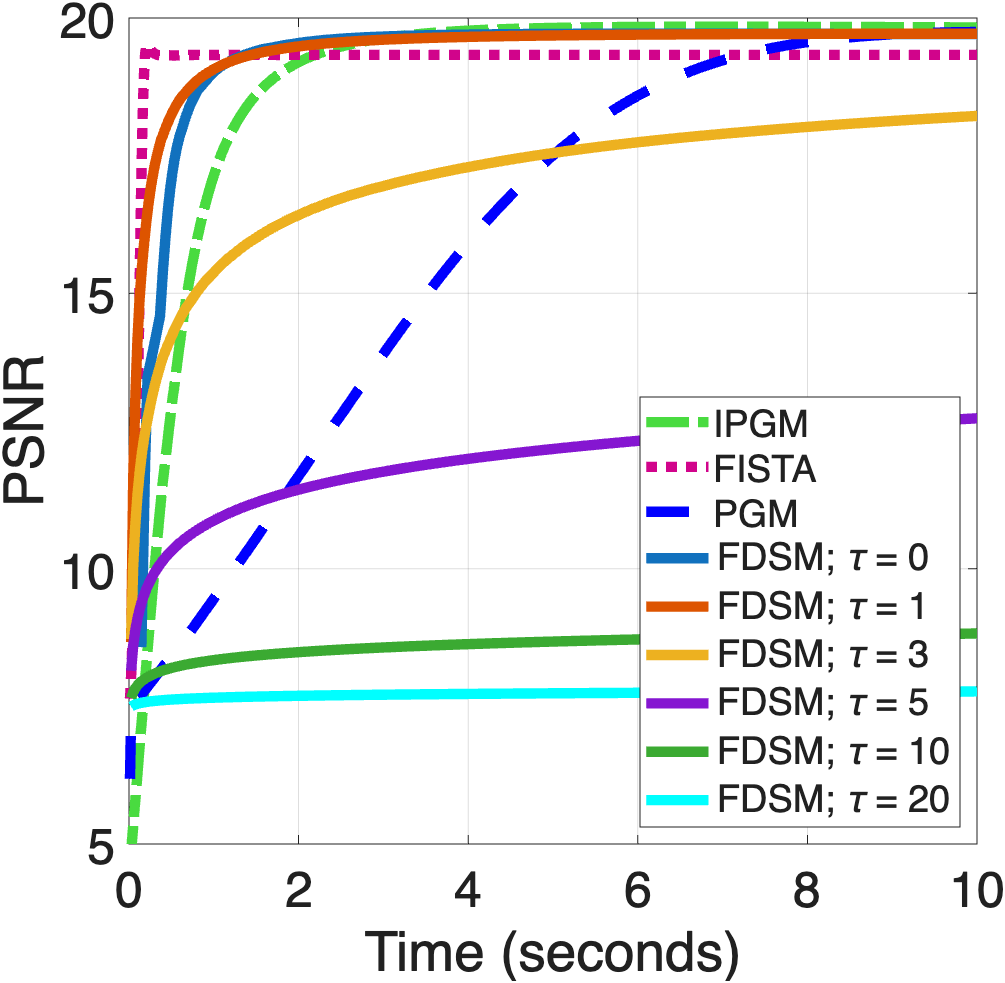}}	\\
	     	(b)   the Blue Lagoon image     
        \end{minipage}
        \begin{minipage}[t]{0.3\linewidth}
		\centering
		{\includegraphics[scale = 0.3]{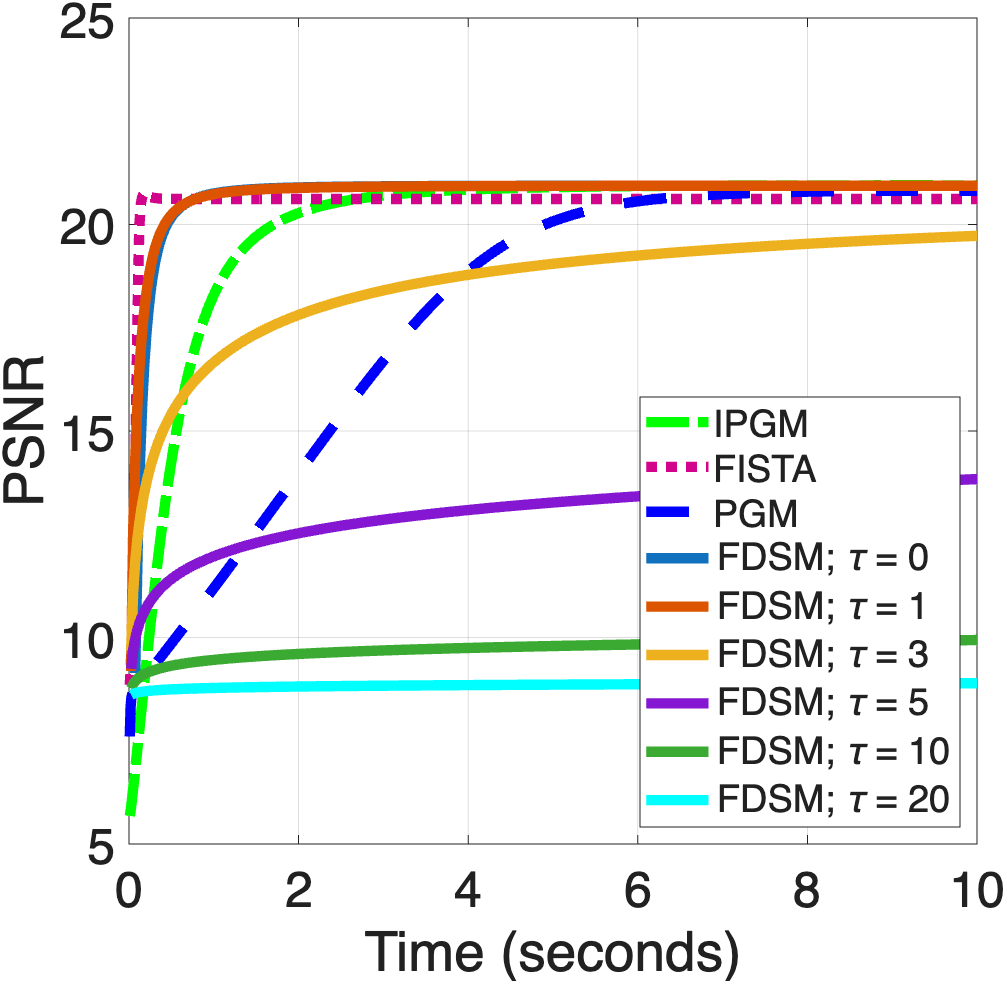}}	\\
	     	(c)   the Pagoda image 
        \end{minipage}      
            \caption{\label{G add method}Behavior of the PSNR values for the three tested images.}
	\end{center}
\end{figure}

According to Figure~\ref{G add method}, the IPGM, the FISTA, the PGM and the FDSM with $\tau \in \{0,1\}$ give close PSNR values for all three tested images within $10$ seconds.

Notice that the FISTA attains a high PSNR value fastest; however, it seems to stabilize at a slightly lower level than the FDSM with $\tau  =0,1$ and the IPGM. This may be because it solves the problem~(\ref{pen-pb}) rather than the problem~(\ref{inpainting}). the FDSM with $\tau =0,1$ reach its best PSNR value within about one second, comparable to the FISTA, while the IPGM and the PGM need about three and eight seconds, respectively, to reach a similar level.

}
\section{Conclusions}\label{conclusion}

\quad\, In this paper, first, we introduced the FDSM for solving the problem~(\ref{main-pb}). The method is applicable to convex objective functions that are either smooth or nonsmooth. The nonlinear operator in our method is assumed to be firmly nonexpansive, since the convergence analysis relies on the properties of firmly nonexpansive, nonexpansive and cutter operators.
Although the parallel subgradient method~\cite[Algorithm~3.1]{I-16} can also solve the problem~(\ref{main-pb}), it requires the subgradient at the current iterate to compute the next iterate. In contrast, our method allows the use of previously computed subgradients in the updates.

Second, under a suitable step-size control and certain assumptions, we established an approximation of the optimal value and the convergence of the sequence generated by the proposed method.

Third, we applied the FDSM to the image inpainting problem and compared its efficiency under different delays.

\section*{Acknowledgements}

The authors are thankful to the Editor and anonymous referees for comments and remarks which improved the quality
and presentation of the paper. O. Pankoon was supported by the Development and Promotion of Science and Technology Talents Project (DPST).

\section*{Funding information} The research (Project No. 218510) was supported by the Fundamental Fund (FF) of Khon Kaen University for the fiscal year 2026 from National Science, Research and Innovation Fund or NSRF, Thailand

\section*{Conflict of interest}
The authors have no competing interests.

\section*{AI usage statement}
After the first round of review, the authors used Claude Opus 5.5 (Anthropic) to proofread the English of the revised manuscript. During this process, the tool identified a gap in the original proof of Theorem \ref{thm:con}: in Case 2, the estimate now stated as (\ref{tt7}) did not follow from the previous argument. The tool suggested introducing an auxiliary sequence that combines $\|\xv_n-\xv\|^2$ with a weighted sum of the past $\tau$ increments of the iterates. Based on this suggestion, the authors constructed the sequence $\{V_n(\xv)\}_{n=0}^\infty$, derived Lemma \ref{LL}, and revised the proofs of Theorem \ref{thm:con}, Corollary \ref{cor:main}, and the corresponding results in Sections \ref{inexact} and \ref{distributed} (Lemmas \ref{LL-en} and \ref{LL-Dis} and Theorems \ref{thm:con epsilon} and \ref{thm:con:Dis}). 
The AI tool was not used to generate or modify the numerical results. All mathematical arguments were independently  verified by the authors, who take full responsibility for the content of this article.

\section*{Data availability statement}
Data sharing is not applicable to this work as no datasets were generated or analyzed during this study.


\end{document}